\documentclass[preprint]{siamart171218}

\usepackage[utf8]{inputenc}
\usepackage{amssymb}
\usepackage{amsmath}
\usepackage{algorithm}
\usepackage{algorithmic}
\usepackage{bm}

\usepackage{afterpage}

\usepackage[section]{placeins}

\makeatletter
\AtBeginDocument{%
  \expandafter\renewcommand\expandafter\subsection\expandafter{%
    \expandafter\@fb@secFB\subsection
  }%
}
\makeatother

\usepackage{multirow}
\usepackage{color}

\newtheorem{remark}{Remark}

\usepackage{geometry}

\usepackage{float}

\usepackage{subcaption}

\usepackage{mathtools}

\renewcommand{\chi}{\mathcal{X}}

\newcommand{\He}{\mathcal{H}}
\newcommand{\Ft}{\mathcal{F}}

\newcommand{\gG}{{\bm G}}

\newcommand{\gM}{{\bm M}}

\newcommand{\gV}{{\bm U}}

\newcommand{\fC}{\mathcal{C}}
\newcommand{\fV}{\mathcal{V}}

\newcommand{\bigO}{\mathcal{O}}
\newcommand{\incr}{ \Delta}

\newcommand{\vu}{{\bm u}}

\newcommand{\vx}{{\bm x}}
\newcommand{\vy}{{\bm y}}

\newcommand{\vg}{{\bm \gamma}}

\newcommand{\vX}{{\bm X}}
\newcommand{\vhX}{{\bm \chi}}

\newcommand{\hstro}{ \psi^n }

\newcommand{\R}{{\mathbb{R}}}

\newcommand{\curl}{\nabla \times}
\renewcommand{\div}{\nabla \cdot}
\newcommand{\grad}{\nabla}
\newcommand{\Laplace}{ \mathop{{}  \bigtriangleup }\nolimits }
\newcommand{\PSolve}{ \mathop{{}  \bigtriangleup^{-1} }\nolimits }

\definecolor{myPurp}{rgb}{0.5, 0.0, 1.0}

\title{A Characteristic Mapping Method for the two-dimensional incompressible Euler equations}

\author{Xi-Yuan Yin\thanks{Department of Mathematics and Statistics, McGill University, Montr\'{e}al, Qu\'{e}bec H3A 0B9, Canada}
\and Olivier Mercier\footnotemark[1]
\and Badal Yadav\footnotemark[1]
\and Kai Schneider\thanks{I2M-CNRS, Centre de Math\'{e}matiques et d’Informatique, Aix-Marseille Universit\'{e}, 13453 Marseille Cedex 13, France}
\and Jean-Christophe Nave\footnotemark[1]  \thanks{Corresponding author. E-mail address: \email{jcnave@math.mcgill.ca}.}}

\begin{document}

\maketitle




\begin{abstract}
We propose an efficient semi-Lagrangian method for solving the two-dimensional incompressible Euler equations with high precision on a coarse grid. The new approach evolves the flow map using the gradient-augmented level set method (GALSM).  Since the flow map can be decomposed into submaps (each over a finite time interval), the error can be controlled by choosing the remapping times appropriately.  This leads to a numerical scheme that has exponential resolution in linear time. 
Error estimates are provided and conservation properties are analyzed.
The computational efficiency and the high precision of the method are illustrated for a vortex merger and a four mode and a random flow. Comparisons with a Cauchy-Lagrangian method are also presented.

\end{abstract}

%
%
%
%





\section{Introduction} \label{Intro}

The numerical simulation of incompressible fluids is an important mathematical problem with many scientific and industrial applications. The study of the incompressible Euler equations in two dimensional (2D) space poses many challenges. It is well known that solutions of the 2D Euler equations, although smooth, have fast growing gradients. The growth of the sup-norm of the vorticity gradient can be bounded by a double exponential in time \cite{yudovich1963flow}, this bound has been proven to be sharp in the case of smooth initial data on a disk \cite{kiselev2014small}. Furthermore, on a flat periodic 2-torus, it has been shown that for an arbitrary time interval, there exists a vorticity field whose gradient exhibit exponential growth within the chosen interval \cite{denisov2009infinite}. For a review on summarizing advances of mathematics for the Euler equations, we refer to a 2013 paper by Bardos and Titti \cite{bardos2013mathematics}. These results suggest that the spatial discretization for the numerical solution of Euler equations can be challenging. As the gradient of the solution grows, the numerical resolution required to correctly capture the solution also increases. The resources necessary to avoid excessive spatial truncation errors can thus become prohibitive for long time simulations. On the other hand, allowing for spatial truncation errors by undersampling the solution can generate numerical dissipation akin to a viscosity term which qualitatively affects the simulation. In some cases, truncation errors resonate with the solution, creating spurious oscillations and numerical instability. This phenomenon is analysed in detail by Ray et al. \cite{ray2011resonance} in the case of a conservative Fourier-Galerkin scheme, and a regularization technique has been proposed by Pereira et al. in \cite{pereira2013wavelet}.

Existing numerical methods for fluid simulation lie in the spectrum from fully Lagrangian to fully Eulerian formulation. Eulerian methods use fixed spatial meshes to represent evolving quantities. This gives easy access to the relevant simulated quantities. The spatial features generated by large fluid deformations can be represented as long as the grid has sufficient resolution. However, finer scale features are lost resulting in artificial dissipation. Examples of Eulerian methods include the Fluid-in-Cell \cite{gentry1966eulerian} method and the Marker-in-Cell method \cite{harlow1965numerical}, where fluid quantities are associated to fixed spatial subdivisions and updated in time using the evolution equations. In Lagrangian methods, particles or grid nodes follow the movement of the flow, resulting in less dissipative schemes. However, with the particle or moving grid framework, the evaluation of Eulerian quantities is difficult and less accurate. Furthermore, large fluid deformations can be difficult to resolve and often require frequent remeshing or resampling routines. Examples of purely Lagrangian methods include Smooth Particle Hydrodynamics \cite{gingold1977smoothed, monaghan2005smoothed} and Vortex blob method \cite{chorin1973numerical}, where fluid quantities are transported using a particle method. Many numerical solvers use a mix of Lagrangian and Eulerian description to exploit the advantages of each approach. For instance, Arbitrary Lagrangian-Eulerian \cite{hirt1974arbitrary, donea1982arbitrary} methods use reference coordinates that are neither fully Lagrangian nor Eulerian to describe the fluid configuration. Another large family of numerical solvers are the semi-Lagrangian methods. In these methods, relevant quantities are represented on an Eulerian frame, however, the evolution equations are discretized from a Lagrangian description. Methods of this type include the Cauchy-Lagrangian method \cite{frisch} where the vorticity field is evolved by transport along a moving mesh following the flow, the result is then periodically projected back on an Eulerian grid. Level-set methods are another popular semi-Lagrangian framework used for fluid simulation and implicit interface tracking, see \cite{osher2006level}. In these methods, characteristics are traced backwards in time and Eulerian interpolation schemes are used to update the solution. Semi-Lagrangian schemes have also been used to generate fast simulations intended for computer graphics \cite{stam1999stable, foster2001practical}. One main property of the semi-Lagrangian approach is that it tries to captures the characteristic structure of the equations. In particular, in 2D inviscid flow, the vorticity field is transported along characteristic curves, hence accuracy and stability can be improved by taking the geometric approach of following these characteristic curves in order to propagate the solution in time.

In this paper, we present a method where the evolved numerical quantity is the deformation map of the domain generated by the fluid velocity. This approach is geometric and fully exploits the characteristic structure of the fluid flow. We construct numerically a characteristic map which identifies a point on an arbitrary characteristic curve to its initial position through a spatial transformation. Advected quantities, in particular vorticity, can then be constructed as the function composition of the initial condition with the characteristic map. This method stems from the semi-Lagrangian Gradient-Augmented Level-Set (GALS) \cite{nave2010gradient} and Jet-Scheme methods \cite{seibold2011jet} and constitutes an extension of the Characteristic Mapping (CM) method for linear advection \cite{kohno2013new, CM} to the self-advection in 2D incompressible Euler. The CM method for Euler splits the evolution equations into the advection of the vorticity and the computation of the velocity through the Biot-Savart law. These two parts are connected in that the advection provides the vorticity field for the Biot-Savart kernel, whose resulting velocity field is then used to advect the CM map. In doing so we achieve a separation of scales: under the assumption that the flow is governed by large scale features of the velocity field, the characteristic map can be accurately evolved on a coarse grid. Furthermore, since the vorticity solution is constructed through the pullback of the initial condition by the characteristic map, the functional definition of the solution provides arbitrarily fine spatial resolution.

One main issue in inviscid fluid simulations is the artificial viscosity incurred from the spatial truncations during the evolution of the solution. Some methods such as \cite{frisch} approach this problem by designing high order methods in order to take extremely large time steps, hence minimizing the accumulation of the diffusive error. Others employ an adaptive multi-resolution mesh refinement to efficiently resolve fine scale features \cite{deiterding2016comparison, kolomenskiy2016adaptive}. One unique property of the CM method for 2D Euler demonstrated in this paper is that it completely eliminates artificial viscosity by never directly evolving a discretized vorticity field: the vorticity changes as a consequence of the evolving characteristic map used to compute the pullback. In particular, a straightforward consequence is that the extrema of the vorticity field are conserved for all times. Additionally, in order to correctly represent the arbitrarily fine scales generated by the flow, we use a time decomposition of the characteristic map based on its semigroup structure. This allows the CM method to represent exponentially growing vorticity gradients while only computing on a fixed coarse grid. The resulting scheme achieves arbitrary subgrid resolution, high long term enstrophy conservation and is void of artificial dissipation.

The rest of the paper is organized as follows: in section 2 we lay out the mathematical framework for the CM method and carry out some heuristic analysis of its properties. In section 3, we present in detail the numerical implementation of the method and provide some error bounds. Section 4 contains several numerical tests and discussions on the accuracy and qualitative properties of the solutions. Finally, we make some concluding remarks and propose future directions of work in section 5.

\section{Mathematical Framework} \label{sec:mathFrame}

In this section, we present the mathematical framework behind the Characteristic Mapping (CM) method for 2D incompressible Euler. The section is organized as follows: first, we present the equations with the advection/Biot-Savart splitting used in the method. Then we look at the CM method for the advection equation along with its intrinsic semigroup structure. We will next examine the advection-vorticity coupling of the equation through the Biot-Savart law. Finally, we will put everything together to write the modified equation which naturally arises from a numerical CM method for Euler equations.

The 2D incompressible Euler equations are:
\begin{subequations} \label{eqGroup:Euler}
\begin{align} 
\partial_t \vu + (\vu \cdot \grad ) \vu = - \grad p   &  \quad \quad (\vx, t) \in U \times \R_+ ,  \label{eq:Euler} \\
\div \vu = 0 \label{eq:divfree}  , \\
\vu(\vx, 0) = \vu_0(\vx),
\end{align}
\end{subequations}
where $U$ is the fixed spatial domain, $\vu$ is a vector field describing the instantaneous velocity of a fluid element and $p$ is the pressure. For this paper, we assume for simplicity that $U$ is the flat 2-torus and hence there are no boundary conditions. In general, for a boundary $\partial U$, the boundary condition for inviscid flow is $\vu \cdot \bm{n}_{\partial U} = 0$ on $\partial U$ where $\bm{n}_{\partial U}$ is the normal vector to the boundary. This will require further extensions to the framework and is not covered in this paper.

Define $\omega = \curl \vu$ the scalar vorticity of the fluid, and taking a 2D curl of \ref{eqGroup:Euler}, we obtain the vorticity equations:
\begin{subequations} \label{eqGroup:EulerVort}
\begin{align} 
\partial_t \omega + (\vu \cdot \grad ) \omega = 0   &  \quad \quad (\vx, t) \in U \times \R_+ ,  \label{eq:EulerVort} \\
\div \vu = 0 , \label{eq:divfreeVort} \\
\omega(\vx, 0) = \omega_0 (\vx) .
\end{align}
\end{subequations}
Using the incompressibility assumption, $\vu$ can be obtained from $\omega$ by solving a Helmholtz-Hodge problem, $\vu$ is then given by the Biot-Savart law 
\begin{equation} \label{eq1:BiotSavart}
\vu = - \PSolve \curl \omega .
\end{equation}

Given a solution $\vu(\vx, t)$ of \ref{eqGroup:Euler}, we first make the observation that $\omega$ solves the advection equation \eqref{eq:EulerVort} under the velocity field $\vu$. The method of characteristics for advection equations implies that $\omega$ satisfies
\begin{equation} \label{eq:vortCharEq}
\frac{d}{dt} \omega(\vg(t), t) = 0,
\end{equation}
for any characteristic curve $\vg$ solving
\begin{subequations} \label{eqGroup:charCurve}
\begin{align} 
\frac{d}{dt} \vg(t) = \vu(\vg(t), t) , \\
\vg(0) = \vg_0 .
\end{align}
\end{subequations}
The smoothness of these characteristic curves is proven in \cite{constantin2015analyticity}.

The approach presented in this paper consists in splitting the vorticity equations \eqref{eqGroup:EulerVort} into the coupling of the advection of the vorticity by \eqref{eq:vortCharEq} and the Biot-Savart law \eqref{eq1:BiotSavart}. The advection equation is solved using the Characteristic Mapping method \cite{CM} and the Biot-Savart law is applied in Fourier space. We present details on these two methods in the following sections.

\subsection{Characteristic Mapping Framework}
The Characteristic mapping approach consists in finding the solution operator for the advection equation associated to some transport velocity. A more detailed exposition of this approach can be found in \cite{CM, nave2010gradient, kohno2013new}.

Consider a linear advection equation with a divergence-free transport velocity $\vu$
\begin{subequations} \label{eqGroup:Adv}
\begin{align} 
\partial_t \phi + (\vu \cdot \grad) \phi = 0 &  \quad \quad (\vx, t) \in U \times \R_+ , \\
 \phi(\vx, 0) = \phi_0(\vx) . &
\end{align}
\end{subequations}

A solution $\phi$ of the advection equation must satisfy 
\begin{gather}
\frac{d}{dt} \phi(\vg(t), t) = 0 ,
\end{gather}
or equivalently
\begin{gather} \label{eq:pbCharCurve}
\phi(\vg(t), t) = \phi_0(\vg_0)
\end{gather}
for all characteristic curves $\vg$ satisfying \eqref{eqGroup:charCurve}.

Therefore, we look for a solution operator $\vX : U \times \R_+ \to U $ such that $\vX (\vg(t), t) = \vg_0$ for all characteristic curves $\vg(t)$ associated to the velocity $\vu$. For any fixed $t$, the map $\vX (\cdot, t) : U \to U$ is the inverse of the $t$-time flow map of the velocity $\vu$. We call $\vX$ the \emph{backward characteristic map} since it ``traces back'' a particle at time $t$ to its initial position $\vg_0$. Assuming smooth divergence-free velocity, the flow map is a diffeomorphism, hence, we can write $\vX$ as the solution to the following vector valued advection equation:
\begin{subequations} \label{eqGroup:AdvMapEqn}
\begin{align} 
\partial_t \vX + (\vu \cdot \grad) \vX = 0 & \quad \quad  \forall (\vx, t) \in U \times \R_+ ,\\
 \vX(\vx, 0) = \vx .&
\end{align}
\end{subequations}
One can check that $\frac{d}{dt} \vX(\vg(t), t) = 0$ and hence $\vX(\vg(t), t) = \vX(\vg_0, 0) = \vg_0$. 

\begin{remark}
Wolibner proved in 1933 \cite{wolibner1933theoreme} that for an $L^1$ initial vorticity field (with appropriate decay at infinity), the corresponding velocity field (which is log-Lipschitz) solving the Euler equations generates a H\"older continuous forward characteristic map. That is to say, the transformation $\Phi(\cdot, t) : \gamma_0 \mapsto \gamma(t)$ is a H\"older continuous homeomorphism for all $t$. The existence, continuity and area-preservation of the backward map $\vX$ follows since $\vX(\cdot, t) = \Phi^{-1}(\cdot, t)$.
\end{remark}

The solutions to the advection equation \ref{eqGroup:Adv} can be then computed as the pullback of the initial condition by this backward map using \ref{eqGroup:Adv}
\begin{gather}
\phi(\vx, t) = \phi_0 \left( \vX( \vx, t)  \right) .
\end{gather}

In fact, for any initial condition $\phi_0$, the corresponding advection equation has solution $ \phi_0 \circ \vX$. The characteristic map is independent of the advected quantity, it captures the deformation of the space induced by the velocity $\vu$ and acts as a solution operator of all quantities transported by $\vu$.

\subsubsection{Semigroup structure} \label{sec:SemiGroup}
The treatment of incompressible inviscid flow from the point of view of differential geometry and geodesic flow on the group of volume preserving diffeomorphisms was initiated by Arnold in 1966 \cite{arnold1966geometrie}. The backward characteristic maps are the inverse maps of the elements of the one-parameter semigroup of these volume preserving forward flow maps parametrized by $t$. Hence, they inherit the same semigroup properties.

Consider the solution operator $\vX^\dag$ of the advection equation taking some arbitrary time $t^\dag$ as initial time:
\begin{subequations} \label{eqGroup:SubmapEqn}
\begin{align} 
\partial_t \vX^\dag +  (\vu \cdot \grad) \vX^\dag = 0  , \\
 \vX^\dag(\vx, t^\dag) = \vx .
\end{align}
\end{subequations}
We use the notation
\begin{gather}
\vX_{[t, t^\dag]} (\vx) = \vX^\dag (\vx, t) ,
\end{gather}
for $t > t^\dag$. Here $\vX_{[t, t^\dag]}$ is the backward characteristic map generated by $\vu$ in the time interval $[t^\dag, t]$ and traces back along characteristics a particle at position $\vx$ at time $t$ to its position $\vX_{[t, t^\dag]} (\vx)$ at time $t^\dag$. One can check that given arbitrary times $t_0 < t_1< t_2$, we have
\begin{gather} \label{eq:semigroup}
\vX_{[t_2, t_0]} = \vX_{[t_1, t_0]} \circ \vX_{[t_2, t_1]}.
\end{gather}
This is in fact true without the $t_0 < t_1< t_2$ assumption and would involve forward characteristic maps. This paper will only focus on the use of the backward maps.

This semigroup structure is at the heart of the CM method. The time evolution of the characteristic map is defined though this map composition rather than through the usual PDE. We can see the evolution of $\vX$ as given by
\begin{gather} \label{eq:semigroupInfinitess}
\vX(\cdot, t+ \incr{t}) = \vX (\cdot, t) \circ \vX_{[t+ \incr{t}, t]} ,
\end{gather}
The numerical approximation of $\vX$ relies on discretizing the above using some small time step. 

We also make use of this semigroup property to adaptively adjust the numerical resolution of the map. In particular, a backward characteristic map $\vX( \cdot, t)$ for the time interval $[0, t]$ can be split into arbitrarily many submaps in the following way:
 
We subdivide the interval $[0, t]$ into $m$ subintervals $[\tau_{i-1}, \tau_i]$ with $0 = \tau_0 < \tau_1 < \cdots < \tau_m = t$. One can then check that the following decomposition holds:
\begin{gather} \label{eq:SubmapDecomp}
\vX (\cdot, t) = \vX_{[t, 0]} =  \vX_{[ \tau_1, 0]} \circ  \vX_{[\tau_2, \tau_1]} \circ \cdots \circ  \vX_{[\tau_{m-1}, \tau_{m-2}]} \circ  \vX_{[t, \tau_{m-1}]}
\end{gather}
Each of the submaps $\vX_{[\tau_i, \tau_{i-1}]}$ can be computed individually and stored. The global time map is then defined as the composition of all the stored submaps. This decomposition will provide several important numerical advantages which we will discuss in section \ref{sec:NumMethod}. In particular, each submap has the identity map as initial condition and the subdivision allows for dynamic and adaptive spatial resolution without changing the computational grid.

In 2D incompressible Euler, the vorticity gradient can grow exponentially in time. Using the characteristic map, we can write
\begin{gather}
\omega(\vx, t) = \omega_0 (\vX(\vx, t)).
\end{gather}
We observe from this that the advection operator is responsible for the formation of high vorticity gradient since
\begin{gather}
\grad \omega(\vx, t) = \grad \omega_0 \cdot \grad \vX (\vx, t).
\end{gather}
In cases where the vorticity gradient grows exponentially, we infer that $\grad \vX$ must also grow exponentially. The semigroup decomposition is analogous to the exponential function in one variable, where the natural instantaneous evolution is multiplicative instead of additive in the sense that $\exp (c(t+\incr{t})) \approx  (1 + c\incr{t}) \exp (ct) $ is more natural than $\exp (c(t+\incr{t})) \approx \exp (ct) + c \incr{t} \exp (ct)$ (the latter requiring tracking an integrand which grows exponentially). Similarly, taking the gradient of equation \eqref{eq:SubmapDecomp}, we have
\begin{gather} \label{eq:SubmapDecompGrad}
\grad \vX_{[t, 0]} =  \grad \vX_{[ \tau_1, 0]} \grad \vX_{[\tau_2, \tau_1]}  \cdots   \grad \vX_{[\tau_{m-1}, \tau_{m-2}]} \grad \vX_{[t, \tau_{m-1}]} .
\end{gather}
This means that an exponential growth in gradient can be achieved by the composition of submaps each having bounded gradient.

\subsection{Advection-Vorticity Coupling}
We use the backward characteristic map to rewrite the vorticity equation \ref{eq:EulerVort} as the following coupling of $\vu, \omega$ and $\vX$:
\begin{subequations} \label{eqGroup:CME_Eqn}
\begin{align} 
\omega(\vx, t) = \omega_0 (\vX(\vx, t)) , \label{eq:VortPB} \\
\vu = - \PSolve \curl \omega  , \label{eq:BiotSavart} \\
(\partial_t + \vu \cdot \grad) \vX = 0 . \label{eq:AdvCM}
\end{align}
\end{subequations}

Equation \ref{eq:BiotSavart} is known as the Biot-Savart law in and can be obtained from $\div \vu = 0$ and the definition of the vector Laplacian:
\begin{gather}
\Laplace {\bm F} = \grad \left( \div {\bm F}  \right) - \curl \left(   \curl {\bm F} \right)
\end{gather}
for ${\bm F}$ a  $\R^2 \to \R^2$ vector field. In 2D, we let $\vu = {\bm F}$ and  commute $\PSolve$ and $\curl$ in \ref{eq:BiotSavart}. The velocity is then obtained from the stream function $\psi$:
\begin{gather} \label{eq:stream}
\psi = - \PSolve \omega, \quad \quad \vu = \curl \psi .
\end{gather}

The CM method for Euler equations then consists of numerically evolving $\vX$ in time using equation \eqref{eq:AdvCM}; the velocity and vorticity are defined using \eqref{eq:VortPB}, and \eqref{eq:stream}.

\subsection{Modified Equation} \label{sec:modEqn}

The coupling of the advection of the vorticity and the Biot-Savart law can be thought of as a feedback loop between the characteristic map and the discretized velocity. Let $\vhX^n$ be the numerical characteristic map at a discrete time step $t_n$. This generates a velocity $\vu^n$ which we use to evolve the map to the next time step. In this section, we look at the modified equation which arises from replacing the true velocity $\vu$ by some modified $\tilde{\vu}$ which is better approximated by the discrete $\vu^n$. We give some general error estimates between the true solution and the solution of the modified equation based on the discrepancies between $\vu$ and $\tilde{\vu}$.

For some given numerical solution $\vhX^n$, $n = 0, 1, 2, \ldots$, we look at a modified velocity $\tilde{\vu}$ defined at all times which approximates the velocities $\vu^n$ at discrete time steps $t_n$. The corresponding modified equation for the characteristic map is then
\begin{gather} \label{eq:modEqnChar}
\partial_t \tilde{\vX} + (\tilde{\vu} \cdot \grad ) \tilde{\vX} = 0.
\end{gather}
From here on, we will use the notation \emph{tilda} for variables associated to the modified equation. Discretized variables will be denoted by script letters with a superscript $n$ referring to the corresponding time step $t_n$.

We will estimate the difference between the true solution $\vX$ and the solution to the modified equation $\tilde{\vX}$. This is a useful strategy as it will allow us to bound the error of the numerical solution using $| \vX - \vhX| \leq | \vX - \tilde{\vX} | +  |\tilde{\vX} - \vhX|$.

Consider the evolution of $\vX$ and $\tilde{\vX}$ in some time interval $[t_0 , t]$ (for $t_0 < t$). We can write $\vX$ and $\tilde{\vX}$ in integral form:
\begin{subequations} \label{eqGroup:mapsIntegroDef}
\begin{align} 
\vX_{[t, t_0]} (\vx) = \vx + \int_{t}^{t_0} \vu \left( \vX_{[t, r]} (\vx), r \right) dr, \\
\tilde{\vX}_{[t, t_0]} (\vx) = \vx + \int_{t}^{t_0} \tilde{\vu} \left( \tilde{\vX}_{[t, r]} (\vx), r \right) dr . \label{eq:modMapIntegroDef}
\end{align}
\end{subequations}

Letting $z_\vx (t) = \vX_{[t, t_0]} (\vx) - \tilde{\vX}_{[t, t_0]} (\vx)$, we have that
\begin{gather}
| z_\vx (t) | \leq \int_{t_0}^{t} \left| \vu \left( \vX_{[t, r]} (\vx), r \right) -  \tilde{\vu} \left( \tilde{\vX}_{[t, r]} (\vx), r \right)  \right| dr \leq \int_{t_0}^t C dr + \int_{t_0}^t  A |z(r)| dr  
\end{gather}
for some $C \approx \| \vu - \tilde{\vu} \|_\infty $ and $A \approx \| \grad \vu \|_\infty $. Hence, by Gr\"onwall's lemma, we have that 
\begin{gather} \label{eq:GronwallModMap}
\vX_{[t, t_0]} (\vx) - \tilde{\vX}_{[t, t_0]} (\vx) \lesssim A  (t-t_0) e^{A (t-t_0)} \| \vu - \tilde{\vu} \|_\infty .
\end{gather}

Heuristically, we note from this that if a numerical solution $\vhX$ approximates well the solution $\tilde{\vX}$ of the modified equation generated by the discretized velocity field $\tilde{\vu}$ (or some approximation of it), that is, if the scheme is ``self-consistent'', then it is sufficient to control the difference between $\tilde{\vu}$ and the true velocity $\vu$. This tool makes the error analysis more straightforward and allows us to obtain better bounds on the error in the conservation of various advected quantities. In particular, we can write the modified equation approximated by the numerical vorticity solution:
\begin{gather} \label{eq:modVortEqn}
\partial_t \omega_0 (\tilde{\vX}) = \grad \omega_0 \cdot  \partial_t \tilde{\vX} = - \grad \omega_0 \cdot (\tilde{\vu} \cdot \grad ) \tilde{\vX} = - (\tilde{\vu} \cdot \grad ) \omega_0 (\tilde{\vX}) .
\end{gather}
That is, the numerical vorticity approximates an advection equation under the modified flow generated by $\tilde{\vu}$, making the error ``advective'' rather than diffusive. The exact nature of the modified equation is unclear and depends on the leading order term of $\vu - \tilde{\vu}$. In cases where spatially truncation errors dominate, we shall see that this is closely related to the Lagrangian-Averaged Euler equations.

\subsubsection{Multiscale Evolution} \label{sec:multiscale}
Due to the presence of different scales in the solution and to limited computational resources, we need to make choices on the degree of spatial truncation appropriate to each evolved quantity. In particular, the Biot-Savart law implies that $\vu$ has a faster decay in its Fourier coefficients compared to $\omega$ and $\vX$. Therefore, it can be represented on a coarser grid without incurring excessive $L^\infty$ error. In this method, we make a similar assumption as in Lagrangian-Averaged Euler (LAE-$\alpha$), that is, that the low frequency features of the velocity dictate the global evolution of the flow and high frequency small features do not need to be solved for exactly: it is sufficient to resolve the coarse scales of the velocity field as long as the fine scale features of the transported vorticity are not lost. With this assumption in mind, we design a method where the representation of the instantaneous velocity as well as the short time deformation map (by ``short time'', we refer to the submaps in the decomposition \eqref{eq:SubmapDecomp}) are done on a grid much coarser than the fine scales present in the vorticity solution. However, since the vorticity is defined as the pullback/rearrangement by $\vX$, all scales in the vorticity are preserved and coherently transported under the smoothed velocity field. Indeed, in the CM method, $\omega$ is defined as a function $\omega_0 \circ \vX$ over the entire domain. The absence of a grid-based discretization of $\omega$ means that we do not incur a spatial truncation error on the vorticity during its evolution. This lack of a grid-scale artificial viscosity means that the small scales are not lost. As a result, the CM method achieves arbitrary resolution on the vorticity field and allows us to separate the scales involved in the problems: the large scales are computed and accurately represented whereas small scales are preserved and passively transported.

Here we also make a parallel between the CM method and the LAE-$\alpha$ equations. The LAE-$\alpha$ equations aim at modelling the flow of an incompressible inviscid fluid on a spatial scale larger than $\alpha$ by taking a Lagrangian average of the velocity field. For more details on the LAE-$\alpha$ and LANS-$\alpha$ formulation, readers can refer to \cite{marsden2003anisotropic, mohseni2003numerical, marsden2001global}. In its vorticity form, the LAE-$\alpha$ can be written as
\begin{subequations} \label{eqGroup:LAE_vort}
\begin{align}
\partial_t \omega + (\vu \cdot \grad) \omega = 0 , \\
\vu = - \curl \PSolve (1 - \alpha \Laplace)^{-1} \omega .  \label{eq:BS_LAE} 
\end{align}
\end{subequations}
The above equation also models the flow of a second-grade non-Newtonian fluid; an analysis of the relation between second-grade non-Newtonian fluids, the vortex blob method \cite{chorin1973numerical} and the LAE-$\alpha$ equations can be found in \cite{oliver2001vortex}.

We see from \eqref{eqGroup:LAE_vort} that the vorticity is transported by $\vu$ and that $\vu$ is obtained from the Biot-Savart law on a smoothed vorticity field $(1 - \alpha \Laplace)^{-1} \omega$. This smoothing effect can also be achieved by a spatial filter during the same Biot-Savart computation in CM. Therefore, a coarse grid representation of the velocity field can in effect result in an approximation of the averaged velocity in the LAE-$\alpha$ equations. The difference is that the quantity of interest in LAE-$\alpha$ is the averaged velocity whereas in CM, we are interested in the vorticity field which contains arbitrarily fine scales. It remains that in cases where the sampling of the velocity corresponds to the $(1- \alpha \Laplace)^{-1}$ smoothing in LAE-$\alpha$, the characteristic map from both formulations are the same and hence generate the same dynamics.

\section{Numerical Implementation} \label{sec:NumMethod}
We present in this section the numerical framework for the CM method. First, we present the Hermite interpolation structure used for spatial discretization as well as the spatial definition of the discretized velocity field. We then provide a time semi-discretization of the equation and make the link to the modified equation in \ref{sec:modEqn}. The space and time discretizations are combined to generate the CM solver for the Euler equations for which we then give some error estimates. Finally, we describe the remapping method which aims to adaptively resolve increasingly complicated spatial features.

\subsection{Spatial Discretization}
We review in this section the definition of Hermite cubic interpolants which is at the heart of all spatial computations in this method. We then present the spatial discretization of the velocity field.

\subsubsection{Hermite Cubic Interpolation} \label{subsec:HermiteInterp}
We assume that the computational domain $\Omega$ is discretized in a rectangular meshgrid $\gG$. In 2D, $\gG$ consists of grid points $\vx_{i,j}$ and has cells $C_{i,j}$ with corners $\vx_{i,j}, \vx_{i+1,j}, \vx_{i, j+1}$ and $\vx_{i+1, j+1}$. We will from now on assume for simplicity that $\Omega$ is a flat torus and $\gG$ is a square grid with $C_{i,j}$ having uniform width $\incr x$.

The space of Hermite cubic functions on $\gG$, $\fV_\gG$, is a finite dimensional subspace of $\fC^1 (\Omega)$ consisting of functions that are bicubic in each $C_{i,j}$. To construct the basis functions, we take a tensor product of the 1D cubic basis (see figure \ref{fig:HCBasis}):
\begin{subequations} \label{eqGroup:1DHBasis}
	\begin{align} 
	& Q_0(x) = 	(1+2 |x|)(1-|x|)^2 & x \in [-1, 1] , \\
	& Q_1(x) = x (1- |x|)^2 & x \in [-1, 1] , \\
	& Q_0(x) = Q_1(x) = 0 & x \notin[-1, 1]  .
	\end{align}
\end{subequations}
These functions have the property that $\partial^c Q_k |_y= \delta^c_k \delta^0_y$ for $c, k \in \{0,1\}$ and $y \in \{-1, 0, 1\}$ and are cubic in $[-1, 0]$ and $[0, 1]$ and are everywhere continuously differentiable.
\begin{figure}[h]
\centering
\includegraphics[width = 0.45\textwidth]{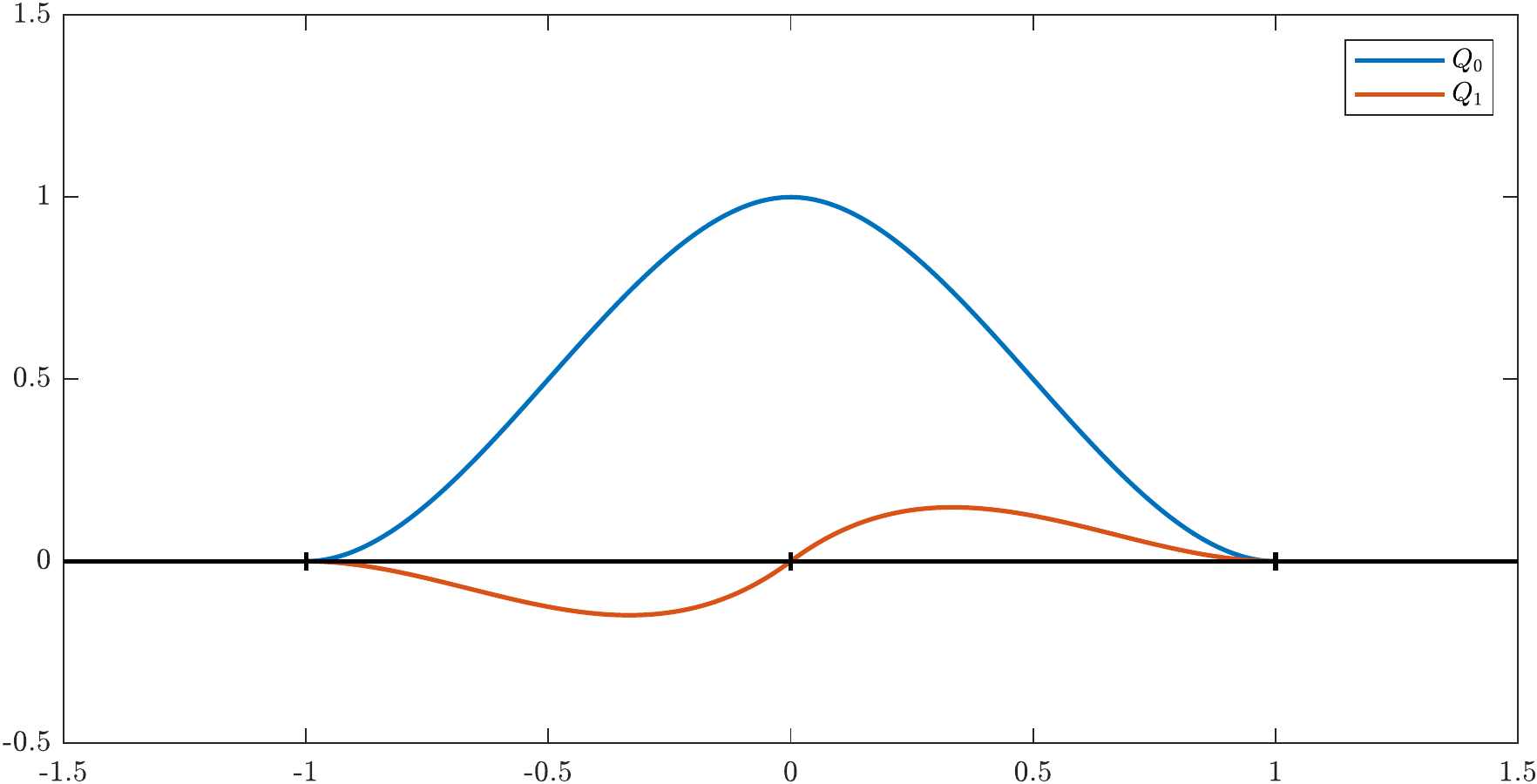}
\caption{1D Hermite cubic basis functions.}
\label{fig:HCBasis}
\end{figure}

We define the 2D basis functions on the grid $\gG$ for $\vx = (x^1, x^2)$:
\begin{gather}
H_{k,l}^{i,j} (\vx) = Q_k \left( \frac{x^1 - x^1_{i,j}}{\incr x} \right)  Q_l \left( \frac{x^2 - x^2_{i,j}}{\incr x} \right) \incr x^{k+l} .
\end{gather}
These basis functions satisfy
\begin{gather}
\partial^{(c,d)} H^{i,j}_{k,l} (\vx_{a, b}) = \delta^a_i \delta^b_j \delta^c_k \delta^d_l .
\end{gather}
Each $H_{k,l}^{i,j}$ is supported on 4 cells $C_{i-1, j-1}, C_{i, j-1}, C_{i-1, j}$ and $C_{i,j}$. It is bicubic in each cell, everywhere $\fC^1$ with continuous mixed derivative $\partial^{(1,1)}$. It has continuous normal derivatives and $\fC^\infty$ tangential derivatives on cell boundaries.

The projection operator $\He_\gG :\fC^1 (\Omega) \to \fV_\gG$ is the interpolation operator
\begin{align}
& \He_\gG [f](\vx)  =  \sum_{i,j} \sum_{ k,l }  f_{i,j}^{k,l} H^{i,j}_{k,l} (\vx) ,  \\
& \text{where} \: f_{i,j}^{k,l} = \partial^{(k,l)} f (\vx_{i,j}) \: \text{ for } f \in \fC^1 (\Omega) . \nonumber
\end{align}

It is well known that for $f \in \fC^4(\Omega)$ the interpolation is order 4:
\begin{gather}
\left\| f - \He_\gG [f]   \right\|_\infty  < c \| D^4 f \|_\infty \incr x^4 ,
\end{gather}
for some constant $c$

Furthermore, we see that if the grid data $f^{k,l}_{i,j}$ is perturbed by $\epsilon$, then the perturbation of the interpolant is
\begin{gather}
\left\| \He_\gG [f] - \He_\gG [f_\epsilon]  \right\|_\infty  = \bigO( \epsilon \incr x^{k+l} ) .
\end{gather}

Therefore, in order to obtain a $4^{th}$ order accurate interpolant of a function $f$, it is sufficient to know the $(k,l)$ derivative of $f$ to order $\bigO( \incr x^{4 - (k+l)} ) $.

\subsubsection{Spatial representation of the velocity field} \label{subsec:VeloDiscr}

In 2D Euler, the vorticity field solves a transport equation under the velocity $\vu$. Assuming we have the global time characteristic map, $\vX (\cdot, t)$, and that the initial condition for the vorticity $\omega_0$ is given analytically, we have that the vorticity at time $t$ is given by 
\begin{gather}
\omega(\vx, t) = \omega_0 (\vX(\vx, t)).
\end{gather}

This evaluation is necessary at every time step for the Biot-Savart law. Numerically, the spatial resolution is limited by the grid's Nyquist frequency resulting in artificial diffusion when defining the velocity $\vu$. Instead of allowing for artificial diffusion by coarse grid sampling (undersampling), we define
\begin{equation} \label{eq:mollVelo}
\vu_\epsilon = \eta_\epsilon * \vu ,
\end{equation}
where $\eta_\epsilon$ is a mollifier supported in a ball of radius $\epsilon$.

Commuting mollification with derivatives, we have that
\begin{subequations} \label{eqGroup:mollBiotSavart}
\begin{align} 
\vu_\epsilon = \curl \psi_\epsilon \quad \text{ for } \: \psi_\epsilon = \eta_\epsilon * \psi , \\
\psi_\epsilon = - \PSolve \omega_\epsilon  \quad \text{ for } \: \omega_\epsilon = \eta_\epsilon * \omega .
\end{align}
\end{subequations}

Finally, using $\omega = \omega_0 \circ \vX$, we can write the pullback of the convolution
\begin{gather} 
\omega_\epsilon(\vx, t_n) = \int_{B_\epsilon(\vx)} \eta_\epsilon(\vx - \vy) \omega_0 \left( \vX(\vy, t_n)  \right) d \vy \\  \nonumber
 = \int_{\vX(B_\epsilon(\vx), t_n)} \eta_\epsilon(\vx - \vX^{-1}(\vy, t_n) ) \omega_0 ( \vy ) | \det (\grad \vX )|^{-1} d \vy \\ \nonumber
 = \int_{\vX(B_\epsilon(\vx), t_n)} \tilde{\eta}_\epsilon(\vX(\vx, t_n), \vy ) \omega_0 ( \vy ) d \vy = \left[  \tilde{\eta}_\epsilon * \omega_0  \right]_{\vX(\vx, t_n)} , \nonumber
\end{gather}
where $\tilde{\eta}_\epsilon (\vx, \vy) = \eta_\epsilon (\vX^{-1}(\vx, t_n) - \vX^{-1}(\vy, t_n))| \det (\grad \vX(\vx, t_n)  )|^{-1}$. The last line above can be verified by plugging in $\vx = \vX^{-1}(\cdot, t_n)$.

During simulations, when sampling the vorticity at discrete locations $\vx_i$, we pick a mollifier such that $\eta_{\epsilon, \vx_i}$ form a partition of unity of the domain. This further guarantees that the numerical average vorticity is equal to the true exact average up to quadrature and Jacobian determinant errors. Given a numerical approximation $\vhX^n$ of $\vX(\cdot, t_n)$, we define the mollified numerical vorticity at time step $t_n$ to be
\begin{gather} \label{eq:pbConvVort}
\omega^n_\epsilon (\vx) := \left[  \eta^n_\epsilon * \omega_0  \right]_{\vhX^n(\vx)} ,
\end{gather}
for $\eta^n_\epsilon (\vx, \vy) = \eta_\epsilon ((\vhX^n)^{-1}(\vx) - (\vhX^n)^{-1}(\vy))| \det (\grad \vhX^n(\vx)  )|^{-1}$. Note that the numerical evaluation of $\omega^n_\epsilon$ is performed by quadrature on the deformed cells $\vhX^n (C_{ij})$ and requires no computations of the inverse or Jacobian determinant of $\vhX^n$. 

\begin{remark}
We only use this mollified vorticity to generate a mollified velocity, the true numerical solution is still defined as $\omega^n = \omega_0 \circ \vhX^n$.
\end{remark}

The velocity field can be obtained from $\omega^n$ through the stream function as in equation \ref{eq:stream}. Numerically we will solve this using a spectral method with Fast Fourier transforms. Let $\Ft$ denote the Discrete Fourier transform operation, we have:
\begin{gather} \label{eq:DiscrStream}
\hstro_\epsilon := - \Ft^{-1} \left[  \PSolve \Ft \left[  \omega^n_\epsilon \right]  \right] ,
\end{gather}
where the above $\PSolve$ solves the Poisson equation in Fourier space and is a diagonal operator.

The divergence-free property of the velocity should be preserved in order for the characteristic map to be volume preserving. Therefore $\vu^n$ should be defined as the curl of some scalar function. To achieve this, we use a grid $\gV$ for the representation of the velocity field and define
\begin{gather} \label{eq:DiscrVelo}
\vu^n_\epsilon := \curl \He_\gV [ \hstro_\epsilon ] . 
\end{gather}

This definition guarantees that $\vu^n_\epsilon \in \text{curl} (\fV_\gV) \subset \{ f \in \fC^0 (\Omega) \: | \: \div f \equiv 0  \}$. Indeed, $\vu^n$ is $C^\infty$ in each cell and continuous across cell boundaries, its divergence is however continuous everywhere and identically 0 due to the continuity of the mixed partials $\partial^{(1,1)} \He_\gV [ \hstro_\epsilon ]$. As a result, the numerical flow is also divergence-free which allows us to control the error on the volume-preserving property of the characteristic map.

\begin{remark}
The grid $\gV$ and the parameter $\epsilon$ used to represent the velocity field are independent of the grid used for the evolution of $\vhX$, the interpolant is also not restricted to Hermite cubics. Consequently, $\gV$ can be made arbitrarily fine and $\epsilon$ arbitrarily small such that $\vu^n_\epsilon$ approach the exact Biot-Savart velocity associated to the numerical vorticity $\omega^n = \omega_0 \circ \vhX$. In fact, choosing a fine grid does not require large computational resource as it only involves an inverse FFT of a zero-padded $\hat{\psi}^n_\epsilon$. On the other hand, it may be computationally advantageous to pick a larger $\epsilon$ as this allows us to perform a coarser sampling of $\omega^n$, a procedure which involves evaluating all the submaps in the characteristic map decomposition.
\end{remark}

\subsection{Time Discretization}

Section \ref{subsec:VeloDiscr} gives us a discretization of the velocity field at time $t_n$ given the characteristic map. In this section, we provide the time discretization for the evolution of the characteristic map based on this $\vu^n_\epsilon$.

The characteristic map can be evolved using the semigroup property \eqref{eq:semigroupInfinitess}. Let $t_n$ be the discrete time steps, with $\incr{t} = t_{n+1} -t_n$, we have
\begin{gather} \label{eq:semigroupStep}
\vX_{[t_{n+1}, 0]} = \vX_{[t_n, 0]} \circ \vX_{[t_{n+1}, t_n]}.
\end{gather}

In order to approximate the one-step map $\vX_{[t_{n+1}, t_n]}$, we extend the velocity $\vu^n_\epsilon$ to the time interval $[t_n, t_{n+1}]$ using an order $p$ Lagrange polynomial in time. This is similar to a multistep method. Let
\begin{gather} \label{eq:defModVelo}
\tilde{\vu}(\vx, t) := \sum_{i=0}^{p-1} l_i(t) \vu^{n-i}_\epsilon(\vx),
\end{gather}
where $l_i$ are the Lagrange basis polynomials for time steps $t_{n-i}$. We note that for each fixed time $t$, $\tilde{\vu}$ is a linear combination of $\vu^{n-i}_\epsilon$. In particular, this implies that if $\div \vu^{n-i}_\epsilon =0$ then $\div \tilde{\vu} =0$.

The one-step map $\vX_{[t_{n+1}, t_n]}$ is then approximated from \eqref{eq:modMapIntegroDef} by a $k$-stage Runge-Kutta integration of the velocity $\tilde{\vu}$ along characteristic curves, for $\incr{t}$ backward in time:
\begin{gather} \label{eq:oneStepMapNum}
\vhX_{[t_{n+1}, t_n]}(\vx) = \vx - \incr{t} \sum_{j=1}^s b_j k_j, 
\end{gather}
with 
\begin{gather} 
k_j = \tilde{\vu} \left(\vx - \incr{t} \sum_{m=1}^{j-1} a_{jm} k_m, t_{n+1} - c_j \incr{t} \right),
\end{gather}
and $k_0 := 0$. The coefficients $a, b, c$ are those of the Butcher tableau corresponding to the explicit Runge-Kutta scheme.

\begin{remark} \label{rmk:smoothVelo}
The smoothness of $\tilde{\vu}$ is required for the convergence of RK schemes. However, the Hermite cubic definition of $\vu^n_\epsilon$ in section \ref{subsec:VeloDiscr} is $\fC^0$. Nonetheless, this does not pose a problem since $\He_\gV [ \hstro_\epsilon ]$ is a piecewise polynomial approximation of $\hstro_\epsilon$, which is smooth: the moduli of smoothness for $\vu^n_\epsilon$ scale with the cell width of $\gV$ (in the limit of infinitely fine grid $\gV$, $\vu^n_\epsilon$ is smooth). Therefore, by appropriately scaling $\gV$ with $\incr{t}$ the lack of smoothness does not affect the convergence.
\end{remark}

\subsection{Characteristic Mapping Method for 2D Incompressible Euler}
We combine the time and space discretization in the previous sections to generate the Characteristic Mapping method for 2D Euler. The evolution of $\vX_{[t, 0]}$ at discrete time steps $t_n$ is given by \eqref{eq:semigroupStep}. We construct the numerical approximation of $\vX$ by evolving a map in the space of Hermite cubic $C^1$ diffeomorphisms in the sense that each coordinate function is a piecewise Hermite cubic polynomial defined on some grid $\gM$. We use the same evolve-project strategy as the CM method for linear advection to advance the characteristic map \cite{CM}:
\begin{align} \label{eq:semigroupStepNum}
\vhX^{n+1} = \He_\gM \left[ \vhX^n \circ \vhX_{[t_{n+1}, t_n]} \right], \\
\vhX^0 (\vx)  = \vx . 
\end{align}
where the one step map $\vhX_{[t_{n+1}, t_n]}$ is given in \eqref{eq:defModVelo} and \eqref{eq:oneStepMapNum}. The velocities $\vu^n_\epsilon$ required to define the one step map are given by \eqref{eq:pbConvVort}, \eqref{eq:DiscrStream} and \eqref{eq:DiscrVelo}. The velocities $\vu^{n-i}_\epsilon$ at previous time steps used in \eqref{eq:defModVelo} are stored until no longer needed (for $p$ steps where $p$ is the order of the Lagrange polynomial chose to represent $\tilde{\vu}$).

\begin{figure}[h]
\captionsetup[subfigure]{justification=centering}
\centering
\begin{subfigure}[h]{0.25\linewidth}
\centering
\includegraphics[width = \linewidth]{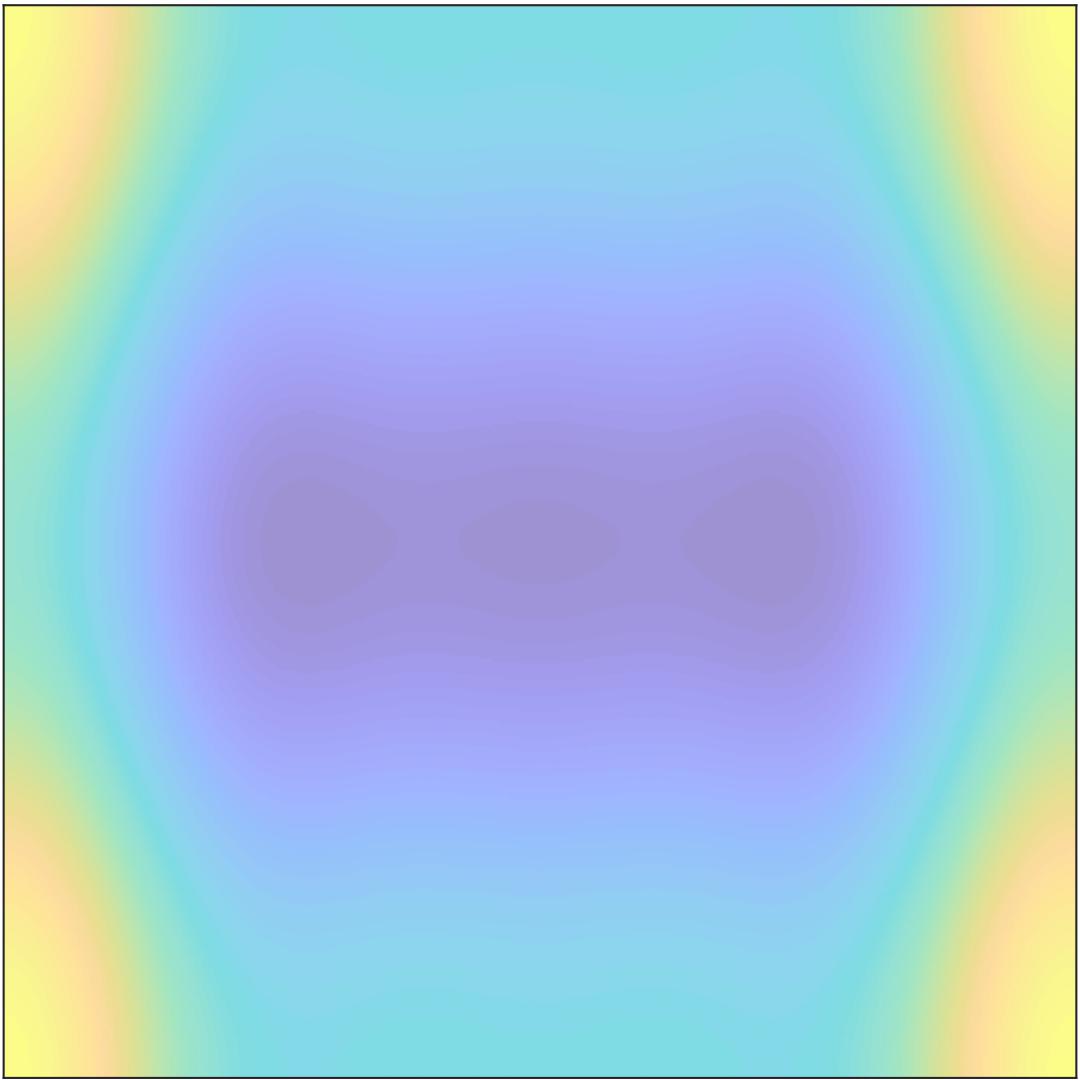}
\caption{Initial condition. \newline}
\label{subfig:wT0}
\end{subfigure}
\begin{subfigure}[h]{0.25\linewidth}
\centering
\includegraphics[width = \linewidth]{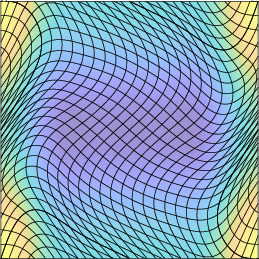}
\caption{Grid deformation from characteristic map.}
\label{subfig:mapT1}
\end{subfigure}
\begin{subfigure}[h]{0.25\linewidth}
\centering
\includegraphics[width = \linewidth]{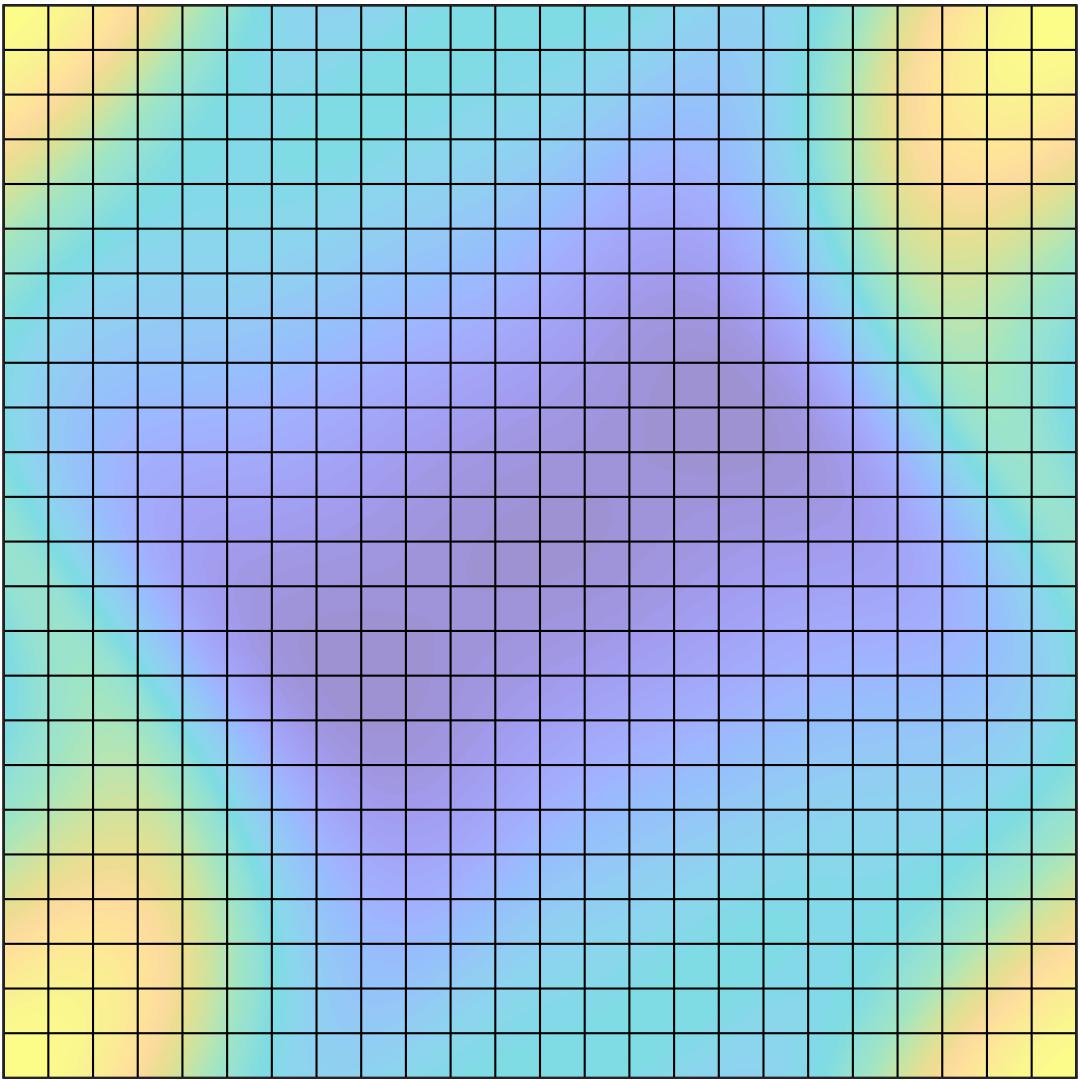}
\caption{Pullback of the initial condition by the map.}
\label{subfig:wmT1}
\end{subfigure}
\caption{Evolution of the characteristic map acting on the initial condition.}
\label{fig:mapExample}
\end{figure}

\subsubsection{Error Estimates} \label{subsec:errors}

We give some estimates on the characteristic map error: 
\begin{gather}  \label{eq:defCharError2}
\mathcal{E}^n := \| \vX ( \vx,t_n)  - \vhX^n (\vx) \|_\infty .
\end{gather}

First, we consider some given numerical solutions $\vhX^n$ for $n = 0, 1, \ldots N$ for some unspecified $N$. The characteristic map at each time step $t_n$ generates a velocity field $\vu^n_\epsilon$. We take the velocity field $\tilde{\vu}(\vx, t)$ for $t \geq 0$ to be the velocity field of the modified equation \eqref{eq:modEqnChar}. We notice that the numerical solution $\vhX$ is exactly the CM discretization of the advection operator $\tilde{\vX}$ for the velocity $\tilde{\vu}$ (taking $\tilde{\vu}$ as given). That is $\vhX$ is a CM method approximation of the advection operator generated by the modified velocity it engenders.

\begin{theorem} \label{thm:CMmodError}
Using an $s$-stage explicit RK integrator with Hermite cubic spatial interpolation, the numerical characteristic map $\vhX^n(\vx)$ approximates $\tilde{\vX}(\vx, t_n)$ to order 
\begin{gather} \label{eq:defFullError}
\| D^{\alpha} (\tilde{\vX}(\vx, t_n) - \vhX^n(\vx)) \|_\infty = \bigO ( t_n ( \incr{x}^{2-|\alpha|} \min( \incr{t}, \incr{x}^2 \incr{t}^{-1} ) + \incr{t}^s )) ,
\end{gather}
for $\alpha \in \{0, 1\}^2$.
\end{theorem}
\begin{proof}
Taking $\tilde{\vu}$ as a fixed velocity, $\vhX$ is simply CM method applied to $\tilde{\vu}$. The error estimates are given in \cite{CM}. It is a property of jet-schemes with Hermite cubic interpolants that we lose one order of convergence for the first mixed derivative only in the spatial error term. This is because time integration in jet-schemes computes the function values and mixed derivatives of degree 1 in each dimension and all interpolants and functions evaluated in the method are at least everywhere $C^1$.

We note that the velocity field $\tilde{\vu}$ is smooth in space (see remark \ref{rmk:smoothVelo}), however it may be discontinuous in time at $t_n$. This does not cause an issue as the smoothness of the velocity is only required in the time step intervals $[t_n , t_{n+1}]$. The local truncation error estimates still hold for the one-step maps $\vhX_{[t_{n+1}, t_n]}$ and $\tilde{\vX}_{[t_{n+1}, t_n]}$; the global truncation error can be obtained from the composition and Hermite interpolation of the one-step maps.
\end{proof}

\begin{corollary} \label{cor:enstrophy}
The CM method for 2D incompressible Euler conserves enstrophy to order $\bigO ( t_n ( \incr{x}^{2} \min( \incr{t}, \incr{x}^2 \incr{t}^{-1} ) + \incr{t}^s ))$.
\end{corollary}
\begin{proof}
Since $\tilde{\vu}$ is by definition divergence-free, we have that $\tilde{\vX}$ is a volume preserving map, i.e. $\det (\grad \tilde{\vX}) = 1$. We get by a change of variable that, with $U = \tilde{\vX}(U)$, 
\begin{gather}
\int_{U} f(\omega_0(\vx))) dx = \int_{\tilde{\vX}(U)} f( \omega_0 (\tilde{\vX}(\vx, t_n))) \det (\grad \tilde{\vX}(\vx, t_n)) dx = \int_{U} f( \omega_0 (\tilde{\vX}(\vx, t_n))) dx ,
\end{gather}
for any measurable $f$. Therefore, we have
\begin{gather}
\int_U  f( \omega_0(\vhX^n(\vx))) dx  - \int_U f(\omega_0(\vx))) dx \approx \int_{U} \grad( f \circ \omega_0) ( \vhX^n(\vx) - \tilde{\vX}(\vx, t_n) ) dx  \\ \leq \| \grad (f \circ \omega_0 ) \|_{L^2} \|  \vhX^n - \tilde{\vX}(\cdot, t_n) \|_{L^2} = \bigO ( t_n ( \incr{x}^{2} \min( \incr{t}, \incr{x}^2 \incr{t}^{-1} ) + \incr{t}^s )) \nonumber
\end{gather}
In particular, taking $f(\omega) = \omega^2$ gives us conservation of enstrophy, and for higher order monomials, implies that the moments of the vorticity are conserved, as they are in the continuous setting.
\end{proof}

To obtain a full error bound, it is sufficient to bound the difference between the true characteristic map and the map from the modified equation. Let 
\begin{gather} \label{eq:defModError}
\tilde{\mathcal{E}}^n := \| \vX(\vx, t_n) - \tilde{\vX}(\vx, t_n) \|_\infty.
\end{gather}
From theorem \ref{thm:CMmodError}, we then have that
\begin{gather} \label{eq:fullErrorDecomp}
\mathcal{E}^n \leq \tilde{\mathcal{E}}^n + \bigO ( t_n ( \incr{x}^{2} \min( \incr{t}, \incr{x}^2 \incr{t}^{-1} ) + \incr{t}^s )) .
\end{gather}

\begin{theorem}
From the above error decomposition we can deduce that the global truncation error for the characteristic map is
\begin{gather} \label{eq:fullCharError}
\mathcal{E}^n = \bigO \left(  \incr{t}^s + \incr{x}^{2} \min( \incr{t}, \incr{x}^2 \incr{t}^{-1} )  +  \incr{t}^p  \right) .
\end{gather}
\end{theorem}

\begin{proof}
It is sufficient to control the evolution of $\tilde{\mathcal{E}}^n$. We note that $\vX_{[t_n, 0]} = \vX_{[t_{n-1}, 0]} \circ \vX_{[t_n, t_{n-1}]}$ and $\tilde{\vX}_{[t_n, 0]} = \tilde{\vX}_{[t_{n-1}, 0]} \circ \tilde{\vX}_{[t_n, t_{n-1}]}$, hence
\begin{gather}
\vX_{[t_n, 0]} - \tilde{\vX}_{[t_n, 0]} = (\vX_{[t_{n-1}, 0]} - \tilde{\vX}_{[t_{n-1}, 0]}) \circ \vX_{[t_n, t_{n-1}]} + \bigO( \vX_{[t_n, t_{n-1}]} - \tilde{\vX}_{[t_n, t_{n-1}]}  ) .
\end{gather}

From our estimates in \eqref{eq:GronwallModMap}, we have that $ \vX_{[t_n, t_{n-1}]} - \tilde{\vX}_{[t_n, t_{n-1}]} = \bigO (\incr{t} \| \vu - \tilde{\vu} \|_\infty )$. Given that $\omega( \vx, t_{n-1}) - \omega^{n-1}(\vx) = \bigO( \mathcal{E}^{n-1})$, we have 
\begin{gather} \label{eq:errorModOneStep}
 \vX_{[t_n, t_{n-1}]} - \tilde{\vX}_{[t_n, t_{n-1}]} = \bigO( \incr{t} \mathcal{E}^{n-1} + \incr{t}^{p+1} ) ,
\end{gather}
where we incurred an extra order $p+1$ error from the Lagrange interpolation.

Therefore, 
\begin{gather} \label{eq:modErrorEvo}
\tilde{\mathcal{E}}^n \leq \tilde{\mathcal{E}}^{n-1} + \bigO( \incr{t} \mathcal{E}^{n-1} + \incr{t}^{p+1} ) \\
= \tilde{\mathcal{E}}^{n-1} + \bigO \left( \incr{t} \tilde{\mathcal{E}}^{n-1}+ \incr{t} ( t_n ( \incr{x}^{2} \min( \incr{t}, \incr{x}^2 \incr{t}^{-1} ) + \incr{t}^s )) + \incr{t}^{p+1} \right) ,
\end{gather}
which implies that
\begin{gather} \label{eq:fullModCharError}
\tilde{\mathcal{E}}^n = \bigO \left(    \incr{x}^{2} \min( \incr{t}, \incr{x}^2 \incr{t}^{-1} ) + \incr{t}^s  + \incr{t}^p  \right) .
\end{gather}
Together with \eqref{eq:fullErrorDecomp}, we obtain the desired error estimate.
\end{proof}

\begin{remark}
Equation \eqref{eq:errorModOneStep} in fact omits a sampling error incurred when defining $\vu^n_\epsilon$, due to sampling the vorticity at discrete points. In this method, we define $\omega^n_\epsilon$ by convolution with a pullback mollifier instead of directly evaluating $\omega_0 \circ \vhX^n$ at sample points. This allows us to justify a Fourier truncation at low number of modes, however, we incur an error term which is second order in the width of the mollifier. This is omitted from the analysis since the sampling grid for $\omega^n_\epsilon$ is independent of the computational grid for the map, hence the sampling error can be controlled separately.
\end{remark}

\subsubsection{Convergence Tests}

We provide here some numerical evidence for the error estimates derived above. We will more extensively test the full method in section \ref{sec:numTests}.

We use a standard four-modes initial condition \eqref{eq:initVortDef} to test the convergence. For the spatial error, we fix $\incr{t}$ at $1/512$ using third order Lagrange interpolation in time and third order explicit Runge-Kutta for time-integration. We then vary the spatial grid size between $32$ to $512$. To test the error in the time variable, we fix the spatial grid at $1024$ and vary $\incr{t}$ between $1/8$ and $1/128$. This time, we test both a second and third order Lagrange interpolant while keeping the same Runge-Kutta scheme as before. This aims to show the independence of the conservation error from the solution error. For all tests, we sample the vorticity on a $1024$ grid and represent the stream function as a piecewise Hermite cubic interpolant on a $2048$ grid. We run the simulation to time $t=1$ and calculate the errors in the following quantities:
\begin{subequations} \label{eqGroup:errorMeasDef}
\begin{align} 
\text{Map error } & := \| \vhX^n - \vX(\cdot, t_n) \|_\infty, \\
\text{Vorticity error } & := \| \omega^n - \omega(\cdot, t_n) \|_\infty , \\
\text{Enstrophy conservation error } & := \|\omega^n\|_{L^2}^2 -  \|\omega_0\|_{L^2}^2 , \\
\text{Energy conservation error } & := \|\vu^n\|_{L^2}^2 -  \|\vu_0\|_{L^2}^2 . 
\end{align}
\end{subequations}

Conservation errors are calculated directly, the sup-norm map and vorticity errors are estimated by comparing each result to the $\incr{x} = 1/1024$, $\incr{t} = 1/512$ test. The functions are evaluated on a 2048 grid, both the $L^\infty$ and $L^2$ norms are approximated by their discrete variant on this grid. The results are shown in figure \ref{fig:convTest}.
\begin{figure}[h]
\captionsetup[subfigure]{justification=centering}
\centering
\begin{subfigure}[h]{0.32\linewidth}
\centering
\includegraphics[width = \linewidth]{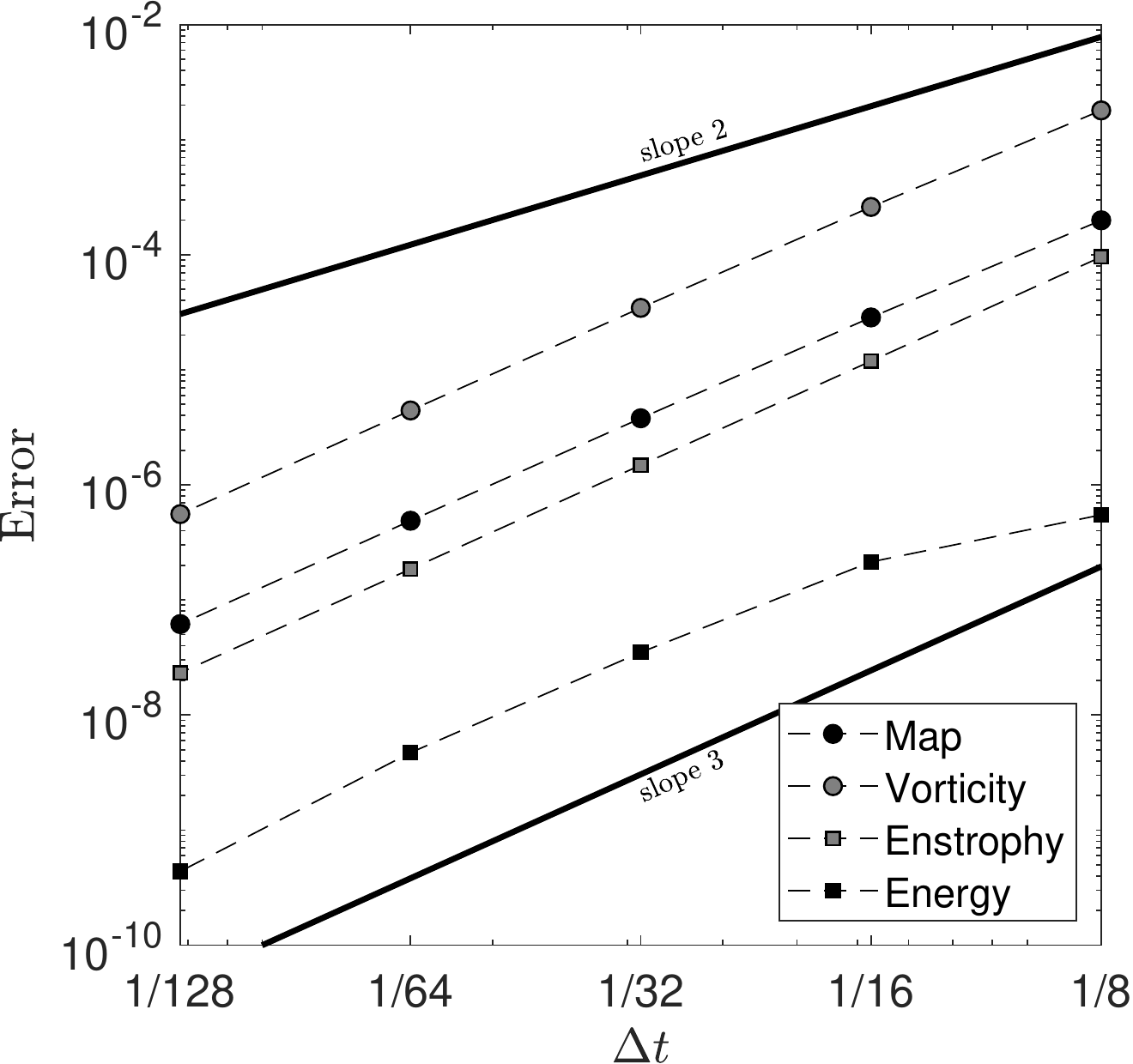}
\caption{Time convergence test,\\$3^{\text{rd}}$ order Lagrange.\\$\incr{x} = 1/1024$.}
\label{subfig:tConvL3RK3}
\end{subfigure}
\begin{subfigure}[h]{0.32\linewidth}
\centering
\includegraphics[width = \linewidth]{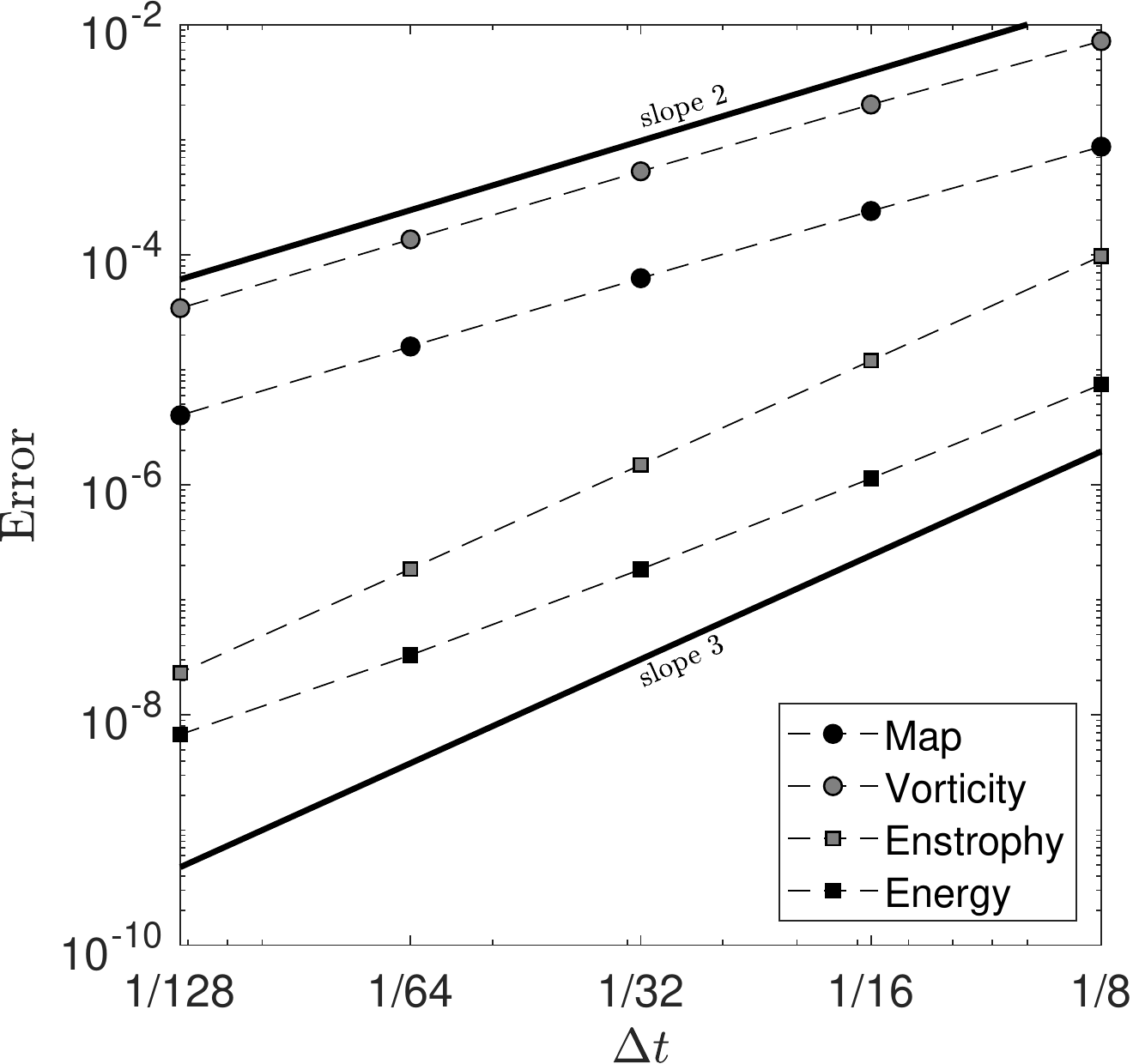}
\caption{Time convergence test,\\$2^{\text{nd}}$ order Lagrange.\\$\incr{x} = 1/1024$.}
\label{subfig:tConvL2RK3}
\end{subfigure}
\begin{subfigure}[h]{0.32\linewidth}
\centering
\includegraphics[width = \linewidth]{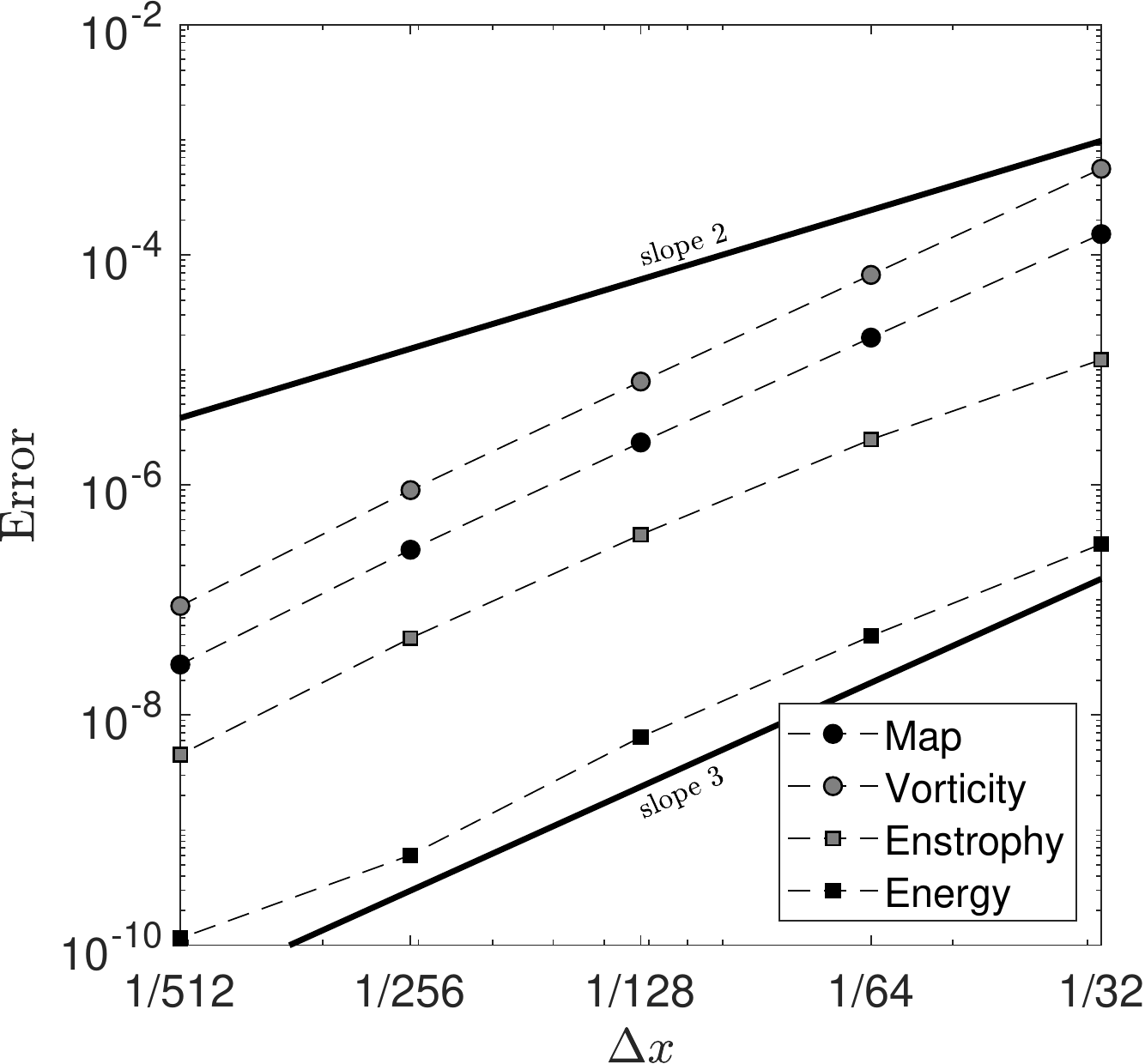}
\caption{Spatial convergence test,\\$3^{\text{nd}}$ order Lagrange.\\$\incr{t} = 1/512$.}
\label{subfig:sConvL3RK3}
\end{subfigure}
\caption{Map and vorticity error and conservation errors of enstrophy and energy for the Characteristic Mapping method without remapping.}
\label{fig:convTest}
\end{figure}

Figures \ref{subfig:tConvL3RK3} and \ref{subfig:tConvL2RK3} are both time convergence plots, the difference is that one uses a third order Lagrange interpolant for the definition of $\tilde{\vu}$ and the other, second order. We see that as expected, the enstrophy conservation error is independent of the choice of Lagrange interpolant and is third order in both cases due to the use of RK3 integration for $\tilde{\vu}$. The error on the map values and vorticity however, do depend on the accuracy of $\tilde{\vu}$ and have third and second order convergence for the respective tests.

Figure \ref{subfig:sConvL3RK3} shows the convergence with respect to $\incr{x}$. Our error estimates suggest a convergence between $\bigO (\incr{t} \incr{x}^2)$ and $\bigO (\incr{t}^{-1} \incr{x}^4)$. For a fixed $\incr{t}$, this is between second and fourth order. This ambiguity comes from the time stepping in GALS methods. The grid data at time $t_{n+1}$ are obtained by evaluating the time $t_n$ Hermite interpolants at $\vhX_{[t_{n+1}, t_n]} (\vx_g)$ for $\vx_g$ a grid point; $\vhX_{[t_{n+1}, t_n]} (\vx_g)$ is commonly called the ``foot-point''. The interpolation error depends on the location of the foot-point relative to grids points. In each dimension, Hermite cubic interpolation errors scales quadratically with both closest grid points. This implies that the interpolation error is $\bigO (\incr{x}^2 \incr{t}^2)$ if $\incr{t} \ll \incr{x}$ and $\bigO( \incr{x}^4)$ otherwise. This is consistent with the third order convergence we see in figure \ref{subfig:sConvL3RK3}.

The experiments in this section suggest that the CM method with Hermite cubic spatial interpolation, third order Lagrange time interpolation and RK3 time integration yields a globally third order method. This is to provide some support for the error estimates in section \ref{subsec:errors}. In practice, since a faithful representation of fine scale features in the velocity field does not contribute very much to the global dynamics and deformation of the domain, we use a coarse grid to represent short time characteristic maps in order to improve efficiency. In this regard, the remapping step presented in the next section will play an important role in maintaining an accurate resolution of the fine scale features in the deformation map generated by long term advection. We will then provide more numerical results and benchmark tests in section \ref{sec:numTests}.

\subsection{Adaptive Remapping and Arbitrary Resolution} \label{subsec:remap}

In the absence of a viscosity term, solutions of Euler equations tend to develop arbitrarily small scale spatial features. As a result, a fixed grid for representing the characteristic map is only valid for a limited amount of time before spatial resolution needs to be increased. Changing the computational grid during simulations can be cumbersome and adversely affect the speed of all computations thereafter. We use instead a decomposition method based on the semigroup property of the characteristic map mentioned in section \ref{sec:SemiGroup}. Numerically, the time $t$ characteristic map can be constructed as the composition of several submaps of time subintervals:
\begin{gather}
\vhX( \cdot, t)  := \vhX_{[ \tau_1, 0]} \circ  \vhX_{[\tau_2, \tau_1]} \circ \cdots \circ  \vhX_{[\tau_{m-1}, \tau_{m-2}]} \circ  \vhX_{[t, \tau_{m-1}]} ,
\end{gather}
for some subdivision $0 < \tau_1 < \tau_2 < \cdots < \tau_{m-1} < t$. Each of these submaps are computed using the CM method described in previous sections. We initialize $\vhX_{[\tau, \tau_i]}$ with the identity map at $\tau =\tau_i$ and evolved until a remapping time $\tau =\tau_{i+1}$ which can be determined dynamically. Once the remapping time is reached, we store $\vhX_{[\tau_{i+1}, \tau_i]}$ in memory and start computing the map for the next subinterval.

Heuristically speaking, each of the subintervals $[\tau_i, \tau_{i+1}]$ should be short enough such that the grid used to discretize $\vhX_{[\tau_{i+1}, \tau_i]}$ can correctly represent the deformation generated by the velocity $\tilde{\vu}$ in this interval. Here we should point out that the  intervals $[\tau_{i}, \tau_{i+1}]$ should be distinguished from the $[t_n, t_{n+1}]$ from previous sections. The latter has length $\incr{t}$ and is used to discretize the one-step maps $\vhX_{[t_{n+1}, t_n]}$; these are immediately composed and projected through $\vhX_{[t_{n+1}, t^\dag]} = \He_\gM \left[ \vhX_{[t_{n}, t^\dag]} \circ \vhX_{[t_{n+1}, t_n]} \right]$, they are used to evolve the maps. The intervals $[\tau_i, \tau_{i+1}]$ on the other hand are longer, each $\vhX_{[\tau_{i}, \tau_{i+1}]}$ comprise of several $\incr{t}$ steps. Once computed, they will be stored in memory.

In this implementation, we use the error in the Jacobian determinant 
\begin{gather} \label{eq:defDetError}
e^n_{\det} := \| \det \grad \vhX^n - 1\|_\infty
\end{gather} 
as a measurement of the map quality, based on which we choose the remapping times. We pick an error threshold $\delta_{det}$ for the submaps. The $i^{\text{th}}$ submap is initialized with the identity map. After each time step, we compute the Jacobian determinant of $\vhX_{[\tau_i + n \incr{t}, \tau_i]}$ at off-grid sample points. If for some $n$, the Jacobian error exceeds $\delta_{det}$ for the first time, we define $\tau_{i+1} := \tau_i + n \incr{t}$ and store $\vhX_{[\tau_{i+1}, \tau_i]}$. The same process is repeated for the $i+1^{\text{st}}$ submap using $\tau_{i+1}$ as initial time.

Although $\| \det \grad \vhX^n - 1\|_\infty < \delta_{\det}$ does not constitute a proper bound on the error in $\vhX$, we can use this as an approximate \emph{a posteriori} error estimate for the gradient of the map compared to that of the map from the modified equation. Indeed, since $\| \vhX - \tilde{\vX} \|_\infty$ is the error from using the CM method on the velocity $\tilde{\vu}$, it is globally third order when $\incr{x} \sim \incr{t}$. Assuming that $\grad \vhX - \grad \tilde{\vX}$ is small, we can justify the following first order expansion
\begin{gather}
e_{\det}= \| \det \grad \vhX - \det \grad \tilde{\vX} \|_\infty \approx \left\|  \text{tr} \left( \grad \tilde{\vX}^{-1} ( \grad \vhX - \grad \tilde{\vX} ) \right) \right\|_\infty = \bigO ( \| \grad \vhX - \grad \tilde{\vX}  \|_\infty  )
\end{gather}

Choosing the remapping times such that for each subinterval, we have $ e_{\det}  < \delta_{\det}$ implies that each submap has $\bigO ( \incr{x} \delta_{\det} )$ error with respect to $\tilde{\vX}$ and hence volume preservation and enstrophy conservation error of the same order. In turn, since the error is 0 at the initial time and we remap at the first time step where this threshold is exceeded, we know that we can assume the error to be small enough to justify the above first order expansion (at least for all previous time steps).

This remapping technique is key in the accurate, dissipation-free resolution of the vorticity field. Qualitatively speaking, there are two types of errors in this method, one is dissipative in nature, and the other, ``advective''. Dissipative error refers to artificial diffusion (or diffusion-like) terms that we incur from spatial truncation of the solution. When we represent an evolving quantity on a fixed spatial grid, the high frequency features of the solution, namely those above the grid's Nyquist frequency, are loss. When these spatial truncations are directly applied to the Euler equations, we get artificial dissipation of $\omega$ or $\vu$ resulting in loss of enstrophy or energy. In the case of Fourier-Galerkin truncation, the dissipative errors can resonate with the solution resulting in numerical artefacts and spurious oscillations \cite{ray2011resonance}.

Due to the discrete nature of numerical computations, truncation errors are somewhat inevitable. In the present method, the evolution of $\vhX^n$ (and only $\vhX^n$) contains a diffusive type error since during each GALS update step, the Hermite cubic interpolation consists of a $4^{\text{th}}$ order averaging of grid values. The leading order error is a $4^{\text{th}}$ order spatial derivative acting like a squared Laplacian. Over time, this accumulated averaging error artificially smooths out the map and resists fine scale deformations which might be present in the true solution. However, since $\omega^n = \omega_0 \circ \vhX^n$, the error in the vorticity is not dissipative in nature. The vorticity is not directly obtained from the previous step $\omega^{n-1}$ and there is no averaging involved. Instead, the error occurs only at the evaluation of $\omega^n$ and is produced by evaluating $\omega_0$ at a wrong position. In fact, since $\vX$ and $\vhX^n$ are both diffeomorphisms of $\Omega$, there exists a diffeomorphism $\bm{\Psi}^n =  (\vhX^n)^{-1} \circ \vX_{[t_n, 0]}$ such that
\begin{gather}
\vX_{[t_n, 0]} = \vhX^n \circ \bm{\Psi}^n,
\end{gather}
The error for $\omega^n$ can then be seen as an advective error in the sense that
\begin{gather}
\omega (\vx, t_n) = \omega^n (\bm{\Psi}^n(\vx)),
\end{gather}
where $\omega$ on the left-hand side refers to the true solution.

This means that qualitatively speaking, the global dynamics of the solution are not obtained from a viscous approximation: the numerical fluid is still inviscid. We make an error on the position of the vortices, controlled by the error of the characteristic map. In particular, it is a straightforward consequence that the numerical solution has the correct $L^\infty$-norm. Moreover, all $L^p$-norms for $1 \leq p < \infty$ are controlled by $\| \vhX - \tilde{\vX} \|_\infty$. Essentially, the CM method places the inevitable diffusive truncation error on the deformation map so that by composition with $\omega_0$, the dissipative error in $\vhX$ manifests itself in $\omega$ as an advective error, hence preserving the inviscid quality of the numerical solution.
 
Going back to the remapping routine, we apply the same principle. In limiting the length of the submap intervals $[\tau_i, \tau_{i+1}]$ by choosing a small $\delta_{\det}$, we limit the amount of artificial diffusion a single submap can accumulate. This prevents the dissipative type error from smoothing out the map and smearing out the fine scale deformations generated by the advection: the global time map is constructed by composition of the short time submaps, also resulting in an advection type error for the map.

In practice, we can interpret the choice of $\delta_{\det}$ in several ways. On one hand, since the use of a smoother $\vu^n_\epsilon$ velocity can be seen as coarse scale discretization of the fluid velocity in the sense of the LAE-$\alpha$ and non-Newtonian fluid equations, we can view the choices of the remapping threshold $\delta_{\det}$ as control on the artificial elasticity of the numerical flow due to spatial truncations. In this sense, the CM method does not make an error on the viscosity, rather it allows for some small controlled elasticity in the fluid. On the other hand, $\delta_{\det}$ controls the error on the volume preserving property of the map. The numerical deformation map is not exactly volume preserving, hence the characteristic paths approximate to those of a fluid that is slightly compressible. Therefore, the CM method avoids numerical dissipation by allowing for a small controlled compressibility in particle paths. It is important to note however that the vorticity field is still advected and is not stretched by the compressibility relaxation of the characteristic map.

Lastly, one can also look at the spatial resolution of the remapping routine from the point of view of the gradient of the represented quantities. Heuristically speaking, the maximum gradient that can be accurately represented on a grid of cell width $\incr{x}$ is $\bigO ( \incr{x}^{-1} )$, that is, the gradient scales roughly with $N$, the number of grid points per dimension. It follows that for an exponentially growing vorticity gradient, the required grid size to avoid excessive truncation errors grows exponentially also. For methods where the evolution of $\omega$ is carried out additively, i.e. methods of the type
\begin{gather}
\omega^{n+1} = \omega^n + \incr{t} \partial_t \omega^n,
\end{gather}
it implies that computations for $\partial_t \omega^n$ must be carried out on an exponentially growing grid.

On the other hand, $\omega$ is not evolved additively in the CM method, the gradient is instead generated by the characteristic map:
\begin{gather}
\grad \omega( \cdot, t_n) = \grad \omega_0 \grad \vhX_{[t_n, 0]}.
\end{gather}
Here, the exponential quantity is $\grad \vhX^n$. This growth is however a natural property of the characteristic map (the map being itself the exponential flow map of the backward time velocity). In fact, the semigroup decomposition \eqref{eq:SubmapDecompGrad} is the intrinsic generating process of the gradient just as multiplication is the generating process of the exponential function. We have
\begin{gather}  
\grad \vhX_{[t, 0]}  = \prod_{j=1}^m  \grad \vhX_{[\tau_j, \tau_{j-1}]} ,
\end{gather}
where the $\ell^2$ operator norm $\| \grad \vhX_{[\tau_j, \tau_{j-1}]} \|_2 $ of each submap gradient is expected to scale exponentially with $\incr{\tau}_j = \tau_j - \tau_{j-1}$. Therefore, by appropriately choosing the remapping criterion, one can make $\incr{\tau}$ small enough that the gradient of each submap is bounded of order $\bigO (1 + \incr{\tau})$, hence representable on a coarse grid. This means that through the semigroup property, we can generate exponential growth in the vorticity gradient without having to do computations on an exponentially growing grid. As we shall see in the next section, this yields a computationally efficient method which captures arbitrarily fine scales and arbitrarily large gradients in the solution.

\section{Numerical Tests} \label{sec:numTests}

In this section we present some numerical tests using the CM method for 2D Euler to simulate incompressible flows starting from some given initial condition. All computations are done on Matlab using double precision. However, the GALS method for evolving the characteristic map uses a $4^{\text{th}}$ order $\epsilon$-difference method to replace the analytic chain rules involved in map updates. This simplifies implementation, but with our choice of $\epsilon = 5\times 10^{-4}$, it effectively limits the machine precision for the characteristic map to around $10^{-13}$.

\subsection{``4-modes'' test} \label{subsec:num4modes}

We test the CM method for 2D incompressible Euler using the ``4-modes'' initial condition. We use the tests performed in \cite{frisch} as reference. The initial vorticity is given by:
\begin{gather} \label{eq:initVortDef}
\omega_0 (x, y) = \cos(x) + \cos(y) + 0.6 \cos(2x) + 0.2 \cos(3x),
\end{gather}
This flow can roughly be characterised as two vortices of opposite signs partitioning a flat torus. The contour plot of $\omega_0$ and $\Laplace \omega_0$ are shown in figure \ref{fig:initVort}. For this flow, we used a $128^2$ grid for the evolution of the maps with a $512^2$ grid to represent $\tilde{\vu}$. The time step $\incr{t}$ was set at $1/32$ and the remapping determinant error threshold to $10^{-4}$. The simulation was run to times $3.5$ and $4$ for the vorticity spectrum plot in figure \ref{fig:compareSpectrum}, and for several times until $t=8$ for the long time simulation in figures \ref{fig:vortContours} an \ref{fig:LVortContour}.

\begin{figure}[h]
\centering
\begin{subfigure}{0.21\linewidth}
\centering
\includegraphics[width = \linewidth]{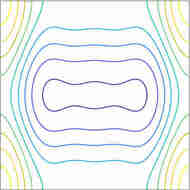}
\caption{$\omega_0$}
\label{subfig:w0}
\end{subfigure}
\hspace*{5pt}
\begin{subfigure}{0.21\linewidth}
\centering
\includegraphics[width = \linewidth]{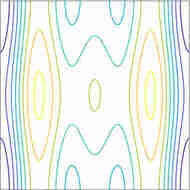}
\caption{$\Laplace \omega_0$}
\label{subfig:Lw0}
\end{subfigure}
\caption{Contour plot of the 4-modes initial vorticity and its Laplacian.}
\label{fig:initVort}
\end{figure}

The same tests were performed in \cite{frisch} up to time $t=5$ using the Cauchy-Lagrange method of various truncation order in time on spatial grids up to $8192^2$. This was necessary due to the presence of large high frequency components in the vorticity at large times and to the necessary anti-aliasing routines in Fourier pseudo-spectral methods. On the other hand, one can justify using a coarser $128^2$ grid for the submap evolution in the CM method since the submaps are remapped and reset to identity before large high frequency features can form. Furthermore, these maps are evolved using the velocity field $\tilde{\vu}$ which, by the Biot-Savart law, has a faster decay in its Fourier coefficients than the vorticity field.

The simulations in this section were carried out on Matlab on a computer with an Intel Core i5-2320 3.00GHz 4 cores processor and 8GB of RAM. For the current simulations, we have not used any parallelization routines. However, almost all the computational time is spent on Hermite interpolations. Parallel and GPU implementation of these operations are standard and could drastically improve the speed of the simulations. Parallelization and application of domain decomposition techniques (possibly for the Biot-Savart kernel) may be of interest for future work.

Figures \ref{fig:vortContours} and \ref{fig:LVortContour} show the contour plots of the vorticity and its Laplacian at length 1 time intervals between 0 and 8. The characteristic maps are computed on a coarse gird, we only use a fine grid sampling of the vorticity to generate the figures. Table \ref{tab:rmrt_T8} shows the number of remaps and total computational times required to reach the various plotting times.

\begin{table}[h] \footnotesize
\begin{center}
{\renewcommand{\arraystretch}{1.5}
\begin{tabular}{c | c c c c c c c c}
\hline
\hline
$t$ & 1 & 2 & 3 & 4 & 5 & 6 & 7 &8 \\
\hline
Number of remaps & 1 & 4 & 14 & 30 & 47 & 65 & 88 & 105 \\
Total CPU time  & 20 s & 41 s & 66 s & 100 s & 145 s & 202 s & 274 s & 359 s \\
\hline
\end{tabular}} \\
\end{center}
\caption{Number of remaps and CPU times for the 4-modes test.}
\label{tab:rmrt_T8}
\end{table}

\begin{figure}[h]
\captionsetup[subfigure]{justification=centering}
\centering
\begin{subfigure}{0.21\linewidth}
\centering
\includegraphics[width = \linewidth]{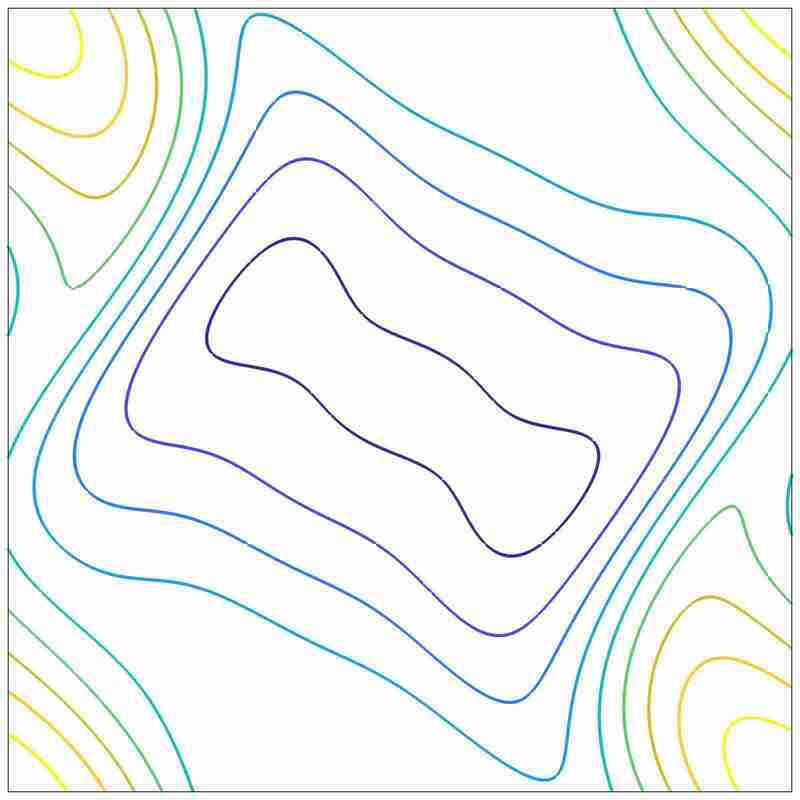}
\caption{$t=1$}
\label{subfig:W1}
\end{subfigure}
\begin{subfigure}{0.21\linewidth}
\centering
\includegraphics[width = \linewidth]{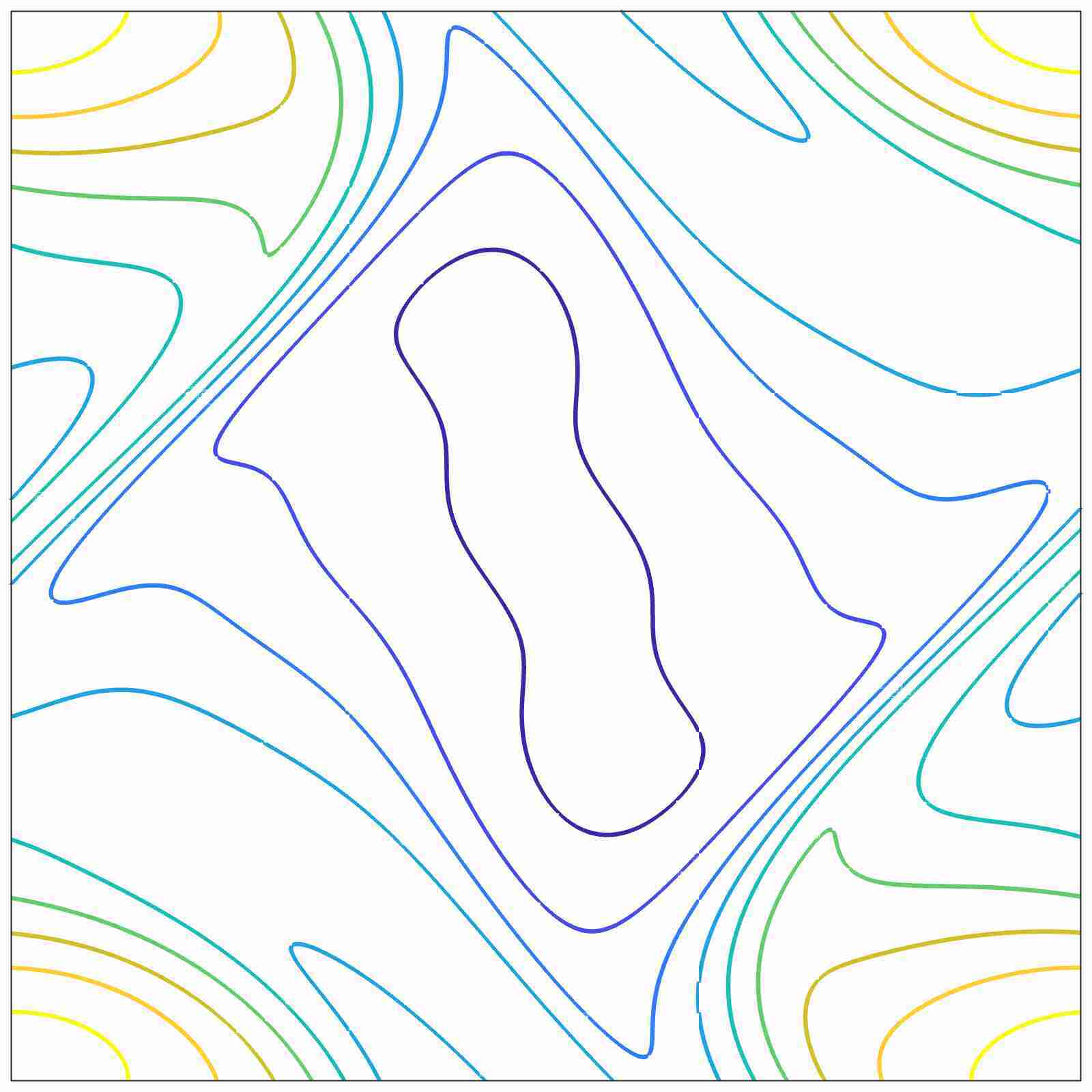}
\caption{$t=2$}
\label{subfig:W2}
\end{subfigure}
\begin{subfigure}{0.21\linewidth}
\centering
\includegraphics[width = \linewidth]{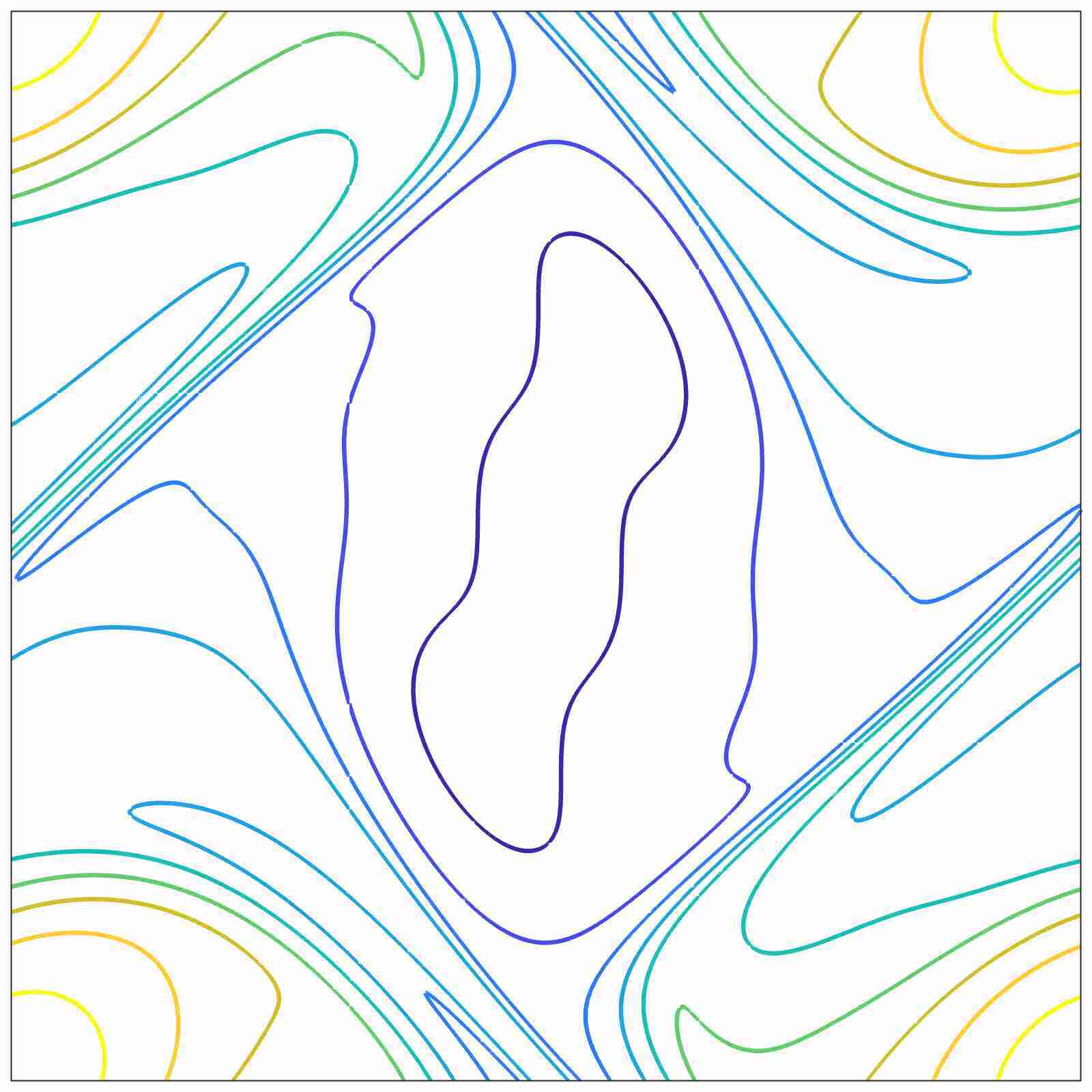}
\caption{$t=3$}
\label{subfig:W3}
\end{subfigure}
\begin{subfigure}{0.21\linewidth}
\centering
\includegraphics[width = \linewidth]{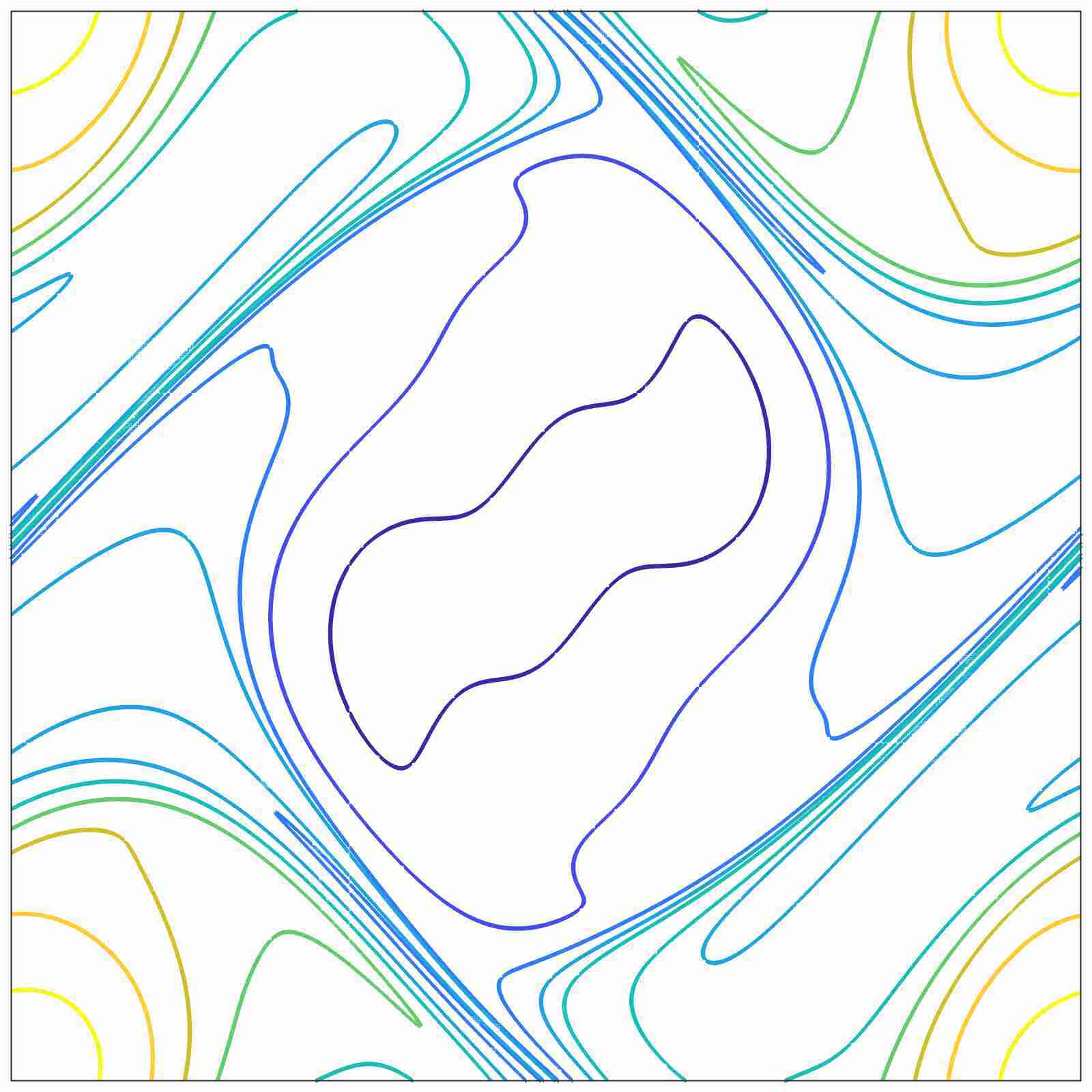}
\caption{$t=4$}
\label{subfig:W4}
\end{subfigure}
\begin{subfigure}{0.21\linewidth}
\centering
\includegraphics[width = \linewidth]{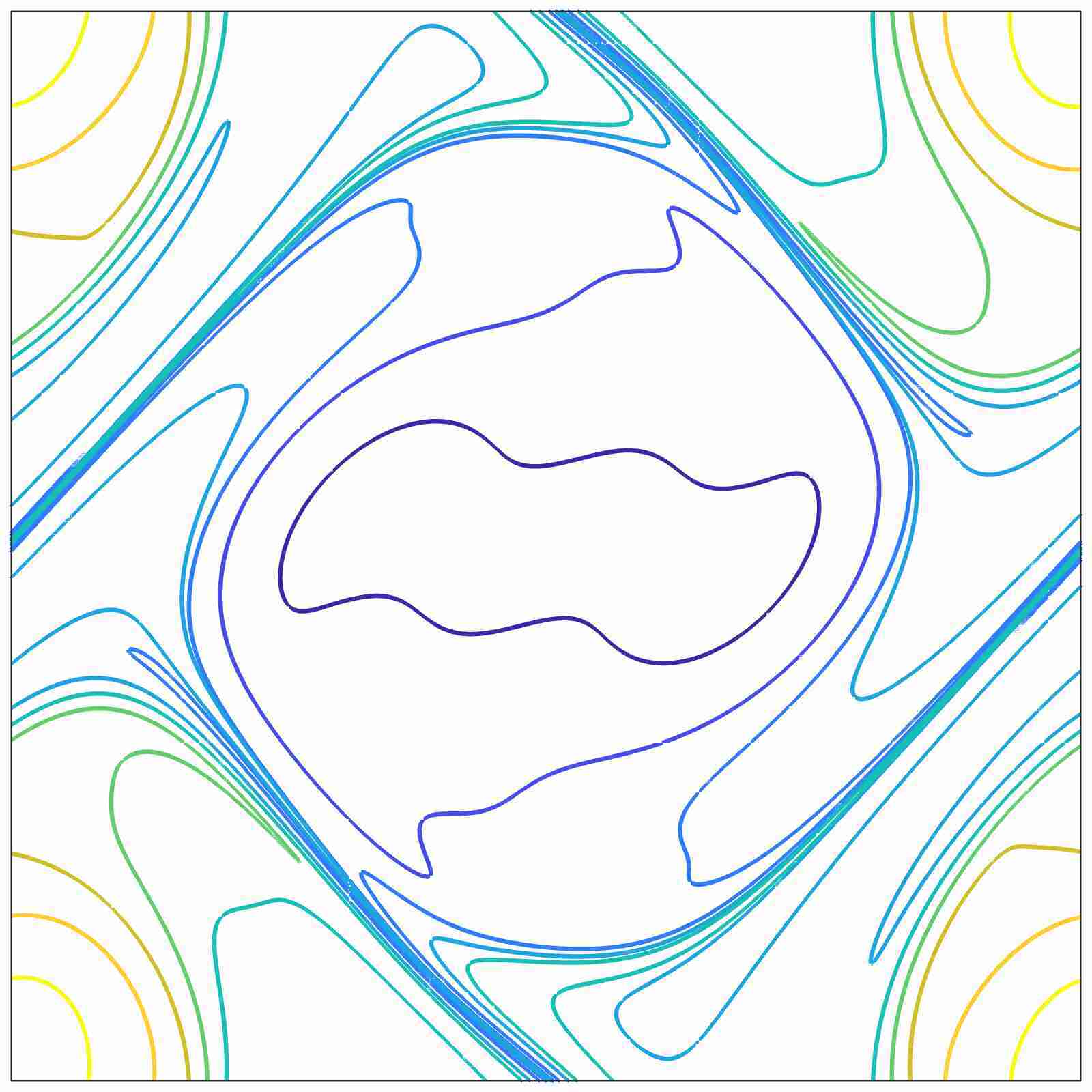}
\caption{$t=5$}
\label{subfig:W5}
\end{subfigure}
\begin{subfigure}{0.21\linewidth}
\includegraphics[width = \linewidth]{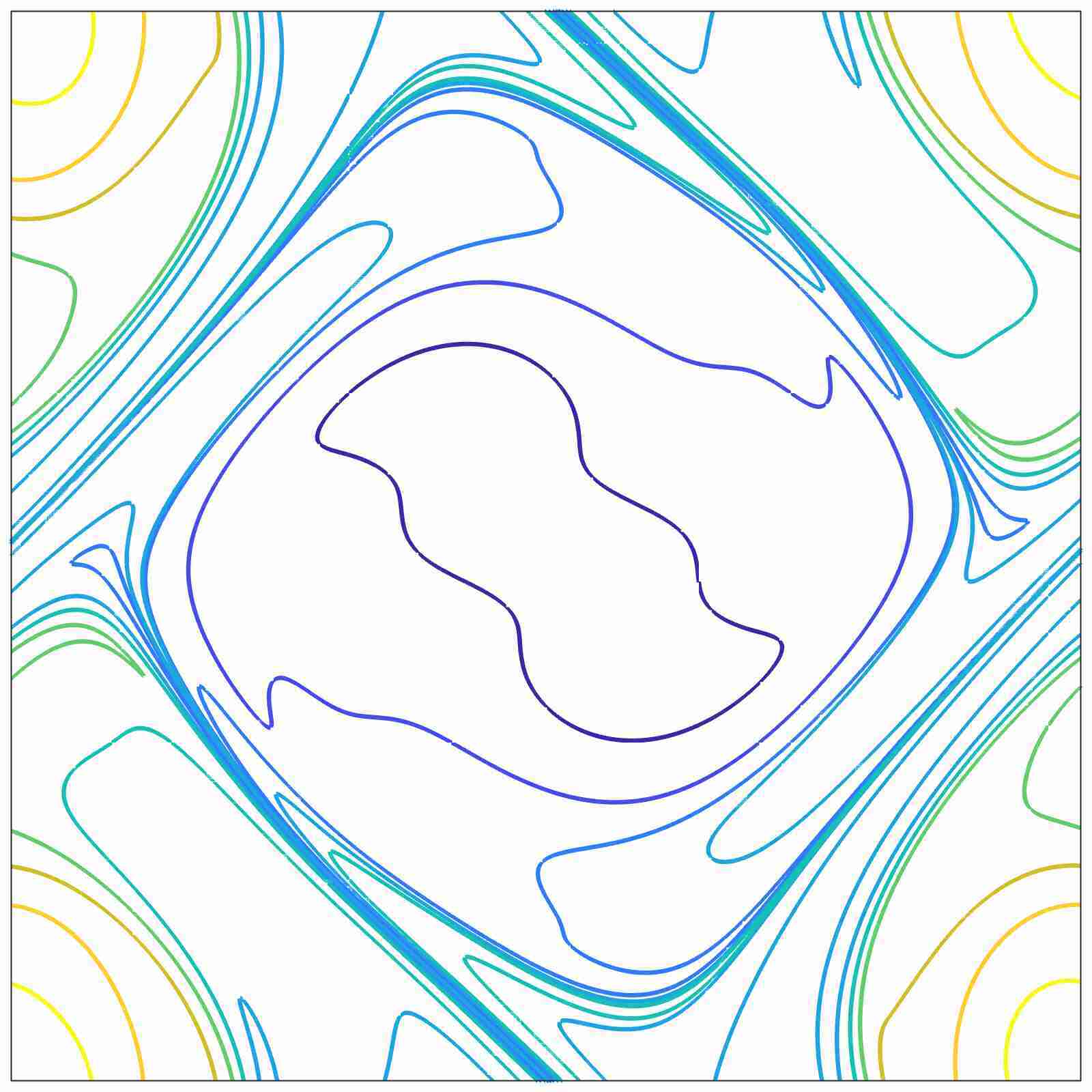}
\caption{$t=6$}
\label{subfig:W6}
\end{subfigure}
\begin{subfigure}{0.21\linewidth}
\includegraphics[width = \linewidth]{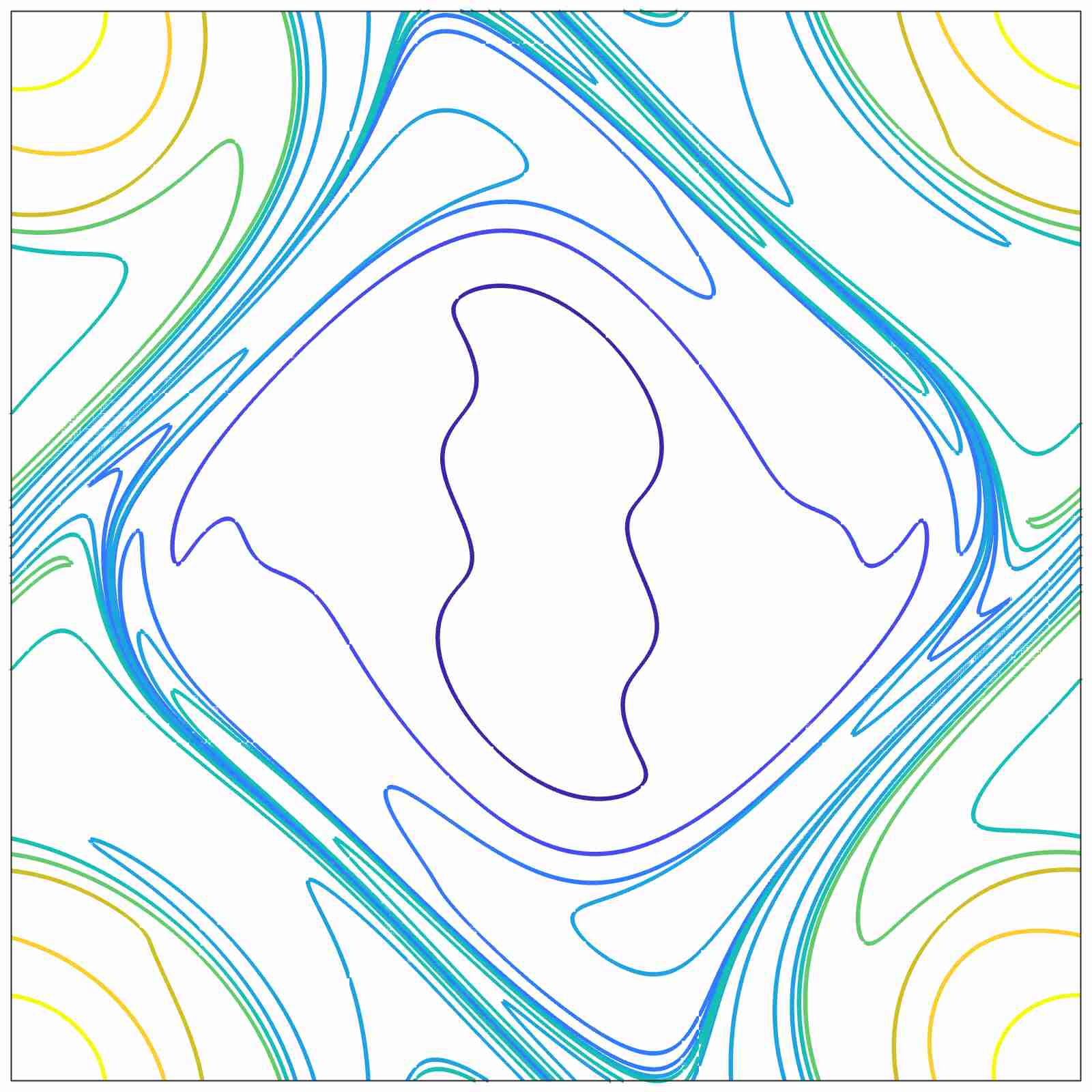}
\caption{$t=7$}
\label{subfig:W7}
\end{subfigure}
\begin{subfigure}{0.21\linewidth}
\includegraphics[width = \linewidth]{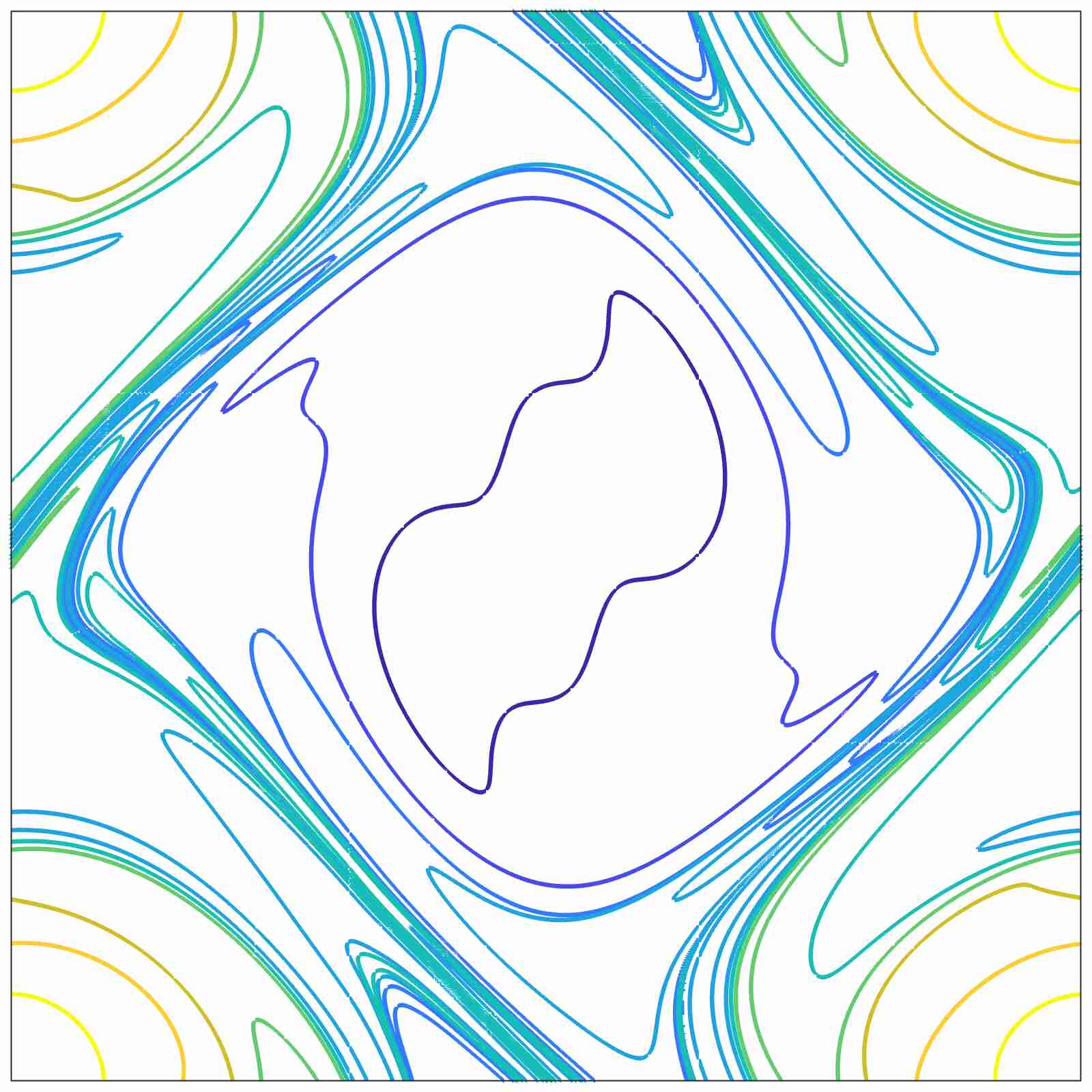}
\caption{$t=8$}
\label{subfig:W8}
\end{subfigure}
\caption{Contour plot of the vorticity using $128^2$ grid for $\vhX^n$, $512^2$ grid for representing $\psi^n$, $\incr{t} = 1/32$ and $\delta_{\det} = 10^{-4}$.}
\label{fig:vortContours}
\end{figure}

\begin{figure}[h]
\captionsetup[subfigure]{justification=centering}
\centering
\begin{subfigure}{0.21\linewidth}
\centering
\includegraphics[width = \linewidth]{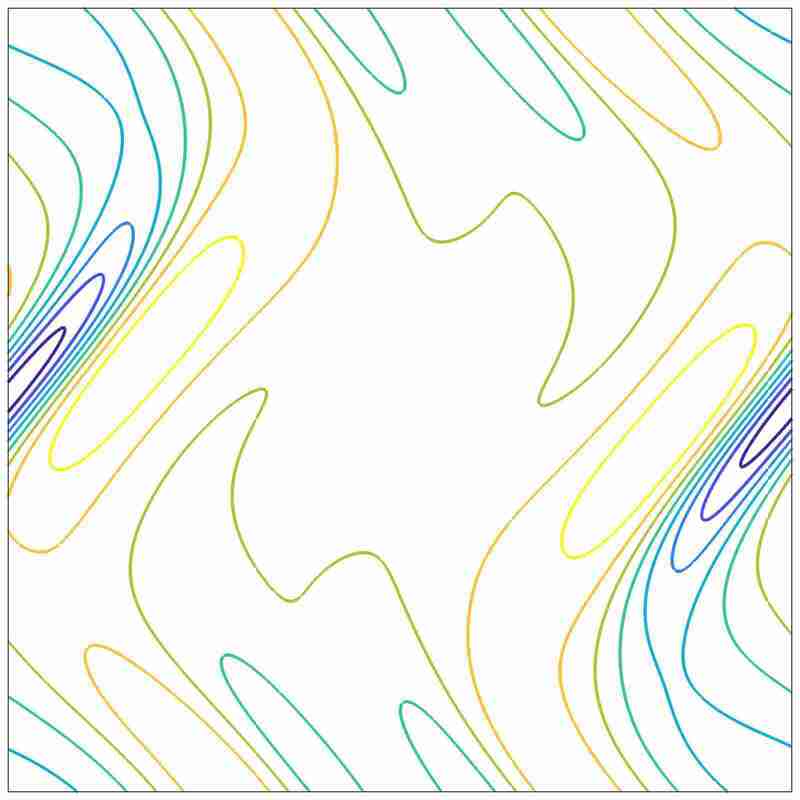}
\caption{$t=1$}
\label{subfig:LW1}
\end{subfigure}
\begin{subfigure}{0.21\linewidth}
\centering
\includegraphics[width = \linewidth]{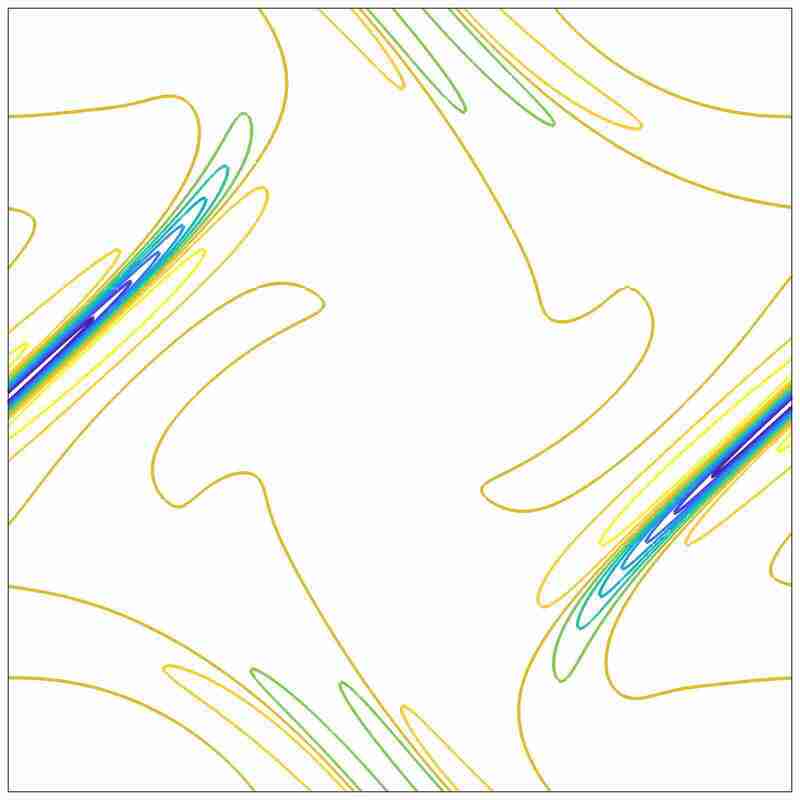}
\caption{$t=2$}
\label{subfig:LW2}
\end{subfigure}
\begin{subfigure}{0.21\linewidth}
\centering
\includegraphics[width = \linewidth]{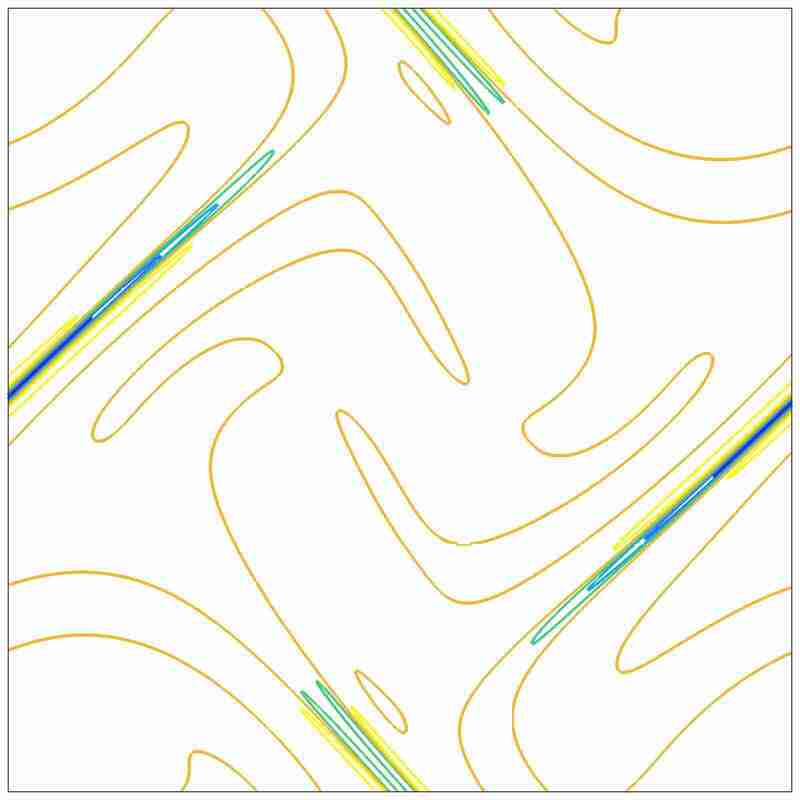}
\caption{$t=3$}
\label{subfig:LW3}
\end{subfigure}
\begin{subfigure}{0.21\linewidth}
\centering
\includegraphics[width = \linewidth]{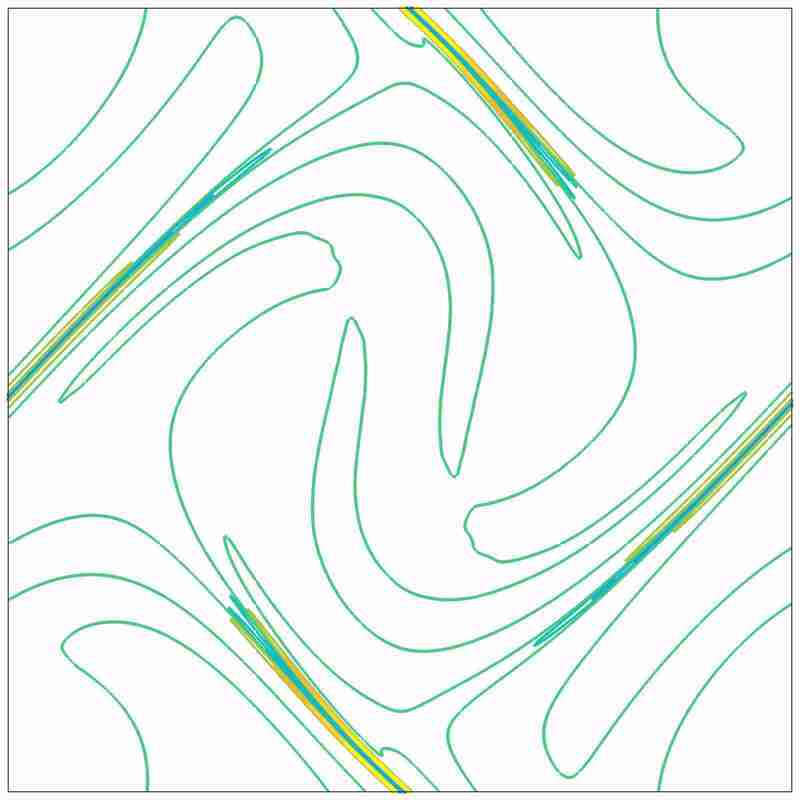}
\caption{$t=4$}
\label{subfig:LW4}
\end{subfigure}
\begin{subfigure}{0.21\linewidth}
\centering
\includegraphics[width = \linewidth]{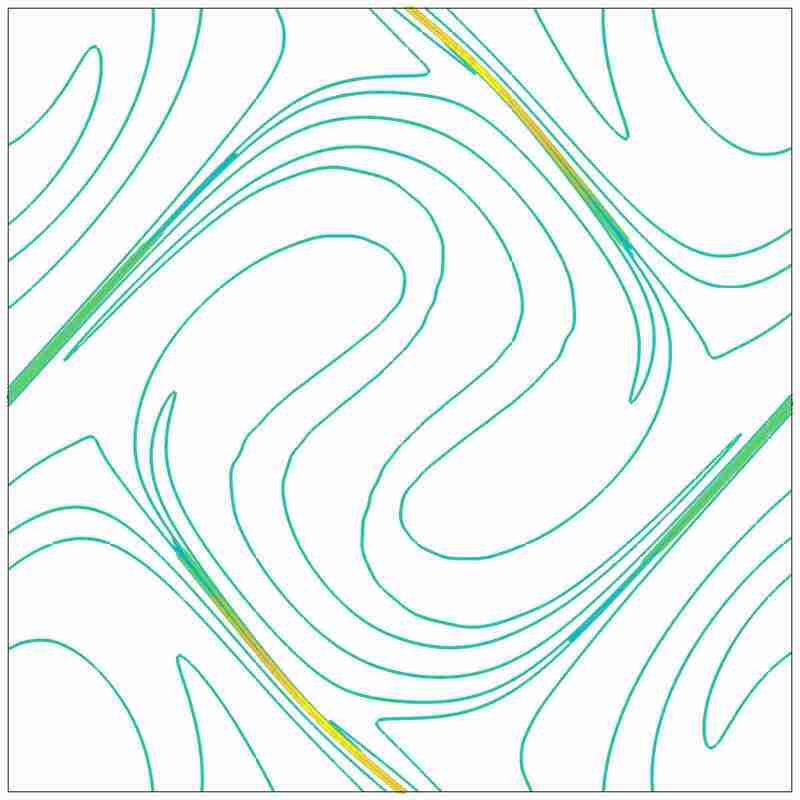}
\caption{$t=5$}
\label{subfig:LW5}
\end{subfigure}
\begin{subfigure}{0.21\linewidth}
\includegraphics[width = \linewidth]{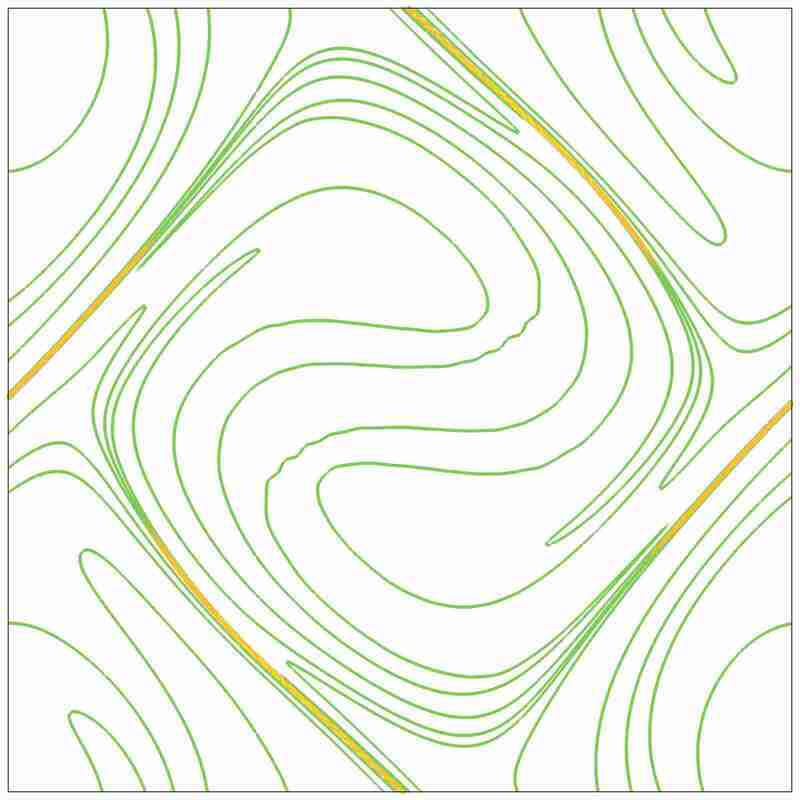}
\caption{$t=6$}
\label{subfig:LW6}
\end{subfigure}
\begin{subfigure}{0.21\linewidth}
\includegraphics[width = \linewidth]{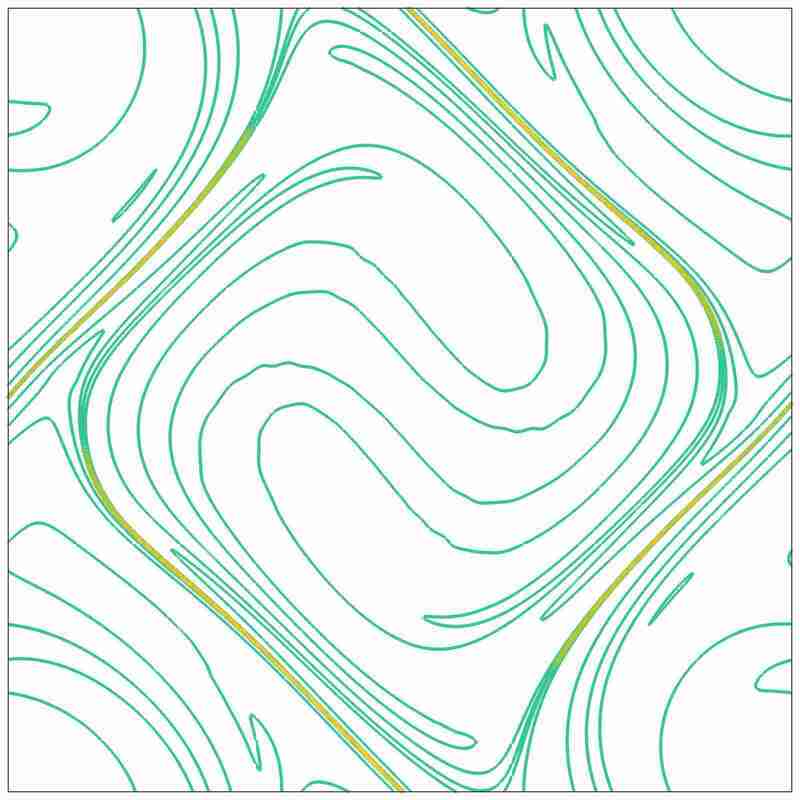}
\caption{$t=7$}
\label{subfig:LW7}
\end{subfigure}
\begin{subfigure}{0.21\linewidth}
\includegraphics[width = \linewidth]{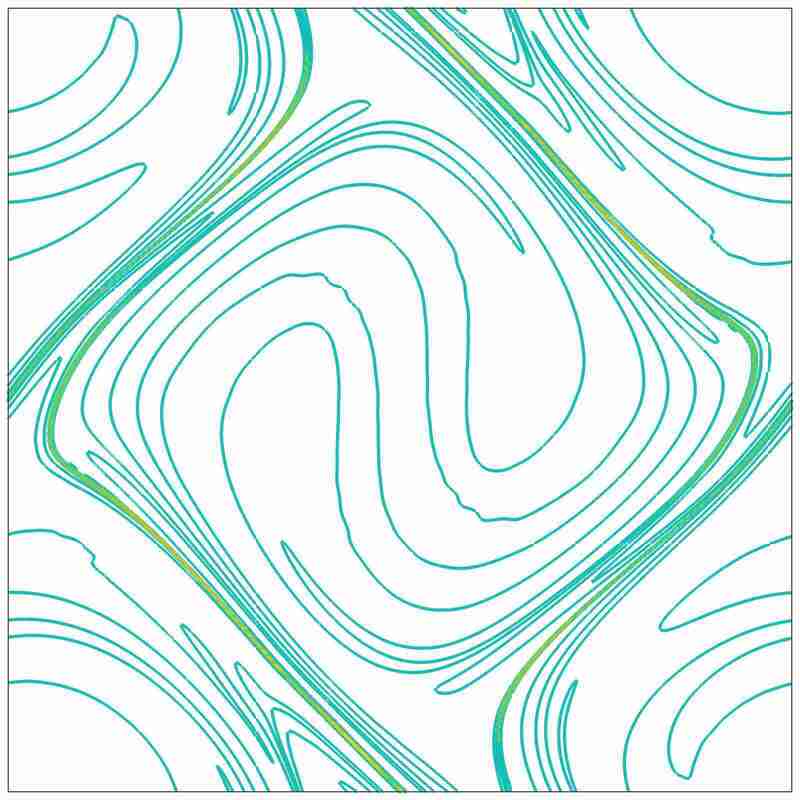}
\caption{$t=8$}
\label{subfig:LW8}
\end{subfigure}
\caption{Contour plot of the Laplacian of the vorticity using $128^2$ grid for $\vhX^n$, $512^2$ grid for representing $\psi^n$, $\incr{t} = 1/32$ and $\delta_{\det} = 10^{-4}$.}
\label{fig:LVortContour}
\end{figure}
From figures \ref{fig:vortContours} and \ref{fig:LVortContour}, we see that for larger times, the flow forms very thin vortex sheets where the two vortices meet. These regions have high vortex gradient and present increasingly fine scale features. For methods employing a fixed grid to represent the solution $\omega^n$, these fine features will eventually become smaller than the grid resolution after which they are lost. This can be interpreted as numerical diffusion associated to the grid size and eventually destroys sharp features of the solution. In the case of conservative high resolution methods such as \cite{frisch}, the truncation of high frequency modes can provoke resonance in the solution leading to a type of spurious oscillations called ``Tygers'': essentially, for a given spatial resolution the numerical solution will reach a time after which numerical artefacts become visible and the solution becomes unstable.

The CM method circumvents this issue by obtaining the solution as a ``rearrangement'' of the initial condition using the backward map. The lack of spatial resolution due to the discrete representation of the map does not result in a diffusive type error in the vorticity. This can be observed in figures \ref{fig:vortContours} and \ref{fig:LVortContour} where we can see that for large times, the solution still contains fine scale features and there are no spurious oscillations, implying that we do not incur diffusive type error in $\omega^n$. There is however a diffusive type error in $\vhX$ akin to an elasticity term. From the point of view of vorticity, the effect of this diffusive error in the map is that the vorticity will be transported along a less violent flow. This error is controlled by the remapping routine. Normally, if we evolve a single characteristic map on a $128^2$ grid, the accumulated diffusive error will prevent sharp deformations to form. Using the remapping method with $\delta_{\det} = 10^{-4}$ we limit the amount of diffusive error in each submap. The global map is constructed from the composition of the submaps and hence is able to represent large shears and the formation of thin vortex sheets. Indeed, we can see from the results that the vorticity develops scales much finer than the $128^2$ grid used for the submap evolution. These scales were absent in the initial condition and are generated from the domain deformation represented by the composition of several submaps.

Another advantage of the remapping routine is that it offers some control over the growth of the enstrophy conservation error. Indeed, since each additional submap transports the vorticity at the previous remapping time, we incur the conservation error in corollary \ref{cor:enstrophy} with respect to the enstrophy at the previous remapping. This means that the enstrophy error accumulates additively when remapping. The error from each submap is controlled through the choice of the remapping tolerance $\delta_{\det}$, thereby providing better long term conservation. The enstrophy and energy conservation errors are shown in table \ref{tab:consError_T8}.
\begin{table}[h] \footnotesize
\begin{center}
{\renewcommand{\arraystretch}{1.5}
\begin{tabular}{c | c c c c }
\hline
\hline
$t$ & 1 & 2 & 3 & 4  \\
\hline
Enstrophy & $1.35 \cdot 10^{-6}$ & $2.75 \cdot 10^{-6}$ & $4.39 \cdot 10^{-6}$ & $6.13 \cdot 10^{-6}$  \\
Energy & $-3.21 \cdot 10^{-8}$ & $-4.14 \cdot 10^{-8}$ & $-5.17 \cdot 10^{-8}$ & $4.66 \cdot 10^{-8}$ \\
\hline
\hline
$t$ & 5 & 6 & 7 & 8 \\
\hline
Enstrophy & $7.90 \cdot 10^{-6}$ & $9.76 \cdot 10^{-6}$ & $1.17 \cdot 10^{-5}$ & $1.37 \cdot 10^{-5}$ \\
Energy & $-3.51 \cdot  10^{-7}$ & $-1.64 \cdot 10^{-6}$ & $-3.98 \cdot 10^{-6}$ & $-7.51 \cdot 10^{-6}$ \\
\hline
\end{tabular}} \\
\end{center}
\caption{Conservation errors for the 4-modes test using the CM method.}
\label{tab:consError_T8}
\end{table}

Compared to a direct grid based representation of the vorticity, this growth is much slower. Indeed, in \cite{frisch}, using the $8^{th}$ order Cauchy-Lagragian method on a $1024^2$ grid, the enstrophy error increases from $10^{-14}$ to $10^{-12}$ to $10^{-6}$ for times $1$, $3$ and $5$, whereas for the CM method, the enstrophy error seems to grow linearly with time.

The remapping method combined with the functional representation of the characteristic map offers the possibility of arbitrary spatial resolution of the solution. Indeed, since the interpolation structure of the submaps implies that the global map $\vhX_{[t, 0]}$ can be readily evaluated anywhere in the domain, if follows that the vorticity of any quantity transported by the flow can also be evaluated anywhere. With the accuracy control provided by the remapping method, this means that solutions can be faithfully represented at an arbitrary resolution. We illustrate this property by gradually zooming into the solution at times 4 and 8.

\begin{figure}[h]
\centering
\begin{subfigure}{0.21\linewidth}
\includegraphics[width = \linewidth]{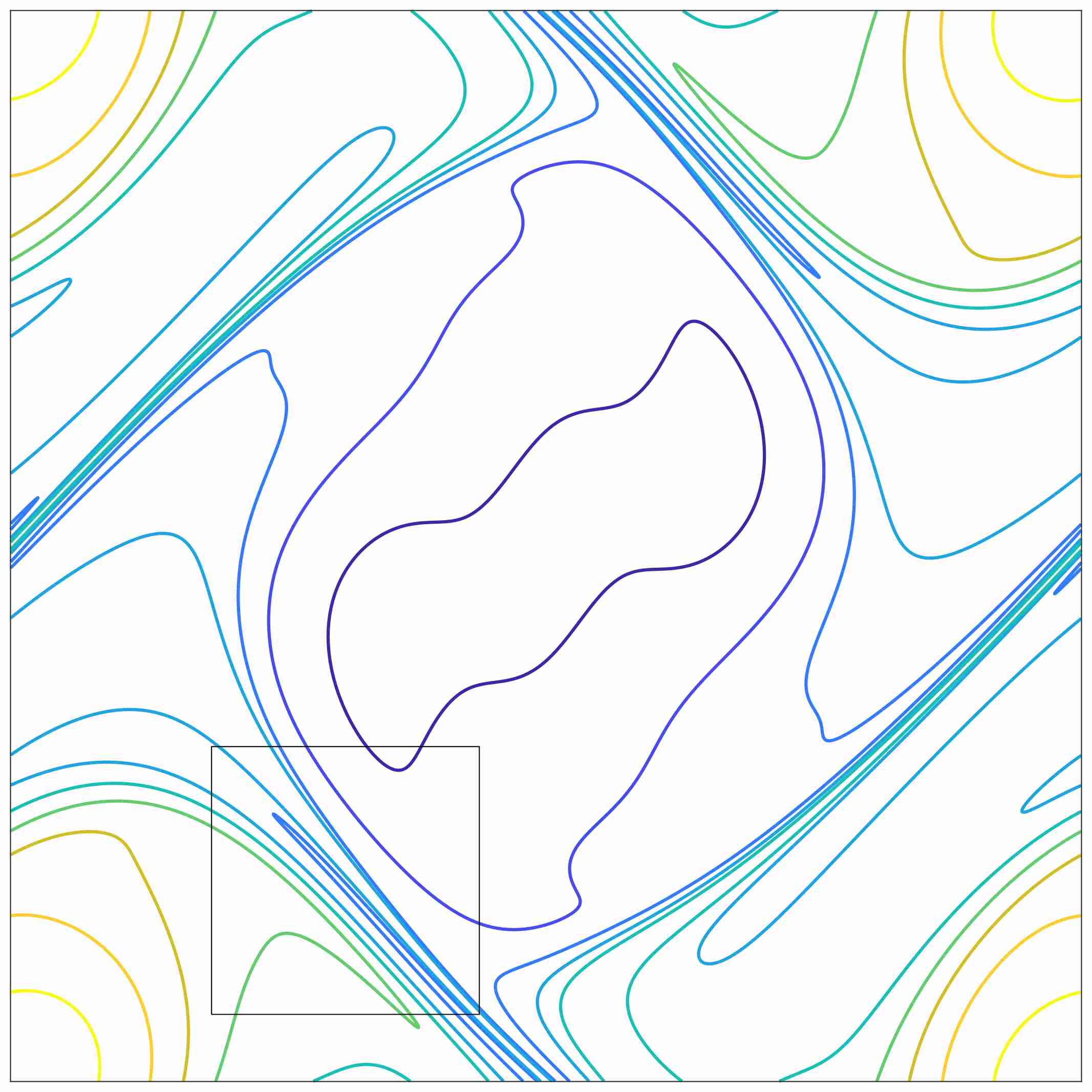}
\caption{$[0,1] \times [0, 1]$}
\label{subfig:Wz1t4}
\end{subfigure}
\begin{subfigure}{0.21\linewidth}
\includegraphics[width = \linewidth]{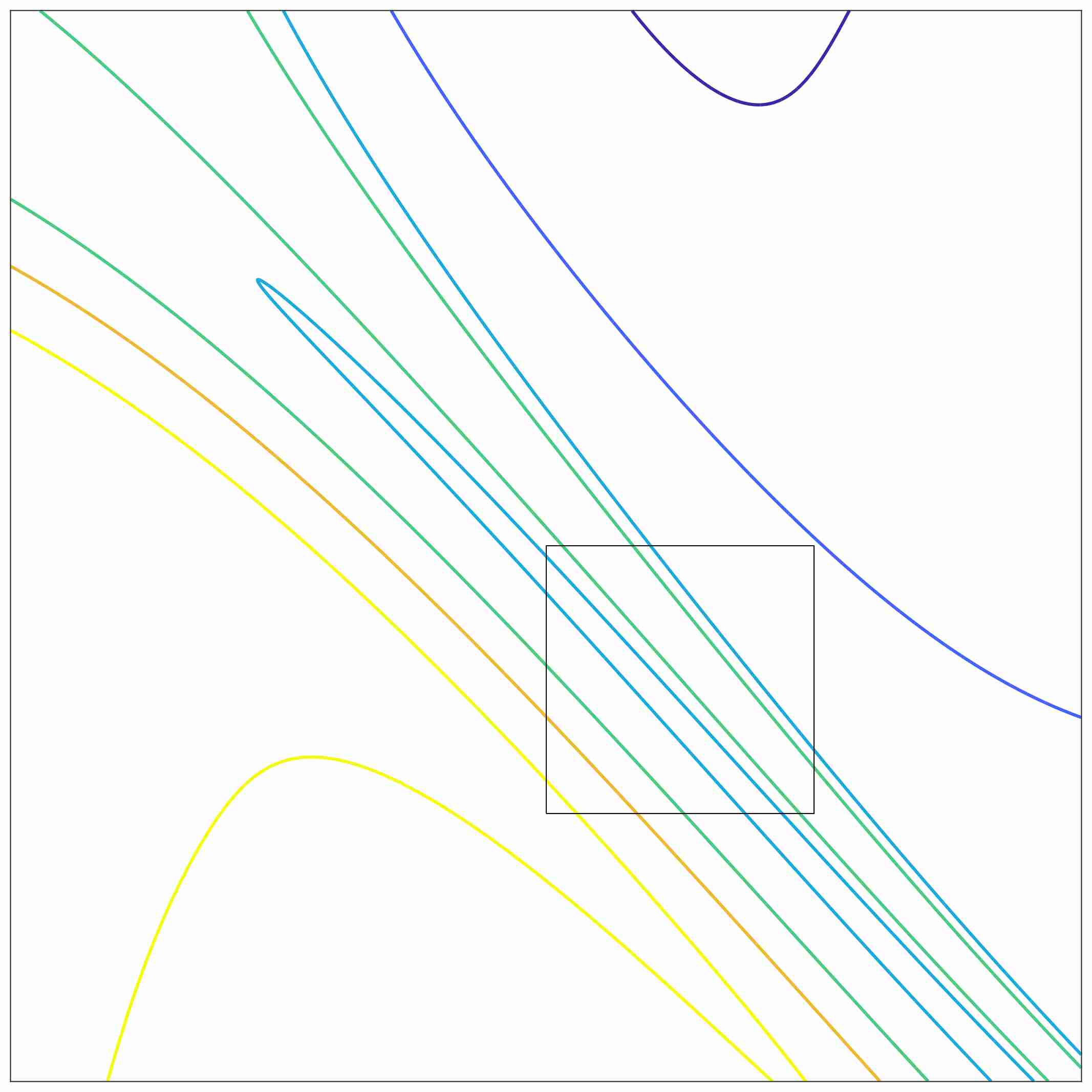}
\caption{$[\frac{3}{16},\frac{7}{16}] \times [\frac{1}{16}, \frac{5}{16}]$}
\label{subfig:Wz2t4}
\end{subfigure}
\begin{subfigure}{0.21\linewidth}
\includegraphics[width = \linewidth]{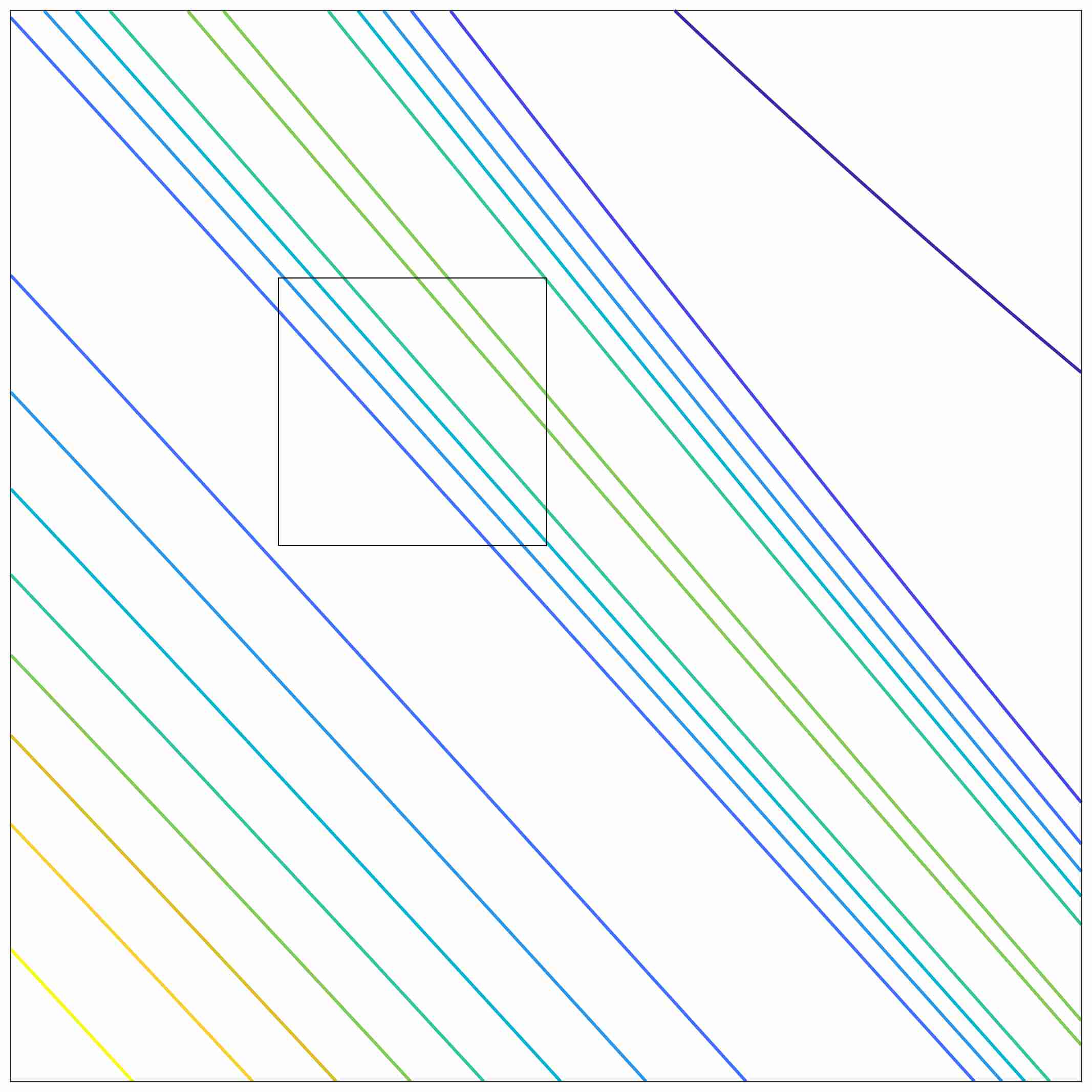}
\caption{$[\frac{5}{16}, \frac{6}{16}] \times [\frac{2}{16}, \frac{3}{16}]$}
\label{subfig:Wz3t4}
\end{subfigure}
\begin{subfigure}{0.21\linewidth}
\includegraphics[width = \linewidth]{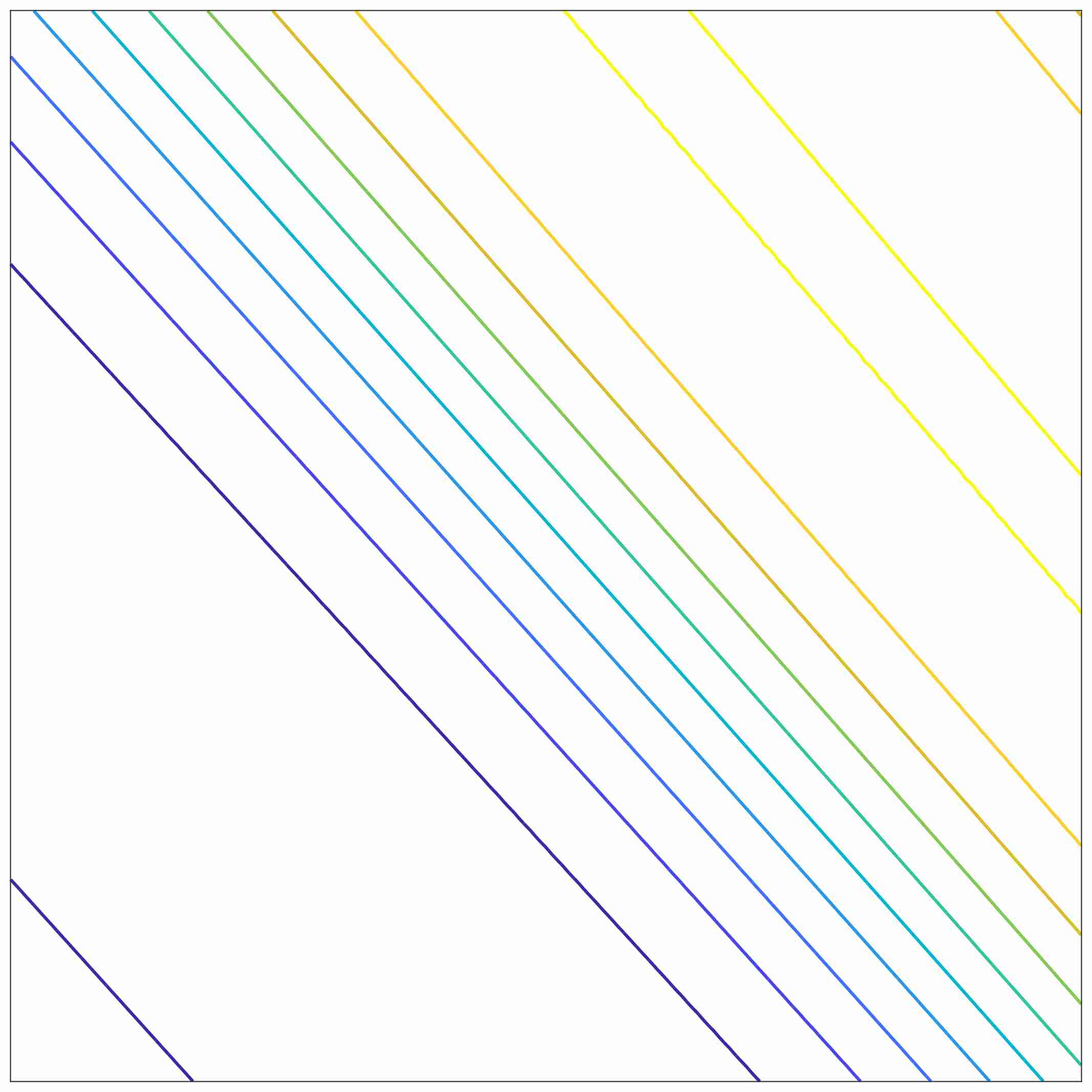}
\caption{$[\frac{21}{64}, \frac{22}{64}] \times [\frac{10}{64}, \frac{11}{64}]$}
\label{subfig:Wz4t4}
\end{subfigure}
\caption{Gradual $64\times$ zoom on the vorticity at $t=4$.}
\label{fig:Wzoom_t4}
\end{figure}

\begin{figure}[h]
\centering
\begin{subfigure}{0.21\linewidth}
\includegraphics[width = \linewidth]{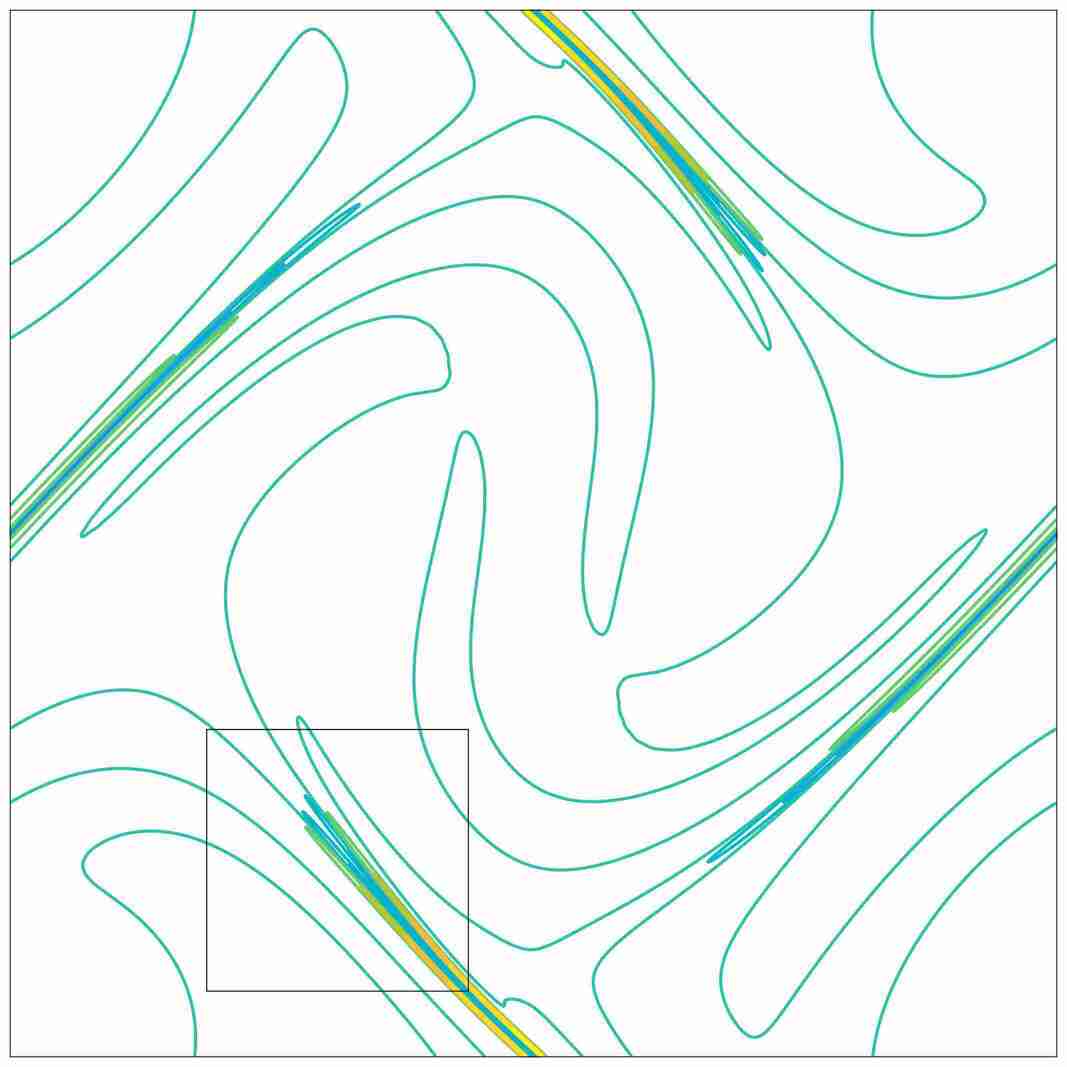}
\caption{$[0,1] \times [0, 1]$}
\label{subfig:LWz1t4}
\end{subfigure}
\begin{subfigure}{0.21\linewidth}
\includegraphics[width = \linewidth]{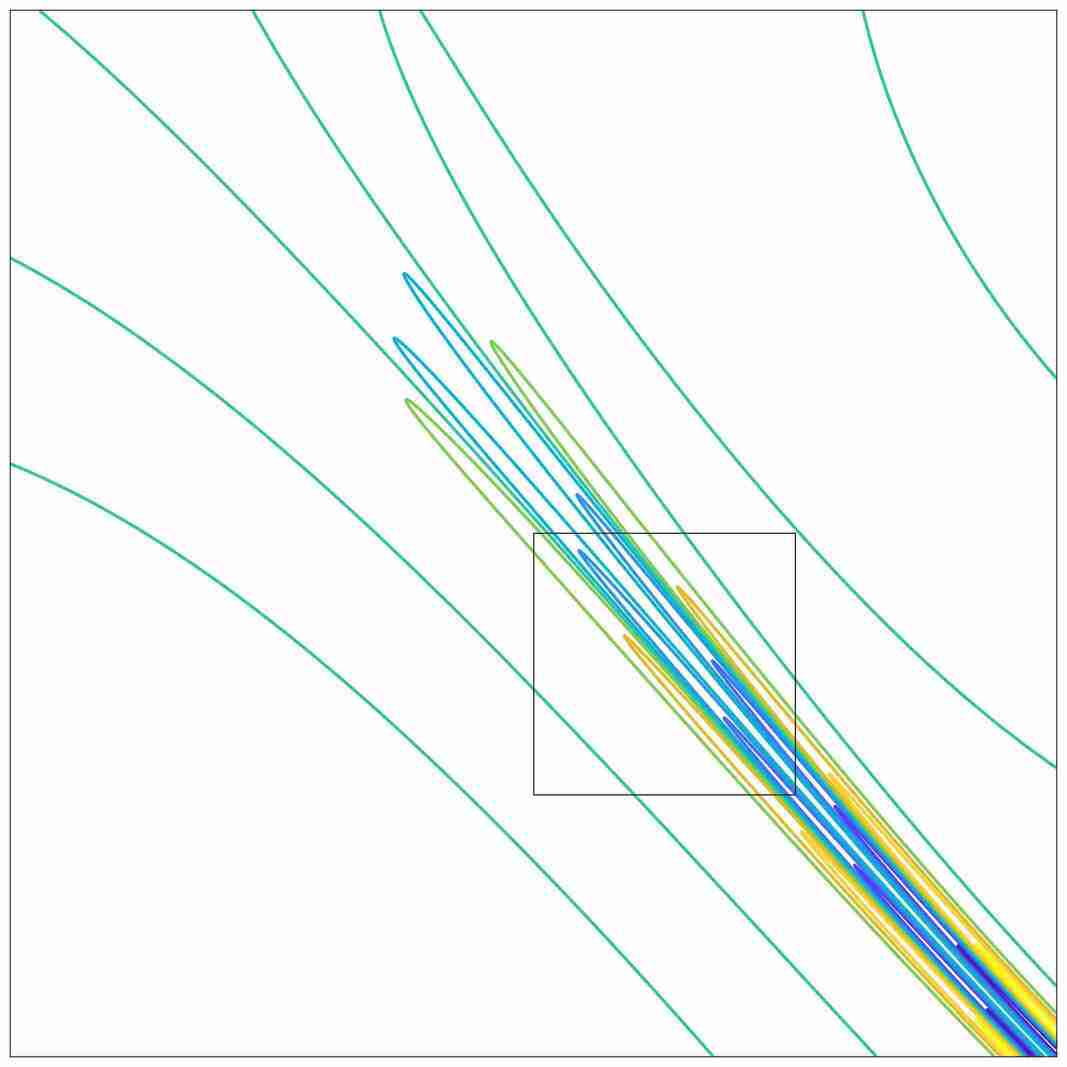}
\caption{$[\frac{3}{16},\frac{7}{16}] \times [\frac{1}{16}, \frac{5}{16}]$}
\label{subfig:LWz2t4}
\end{subfigure}
\begin{subfigure}{0.21\linewidth}
\includegraphics[width = \linewidth]{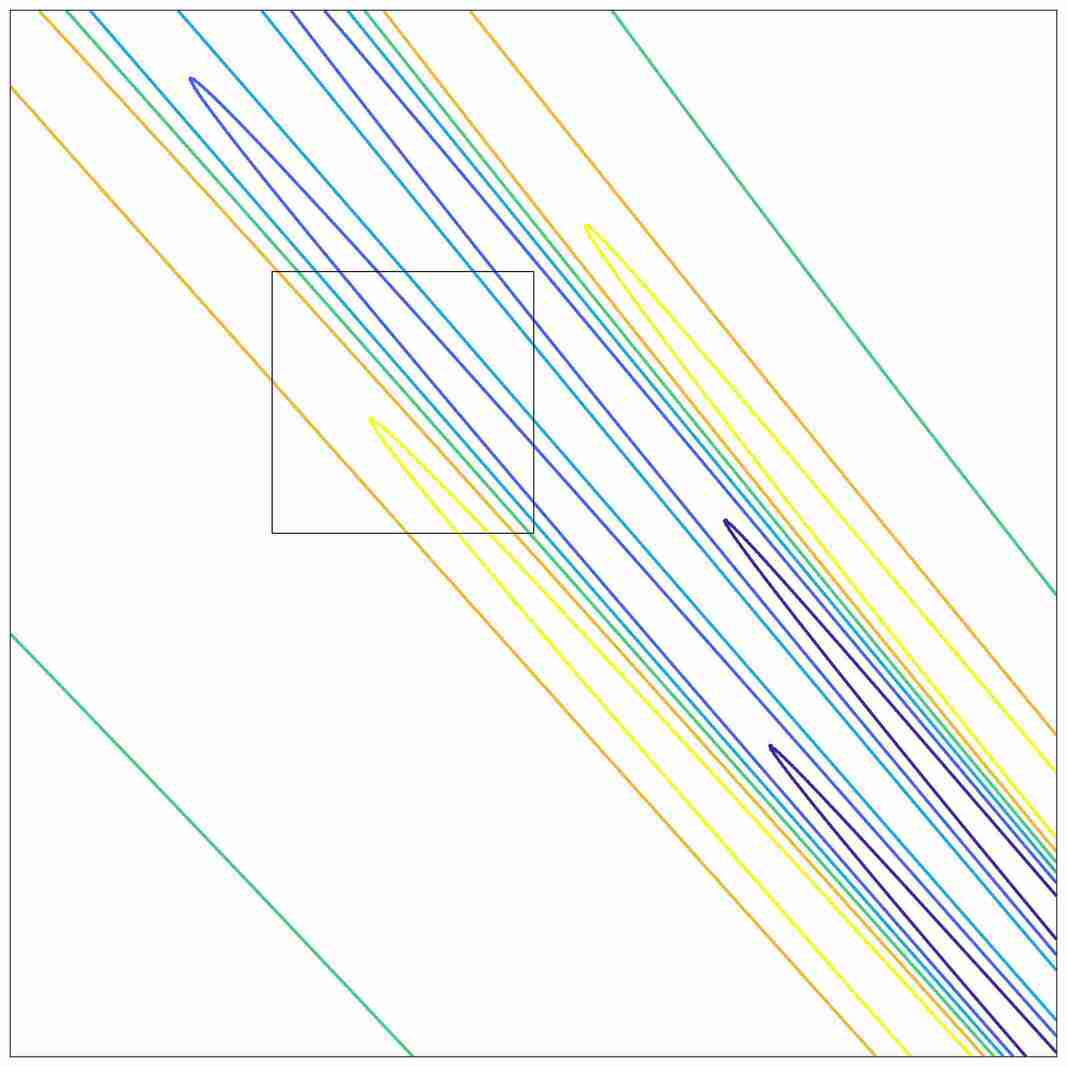}
\caption{$[\frac{5}{16}, \frac{6}{16}] \times [\frac{2}{16}, \frac{3}{16}]$}
\label{subfig:LWz3t4}
\end{subfigure}
\begin{subfigure}{0.21\linewidth}
\includegraphics[width = \linewidth]{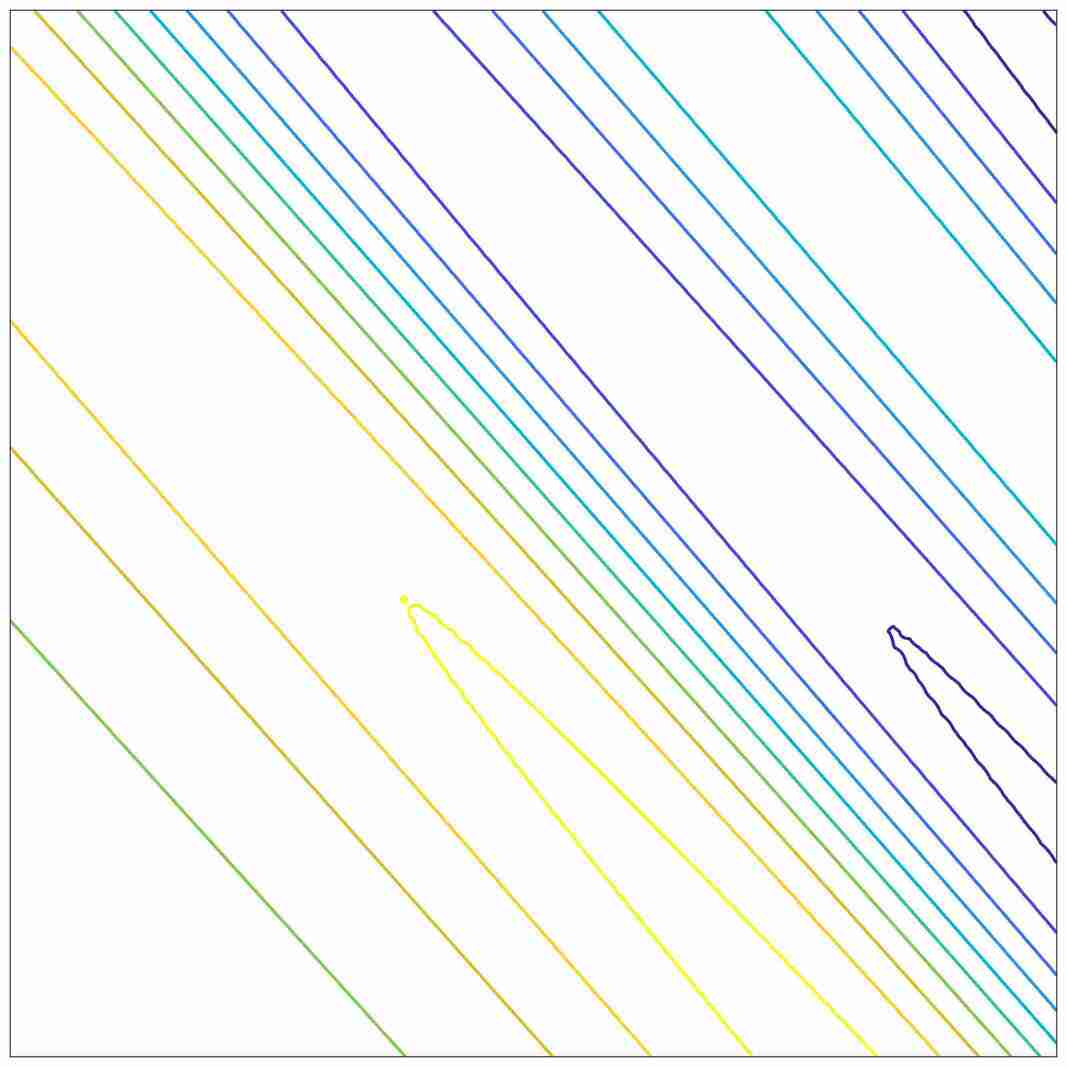}
\caption{$[\frac{21}{64}, \frac{22}{64}] \times [\frac{10}{64}, \frac{11}{64}]$}
\label{subfig:LWz4t4}
\end{subfigure}
\caption{Gradual $64\times$ zoom on the Laplacian of the vorticity at $t=4$.}
\label{fig:LWzoom_t4}
\end{figure}

\begin{figure}[h]
\centering
\begin{subfigure}{0.21\linewidth}
\includegraphics[width = \linewidth]{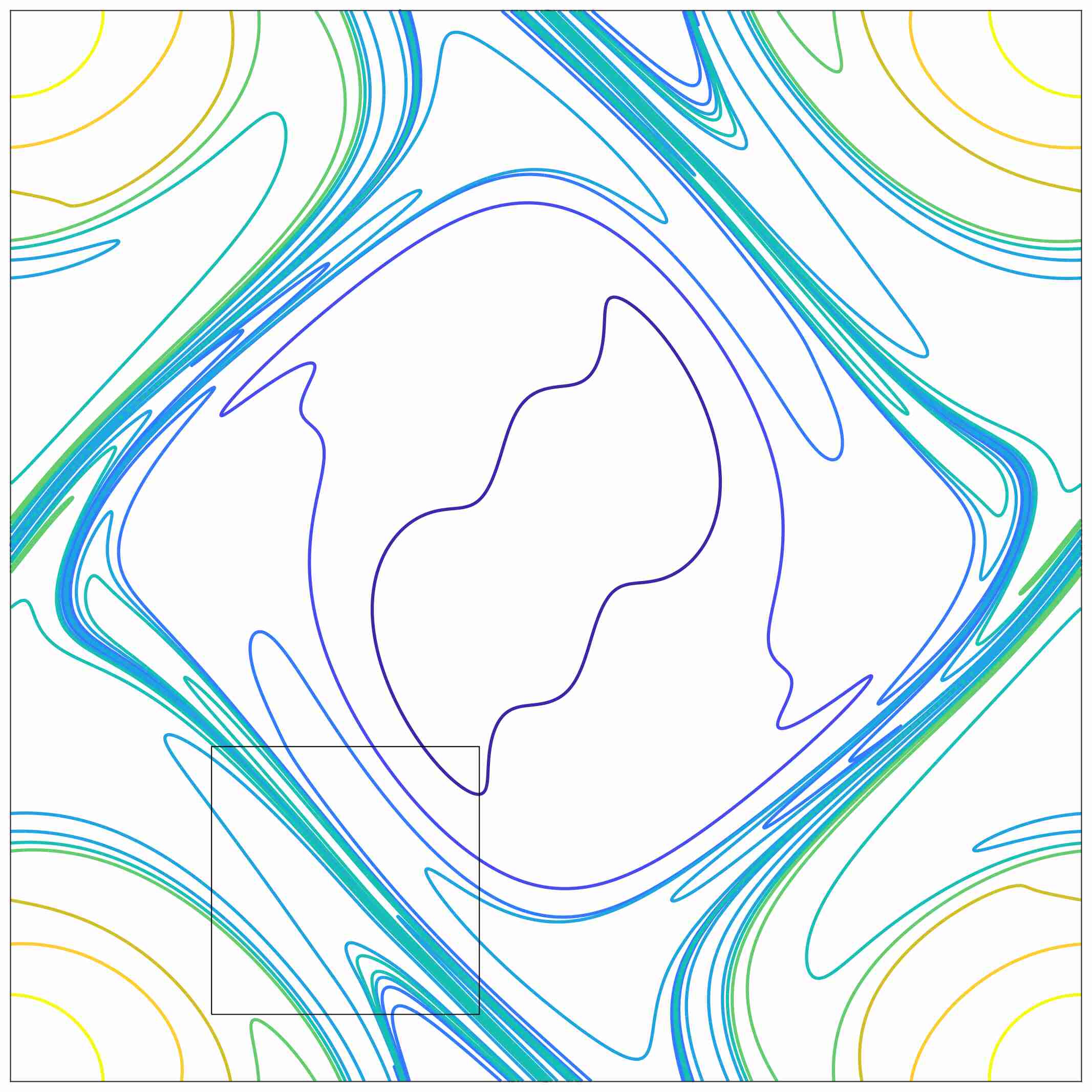}
\caption{$[0,1] \times [0, 1]$}
\label{subfig:Wz1t8}
\end{subfigure}
\begin{subfigure}{0.21\linewidth}
\includegraphics[width = \linewidth]{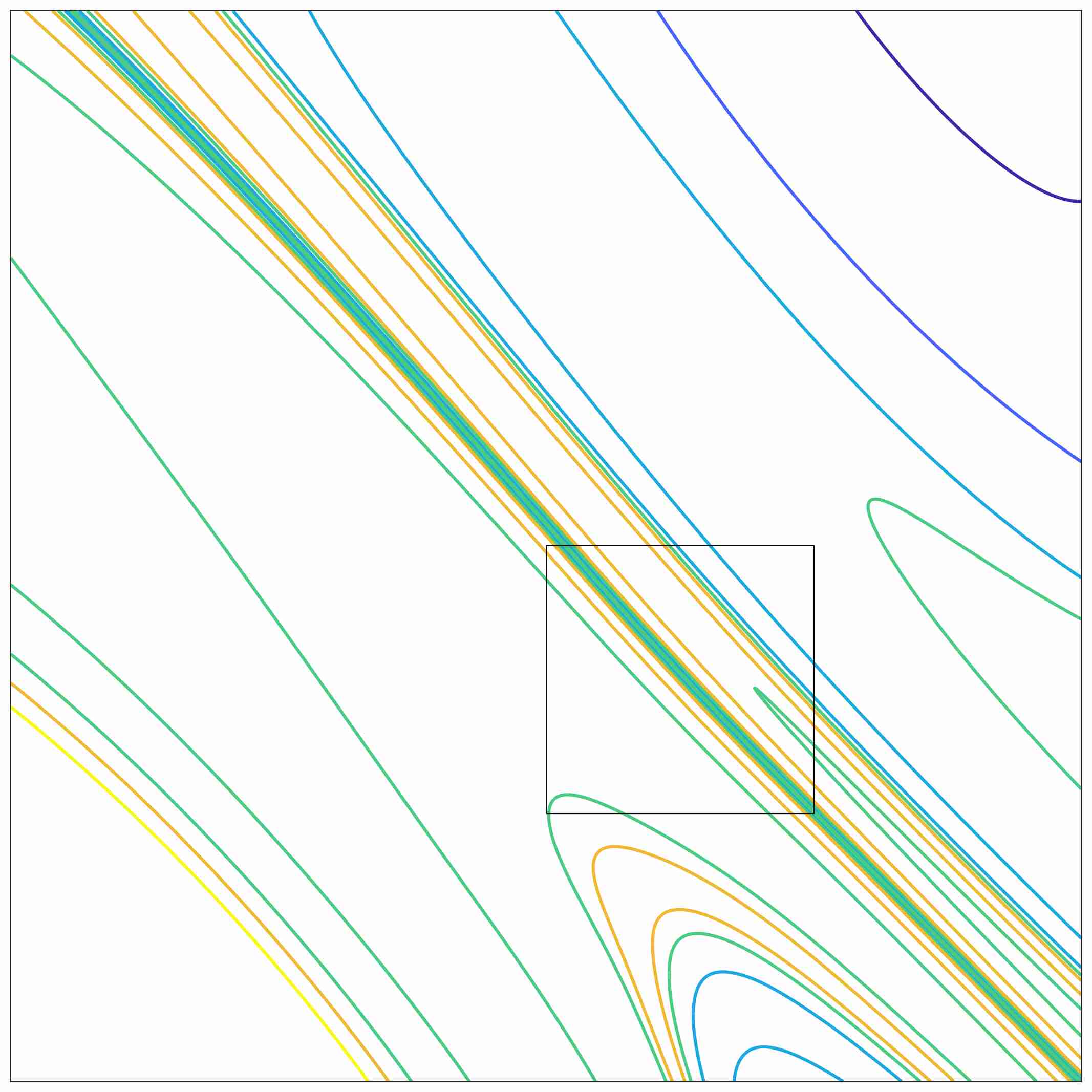}
\caption{$[\frac{3}{16},\frac{7}{16}] \times [\frac{1}{16}, \frac{5}{16}]$}
\label{subfig:Wz2t8}
\end{subfigure}
\begin{subfigure}{0.21\linewidth}
\includegraphics[width = \linewidth]{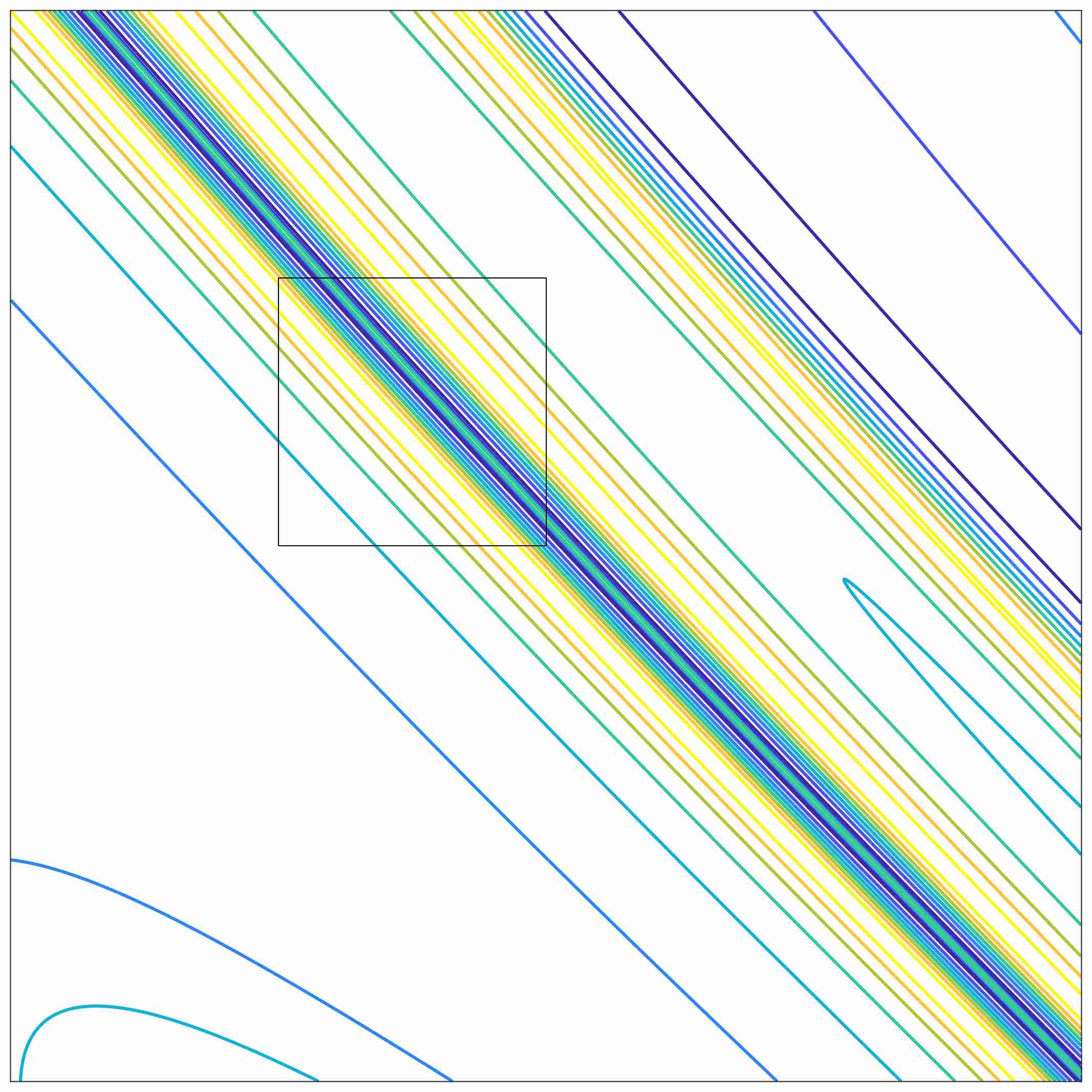}
\caption{$[\frac{5}{16}, \frac{6}{16}] \times [\frac{2}{16}, \frac{3}{16}]$}
\label{subfig:Wz3t8}
\end{subfigure}
\begin{subfigure}{0.21\linewidth}
 \includegraphics[width = \linewidth]{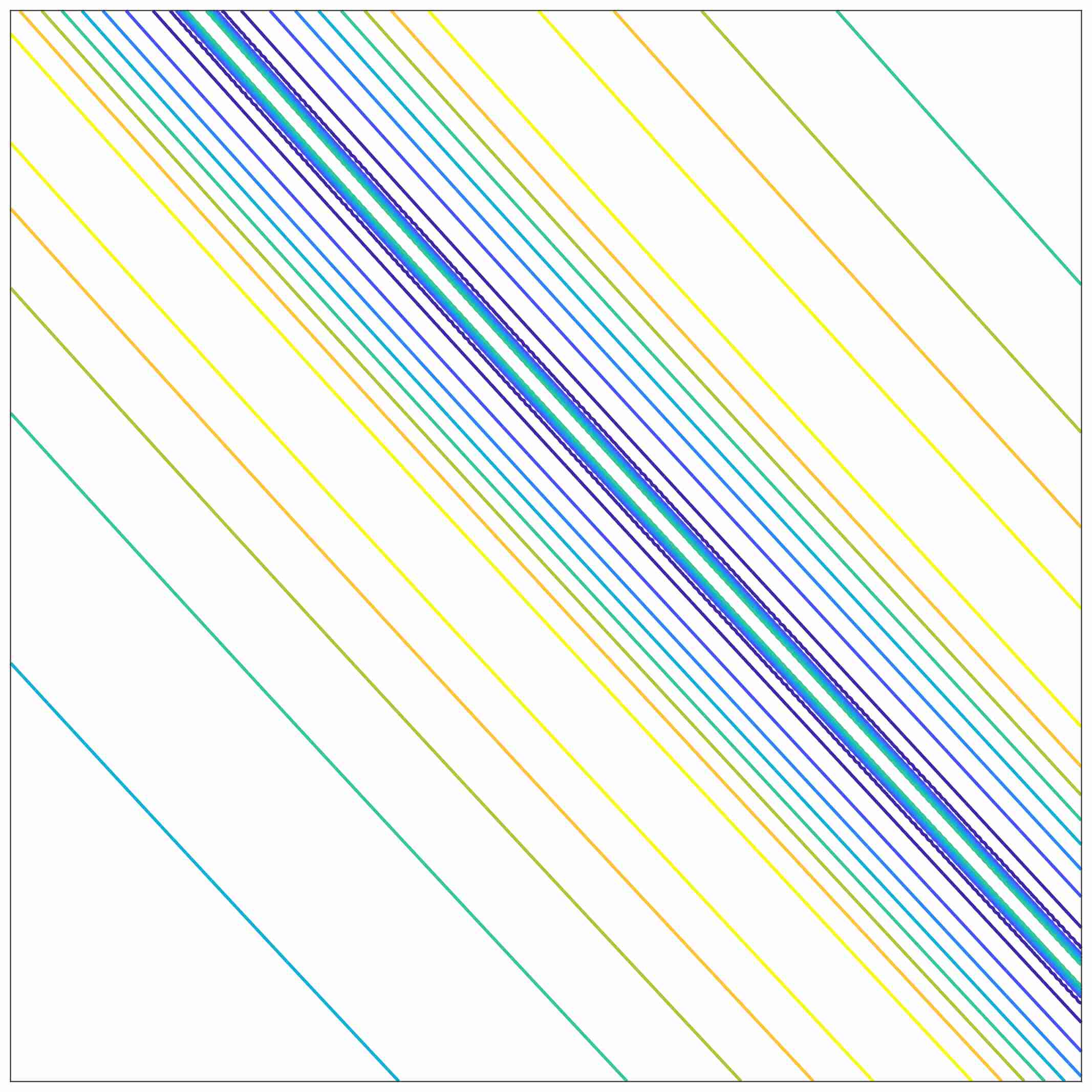}
\caption{$[\frac{21}{64}, \frac{22}{64}] \times [\frac{10}{64}, \frac{11}{64}]$}
\label{subfig:Wz4t8}
\end{subfigure}
\caption{Gradual $64\times$ zoom on the vorticity at $t=8$.}
\label{fig:Wzoom_t8}
\end{figure}

\begin{figure}[h]
\centering
\begin{subfigure}{0.21\linewidth}
\includegraphics[width = \linewidth]{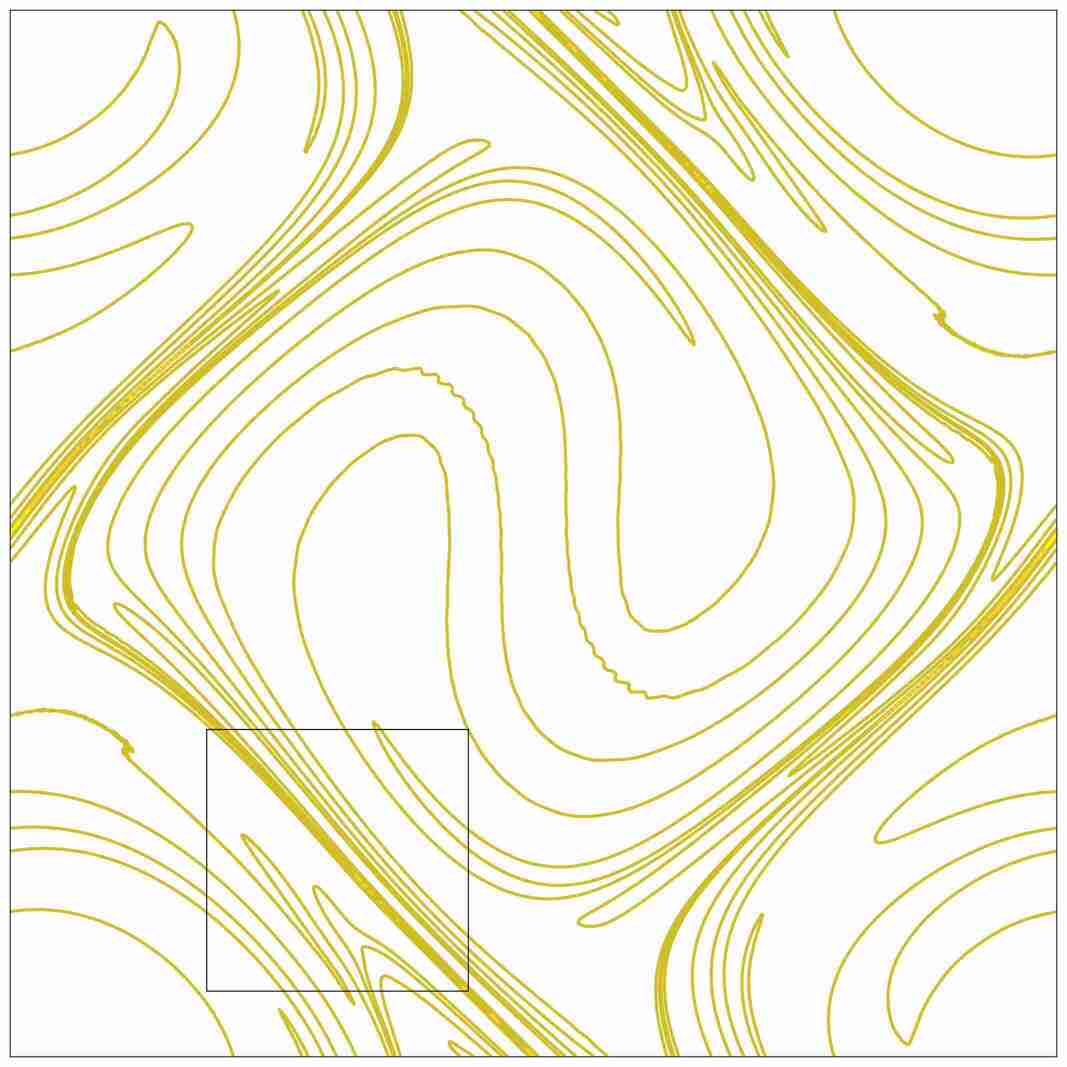}
\caption{$[0,1] \times [0, 1]$}
\label{subfig:LWz1t8}
\end{subfigure}
\begin{subfigure}{0.21\linewidth}
\includegraphics[width = \linewidth]{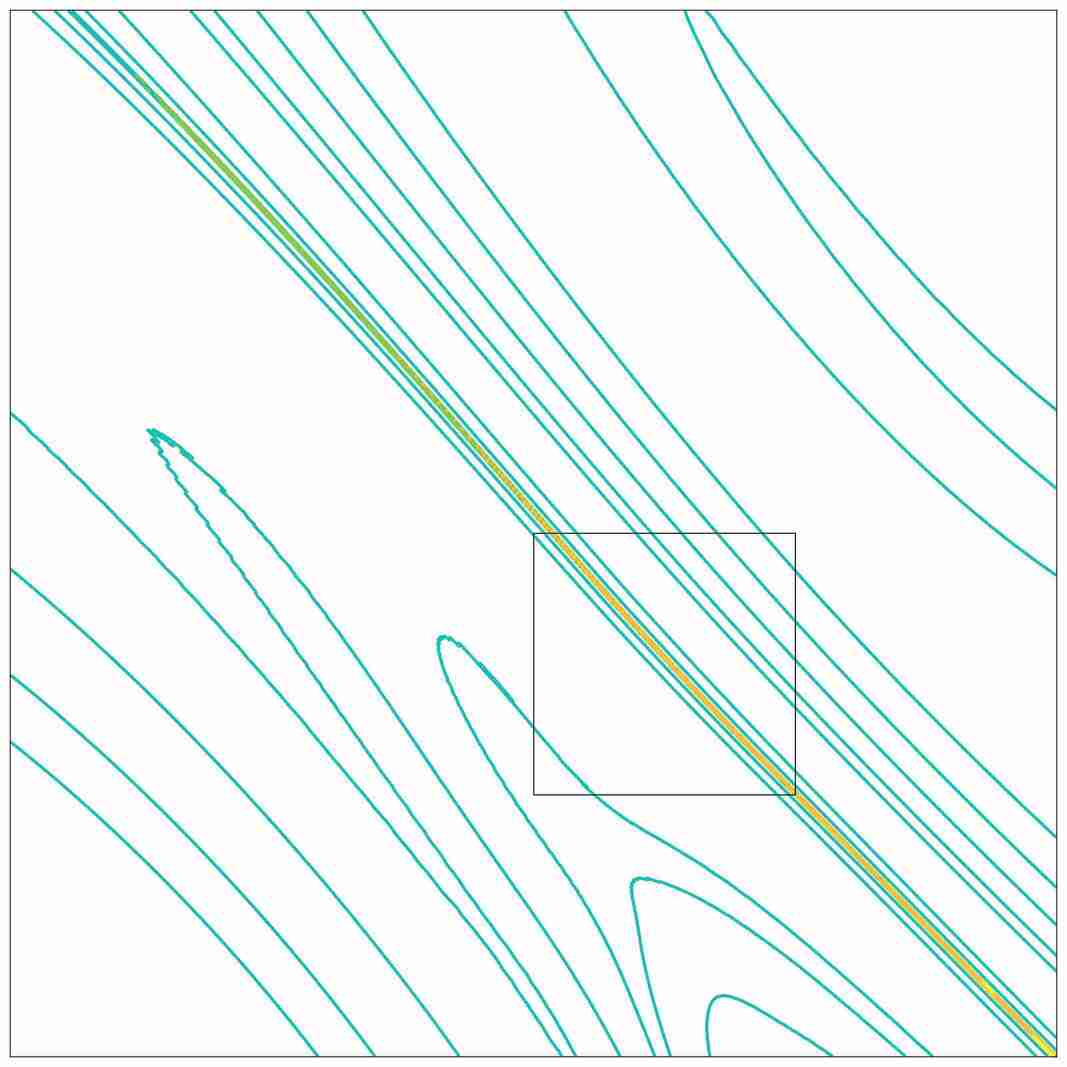}
\caption{$[\frac{3}{16},\frac{7}{16}] \times [\frac{1}{16}, \frac{5}{16}]$}
\label{subfig:LWz2t8}
\end{subfigure}
\begin{subfigure}{0.21\linewidth}
\includegraphics[width = \linewidth]{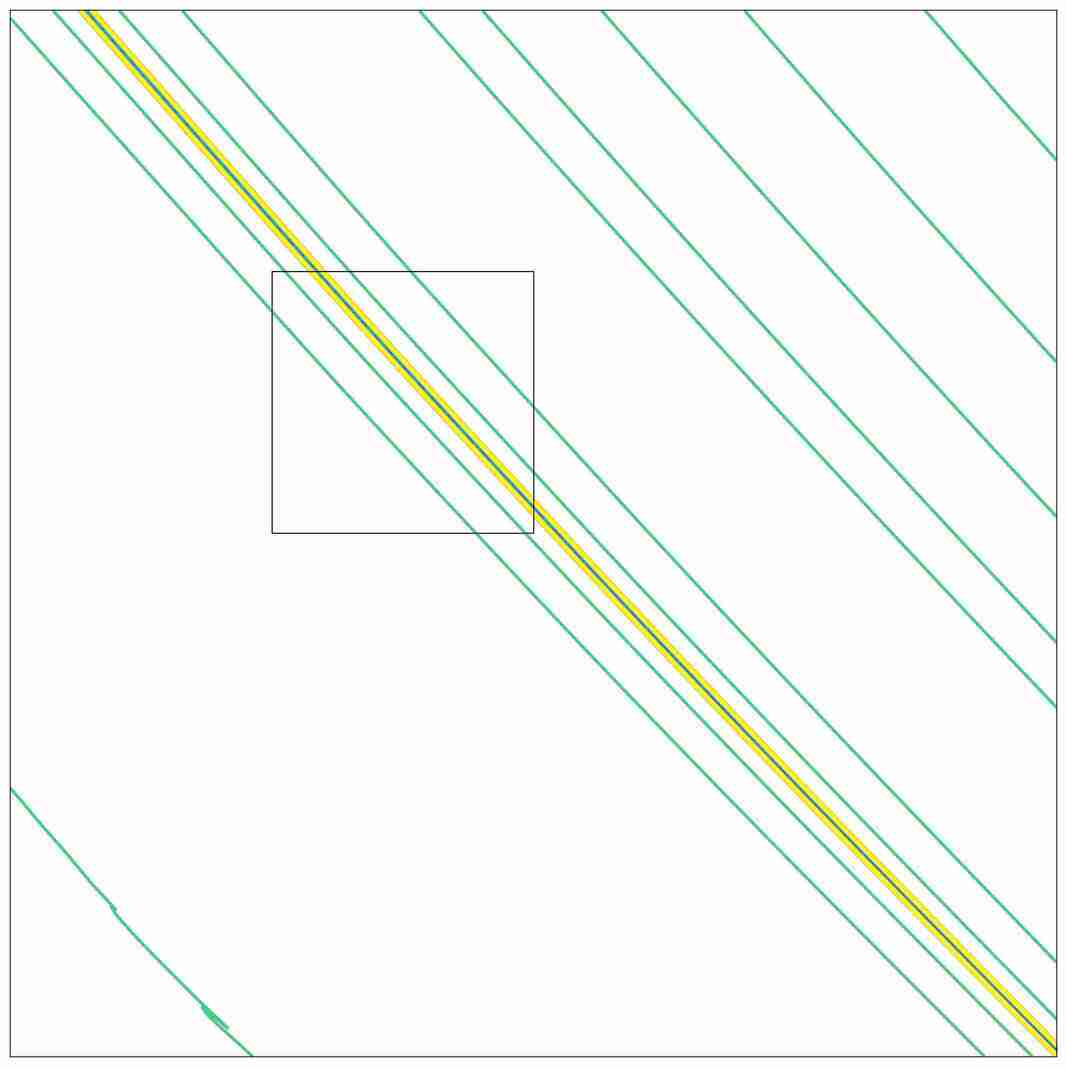}
\caption{$[\frac{5}{16}, \frac{6}{16}] \times [\frac{2}{16}, \frac{3}{16}]$}
\label{subfig:LWz3t8}
\end{subfigure}
\begin{subfigure}{0.21\linewidth}
 \includegraphics[width = \linewidth]{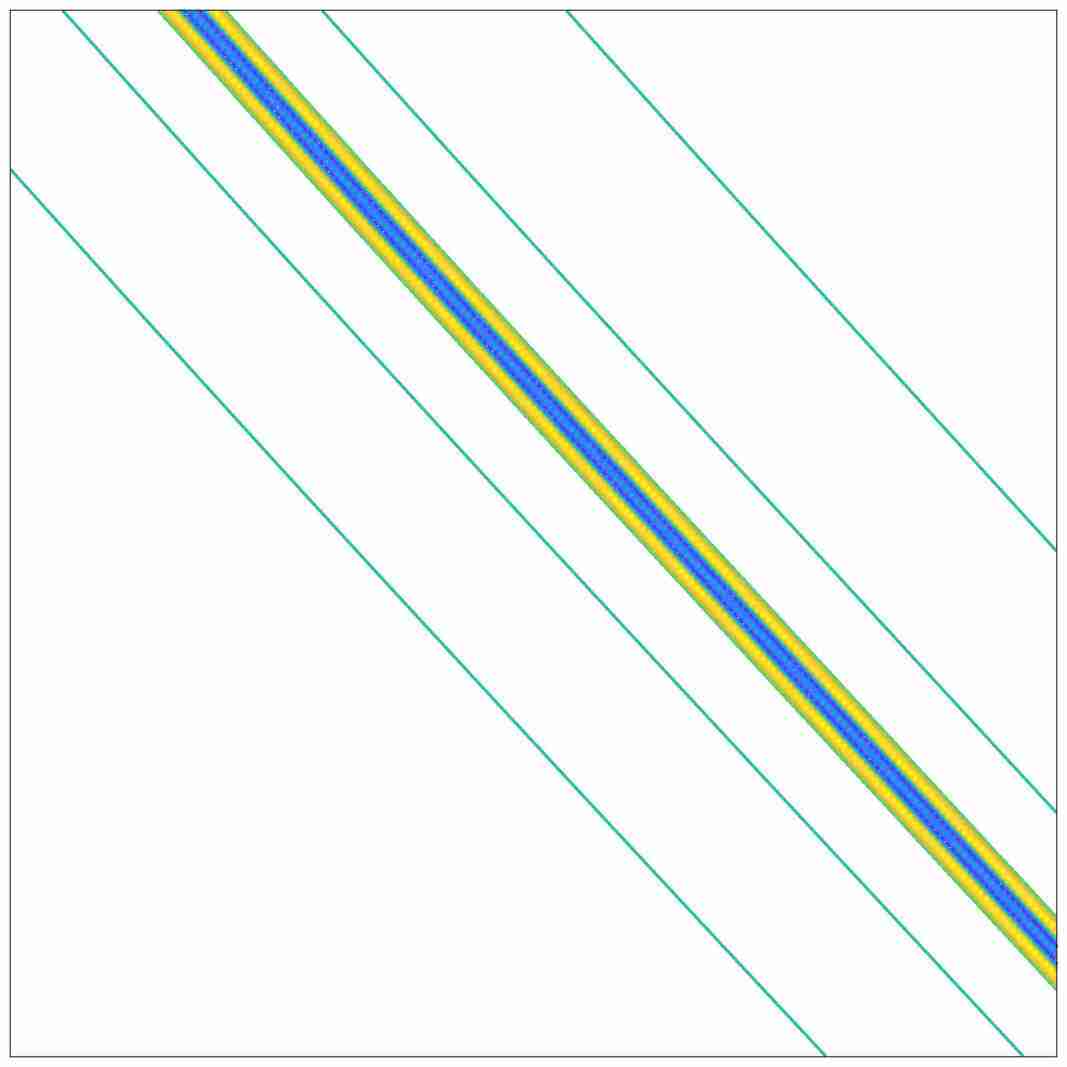}
\caption{$[\frac{21}{64}, \frac{22}{64}] \times [\frac{10}{64}, \frac{11}{64}]$}
\label{subfig:LWz4t8}
\end{subfigure}
\caption{Gradual $64\times$ zoom on the Laplacian of the vorticity at $t=8$.}
\label{fig:LWzoom_t8}
\end{figure}

Figures \ref{fig:Wzoom_t4}, \ref{fig:LWzoom_t4}, \ref{fig:Wzoom_t8} and \ref{fig:LWzoom_t8} illustrate the arbitrary spatial resolution provided by the characteristic map. Each zoomed plot is sampled with the same number of sample points, providing the same image resolution on gradually smaller subdomains. We observe that when we zoom in, we recover additional small scale features not seen on the original plot. The undersampling from the $[0, 1] \times [0, 1]$ plot fails to represent the complexity of the contour lines inside the thin vortex sheet (see figures \ref{subfig:LWz1t8} v.s. \ref{subfig:LWz1t8}). These features are recovered when using a finer sampling. This in fact shows that the solution provided by the CM method is not bound to a fixed set of sample points. Whereas in most methods, once a grid is chosen, any detail finer than this grid is lost, the CM method does not compute the solution on a fixed grid, rather it provides an algorithm to sample the vorticity field $\omega^n = \omega_0 \circ \vhX^n$ defined as a function over the whole domain. This means that the solution can be evaluated anywhere, providing arbitrary spatial resolution. In practice, in case $\omega_0$ is given numerically, this implies that we maintain the same resolution as that of $\omega_0$ throughout the entire simulation: there is no loss of spatial features.

\subsection{Random initial conditions}

In this section, we perform the same tests as in section \ref{subsec:num4modes} on a randomly generated initial condition. The procedure to generate the random initial condition is given in \cite{ray2011resonance}. In short, the vorticity is defined in Fourier space, which is divided into lattice shells, each containing all modes $\bm{k}$ such that $| \bm{k} | \in [K, K+1)$. The $K^{\text{th}}$ shell contains $N(K)$ modes, and for each of these modes, we assign a vorticity Fourier coefficient $\hat{\omega}_{\bm{k}}$ of fixed modulus $2 K^{7/2} \exp (-K^2/4)/N(K)$ and a phase picked randomly from $[0, 2 \pi)$ with uniform distribution. This guarantees that the total vorticity in the $K^{\text{th}}$ shell decays like $2 K^{7/2} \exp (-K^2/4)$. Further, to ensure that the vorticity is real, opposite wave vectors (i.e. $\hat{\omega}_{\bm{k}}$ and $\hat{\omega}_{-\bm{k}}$) are given opposite phases so that the resulting Fourier expansion is Hermitian. The same test was performed in \cite{frisch} until $t=1$ using the Cauchy-Lagrange methods with $2048^2$ spatial Fourier modes.

\begin{figure}[h]
\centering
\begin{subfigure}{0.21\linewidth}
\centering
\includegraphics[width = \linewidth]{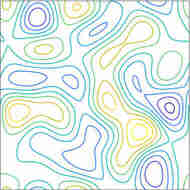}
\caption{$\omega_0$}
\label{subfig:rw0}
\end{subfigure}
\begin{subfigure}{0.21\linewidth}
\centering
\includegraphics[width = \linewidth]{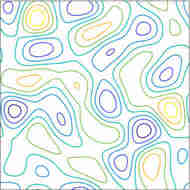}
\caption{$\Laplace \omega_0$}
\label{subfig:Lrw0}
\end{subfigure}
\caption{Contour plot of the random initial vorticity and its Laplacian.}
\label{fig:initVortRand}
\end{figure}

Since the CM method does not work directly with the Fourier transform of the initial vorticity, we defined our initial $\omega_0$ as follows: we first generate $\hat{\omega}_{\bm{k}}$ as described above, we then sample the Fourier series on a $512^2$ grid to obtain a Hermite cubic interpolant which we use as $\omega_0$. We used total of 32 lattice shells, i.e. $K = 0, 1, \ldots 32$; in fact, due to the prescribed decay rate of the coefficients, $| \hat{\omega}_{\bm{k}} |$ is already well below machine precision for $| \bm{k} | = 32$ and is below machine underflow for 64. Using more shells would have no consequence on our $\omega_0$. The initial vorticity and its Laplacian are shown in figure \ref{fig:initVortRand}. We ran the CM method on this initial condition using a $256^2$ grid for the map, $1024^2$ grid to represent $\psi^n$, $\incr{t} = 1/64$ and $\delta_{\det} = 10^{-4}$. Contour plots of the vorticity field and of the its Laplacian are shown at 0.5 time intervals in figures \ref{fig:vortContours_rand} and \ref{fig:LVortContour_rand}. The number of remaps and CPU times required are shown in table \ref{tab:rmrt_rand}.
\begin{table}[h] \footnotesize
\begin{center}
{\renewcommand{\arraystretch}{1.5}
\begin{tabular}{c | c c c c}
\hline
\hline
$t$ & 0.5 & 1 & 1.5 & 2 \\
\hline
Number of remaps & 3 & 18 & 38 & 69  \\
Total CPU time & 233 s & 499 s & 822 s & 1226 s  \\
\hline
\end{tabular}} \\
\end{center}
\caption{Number of remaps and CPU times for the random initial condition test using the CM method.}
\label{tab:rmrt_rand}
\end{table}

\begin{figure}[h]
\captionsetup[subfigure]{justification=centering}
\centering
\begin{subfigure}{0.21\linewidth}
\centering
\includegraphics[width = \linewidth]{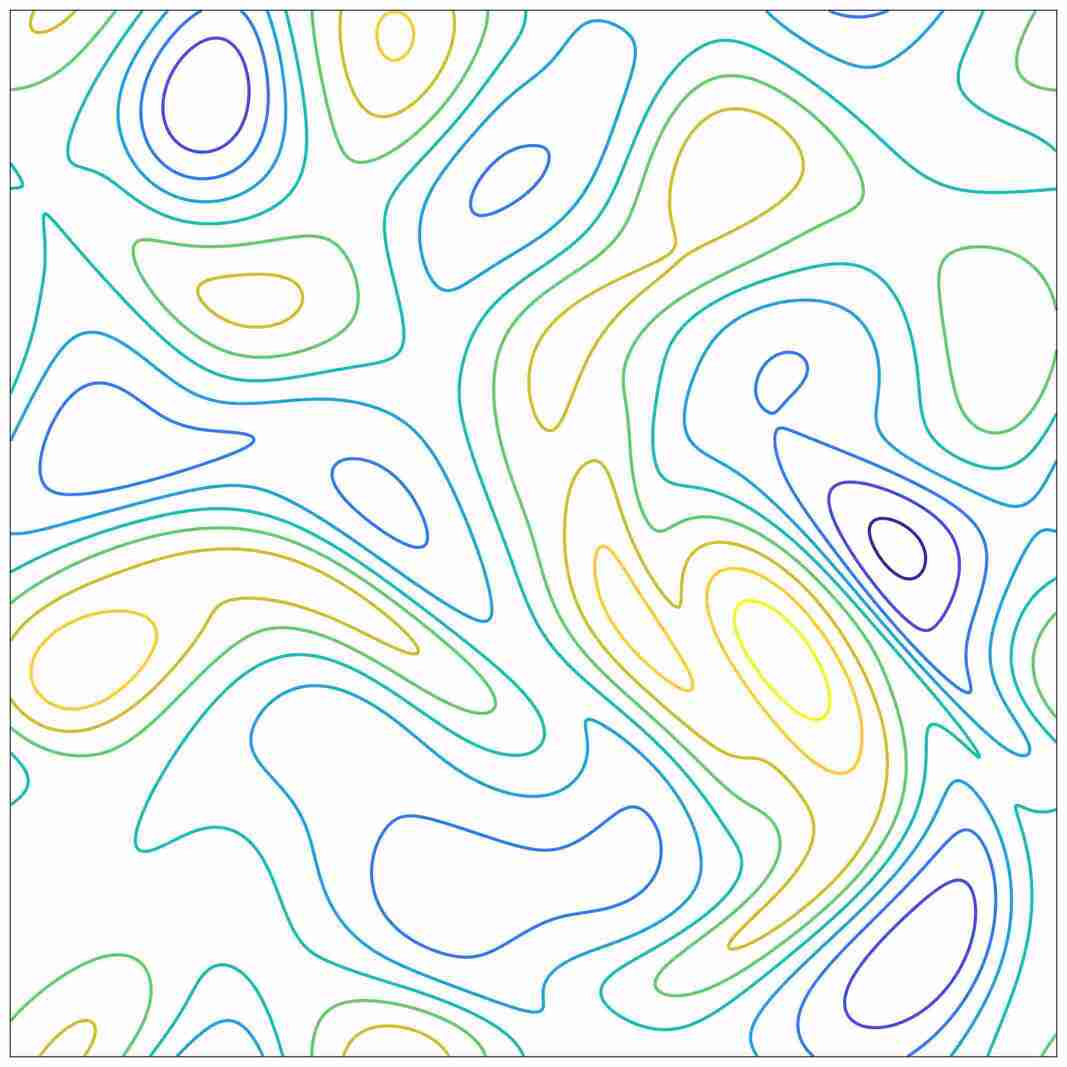}
\caption{$t=0.5$}
\label{subfig:Wr5}
\end{subfigure}
\begin{subfigure}{0.21\linewidth}
\centering
\includegraphics[width = \linewidth]{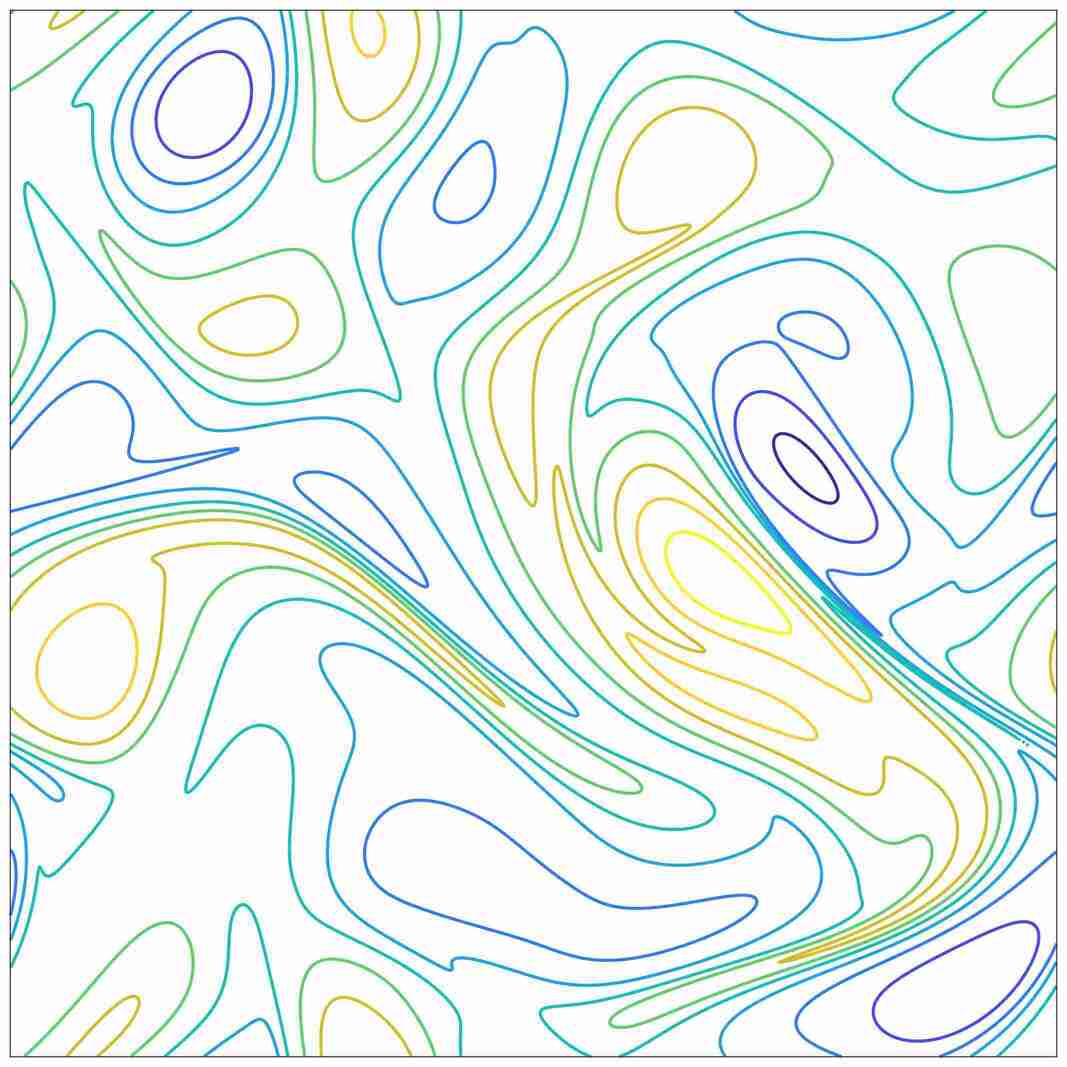}
\caption{$t=1$}
\label{subfig:Wr10}
\end{subfigure}
\begin{subfigure}{0.21\linewidth}
\centering
\includegraphics[width = \linewidth]{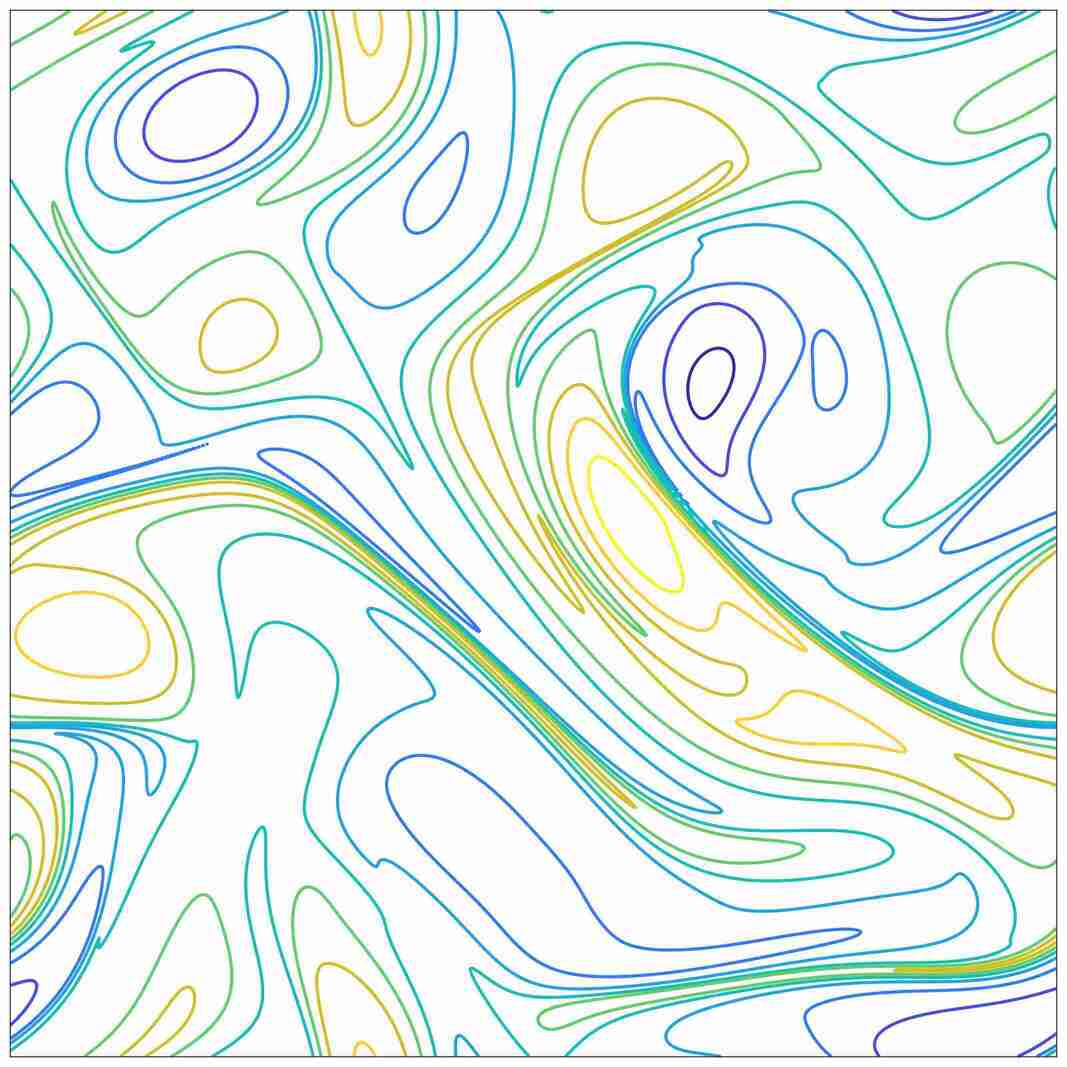}
\caption{$t=1.5$}
\label{subfig:Wr15}
\end{subfigure}
\begin{subfigure}{0.21\linewidth}
\centering
\includegraphics[width = \linewidth]{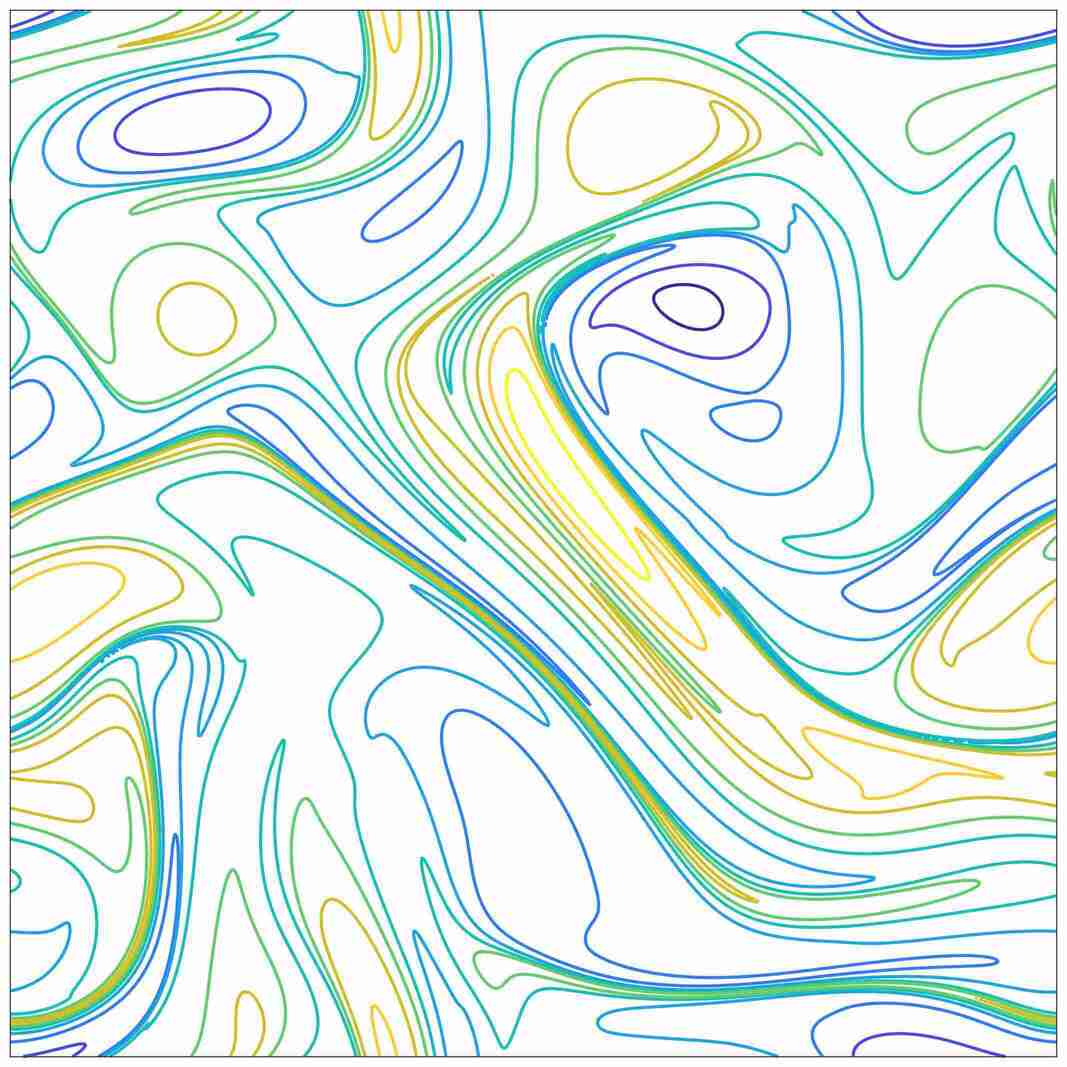}
\caption{$t=2$}
\label{subfig:Wr20}
\end{subfigure}
\caption{Contour plot of $\omega^n$ for the random initial condition test using $256^2$ grid for $\vhX^n$, $1024^2$ grid for representing $\psi^n$, $\incr{t} = 1/128$ and $\delta_{\det} = 10^{-4}$.}
\label{fig:vortContours_rand}
\end{figure}

\begin{figure}[h]
\captionsetup[subfigure]{justification=centering}
\centering
\begin{subfigure}{0.21\linewidth}
\centering
\includegraphics[width = \linewidth]{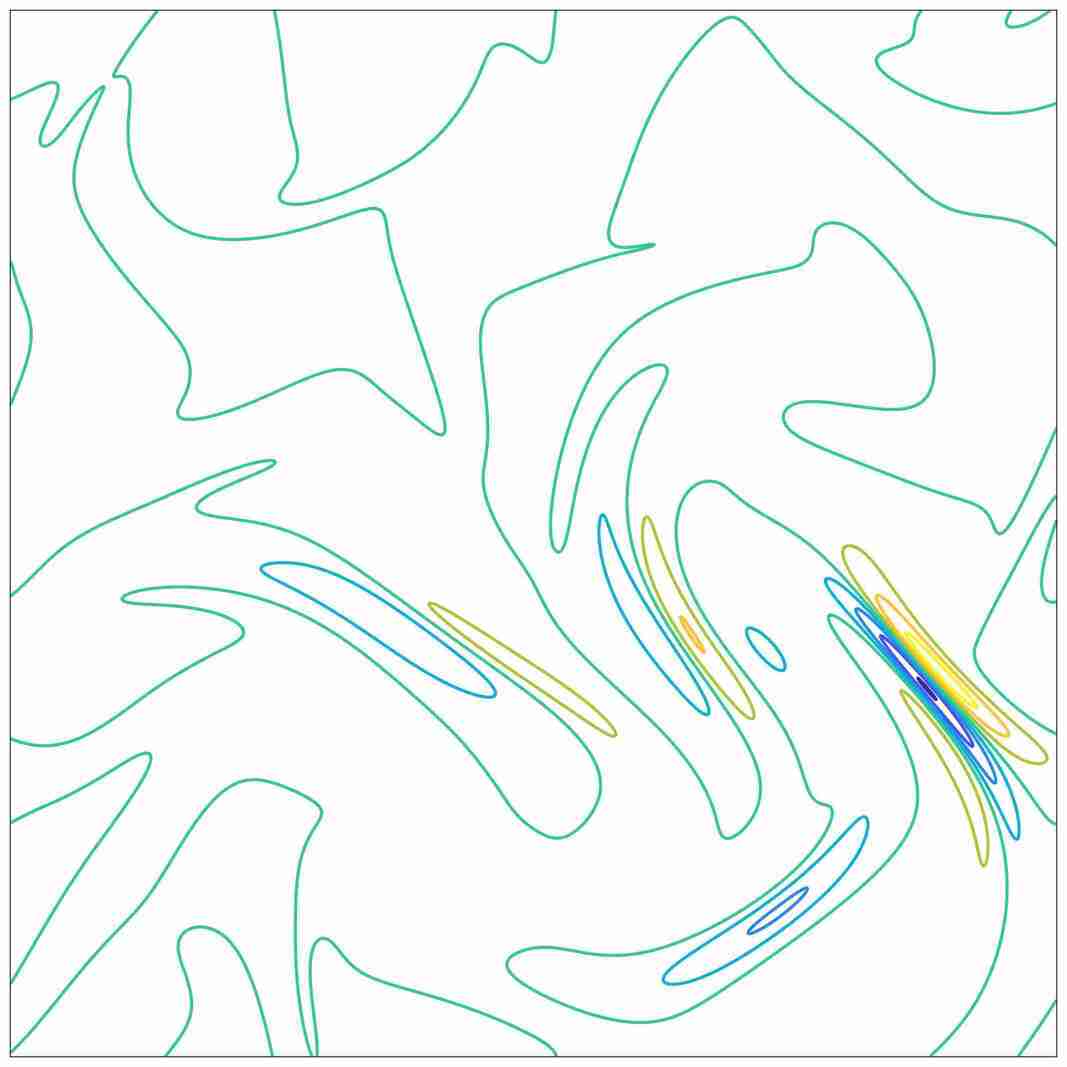}
\caption{$t=0.5$}
\label{subfig:LWr5}
\end{subfigure}
\begin{subfigure}{0.21\linewidth}
\centering
\includegraphics[width = \linewidth]{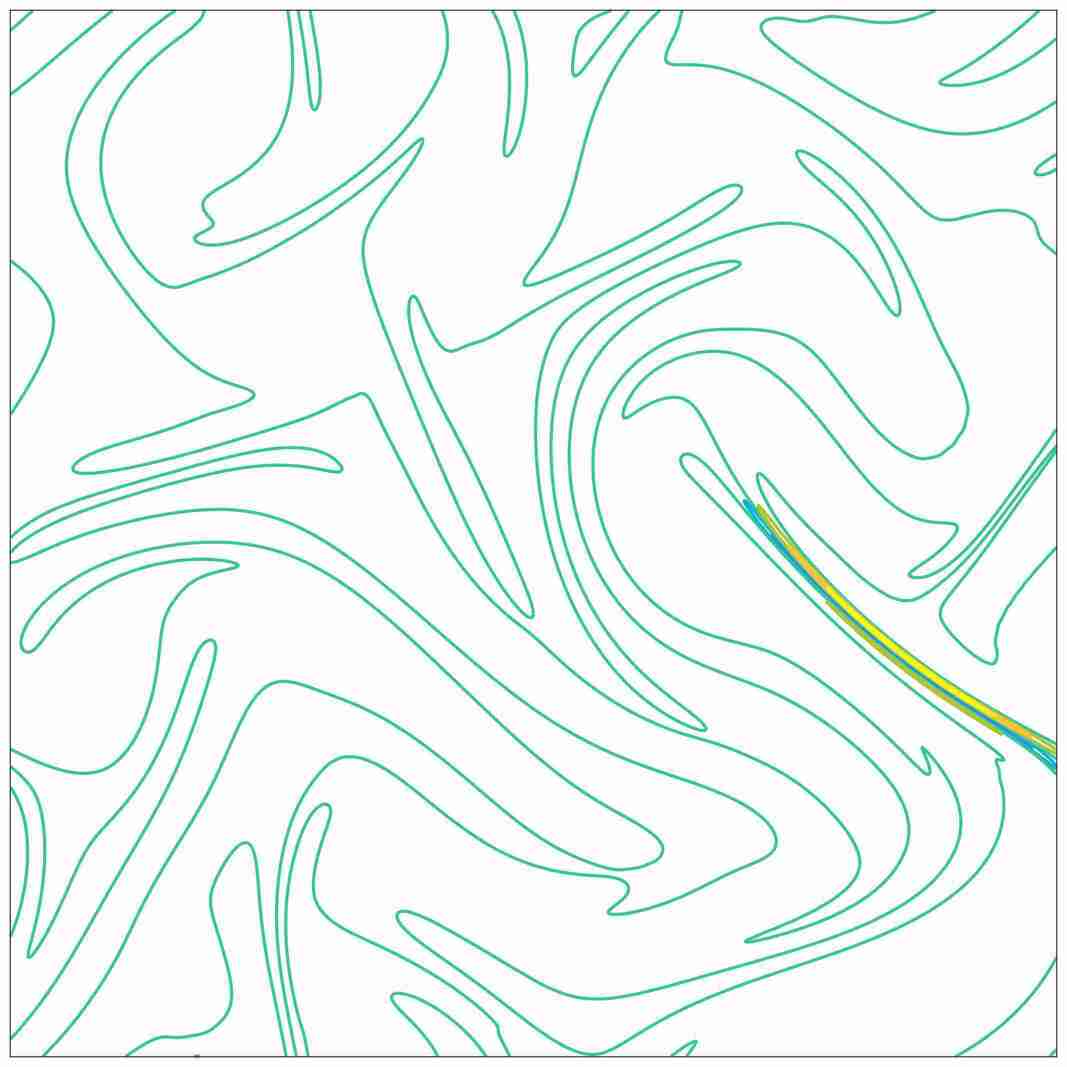}
\caption{$t=1$}
\label{subfig:LWr10}
\end{subfigure}
\begin{subfigure}{0.21\linewidth}
\centering
\includegraphics[width = \linewidth]{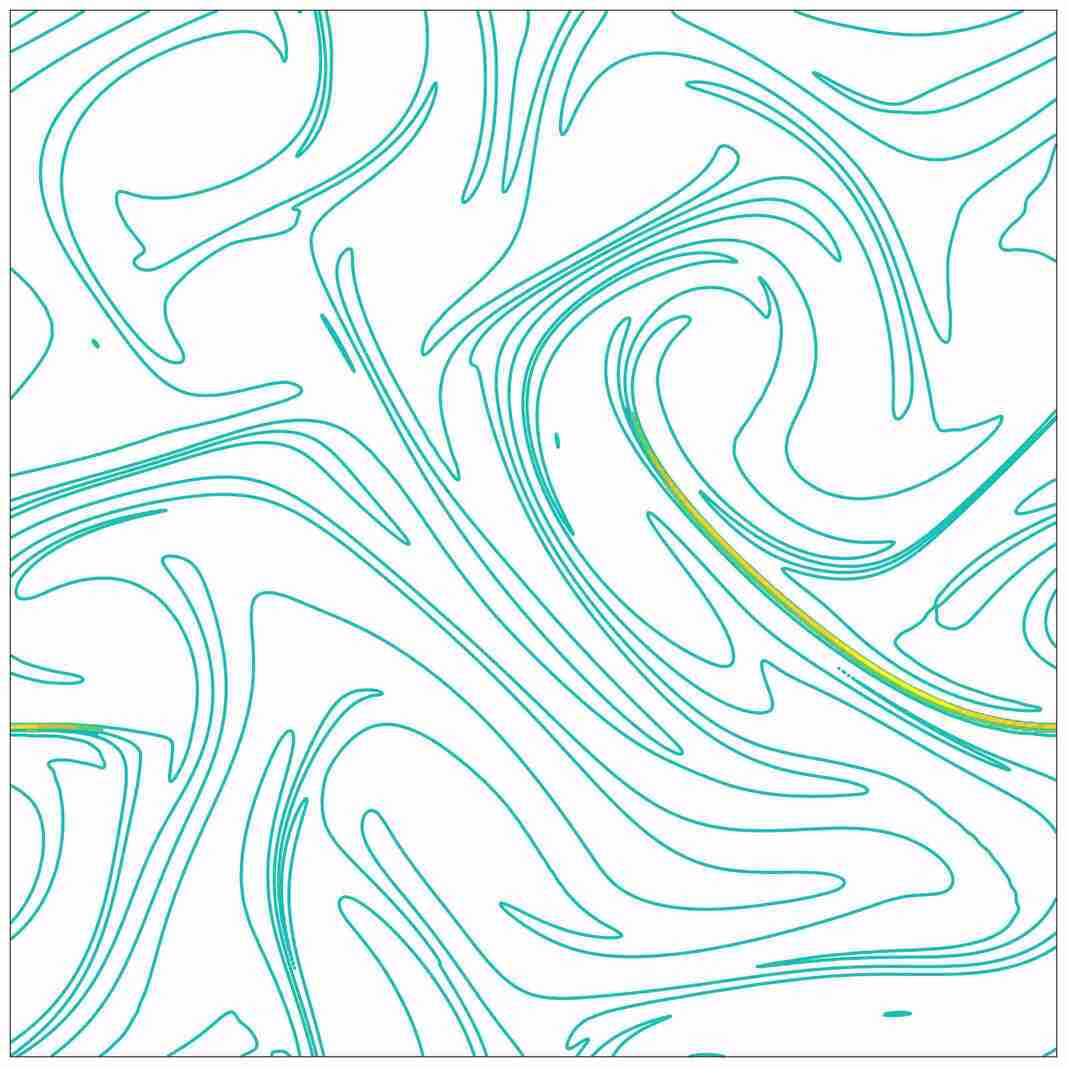}
\caption{$t=1.5$}
\label{subfig:LWr15}
\end{subfigure}
\begin{subfigure}{0.21\linewidth}
\centering
\includegraphics[width = \linewidth]{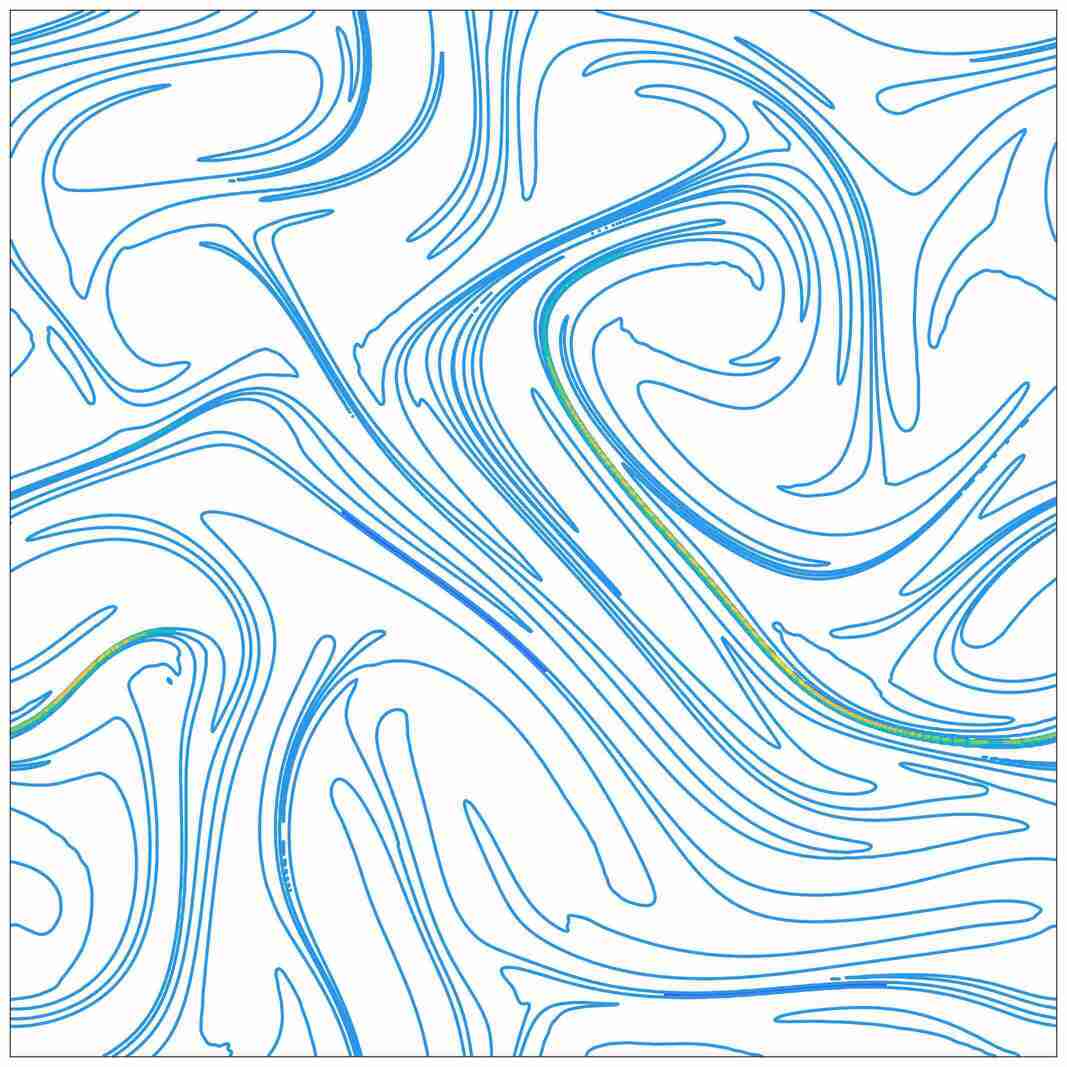}
\caption{$t=2$}
\label{subfig:LWr20}
\end{subfigure}
\caption{Contour plot of $\Laplace \omega^n$ for the random initial condition test using $256^2$ grid for $\vhX^n$, $1024^2$ grid for representing $\psi^n$, $\incr{t} = 1/128$ and $\delta_{\det} = 10^{-4}$.}
\label{fig:LVortContour_rand}
\end{figure}

\begin{figure}[h]
\centering
\begin{subfigure}{0.21\linewidth}
\includegraphics[width = \linewidth]{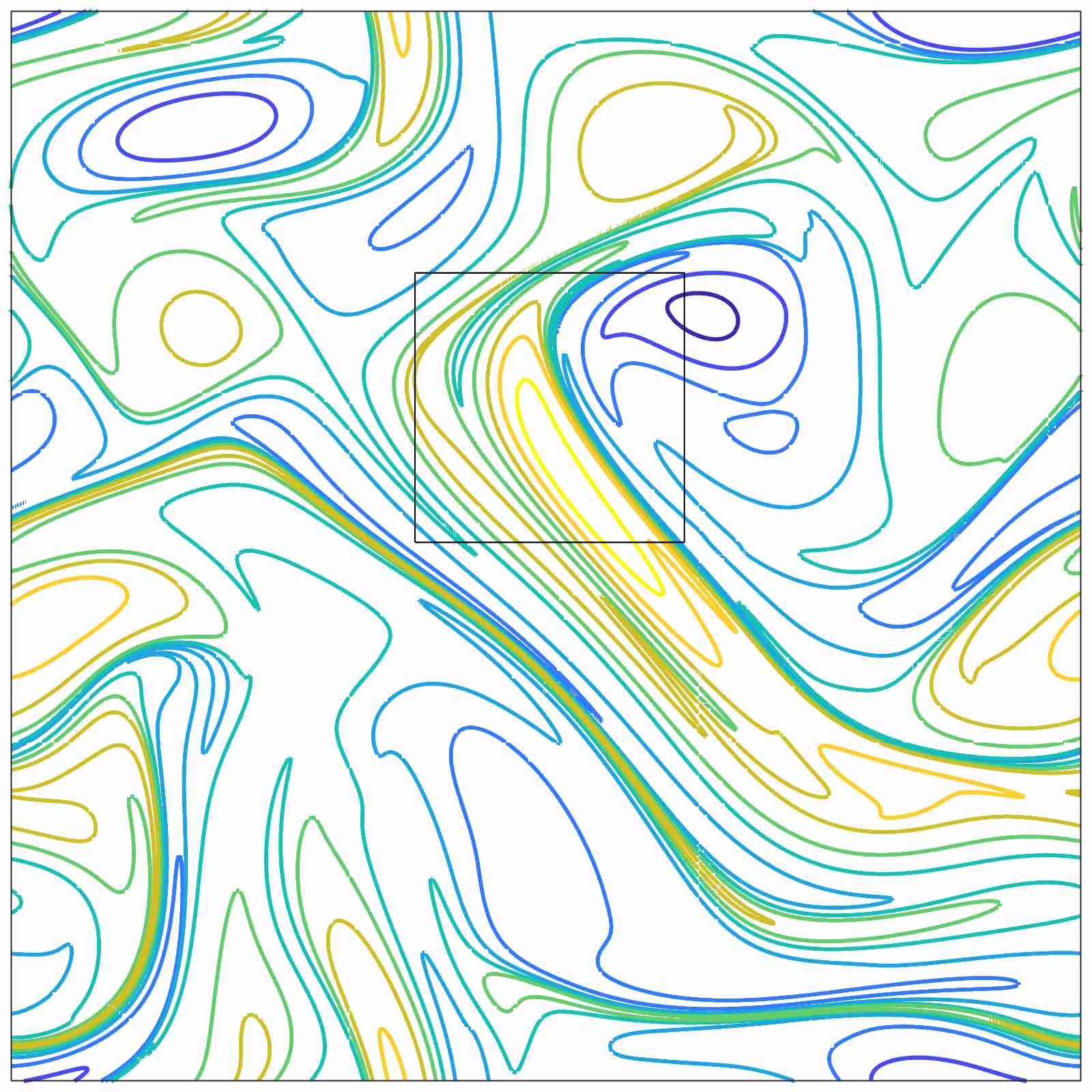}
\caption{$[0,1] \times [0, 1]$}
\label{subfig:Wrz1t2}
\end{subfigure}
\begin{subfigure}{0.21\linewidth}
\includegraphics[width = \linewidth]{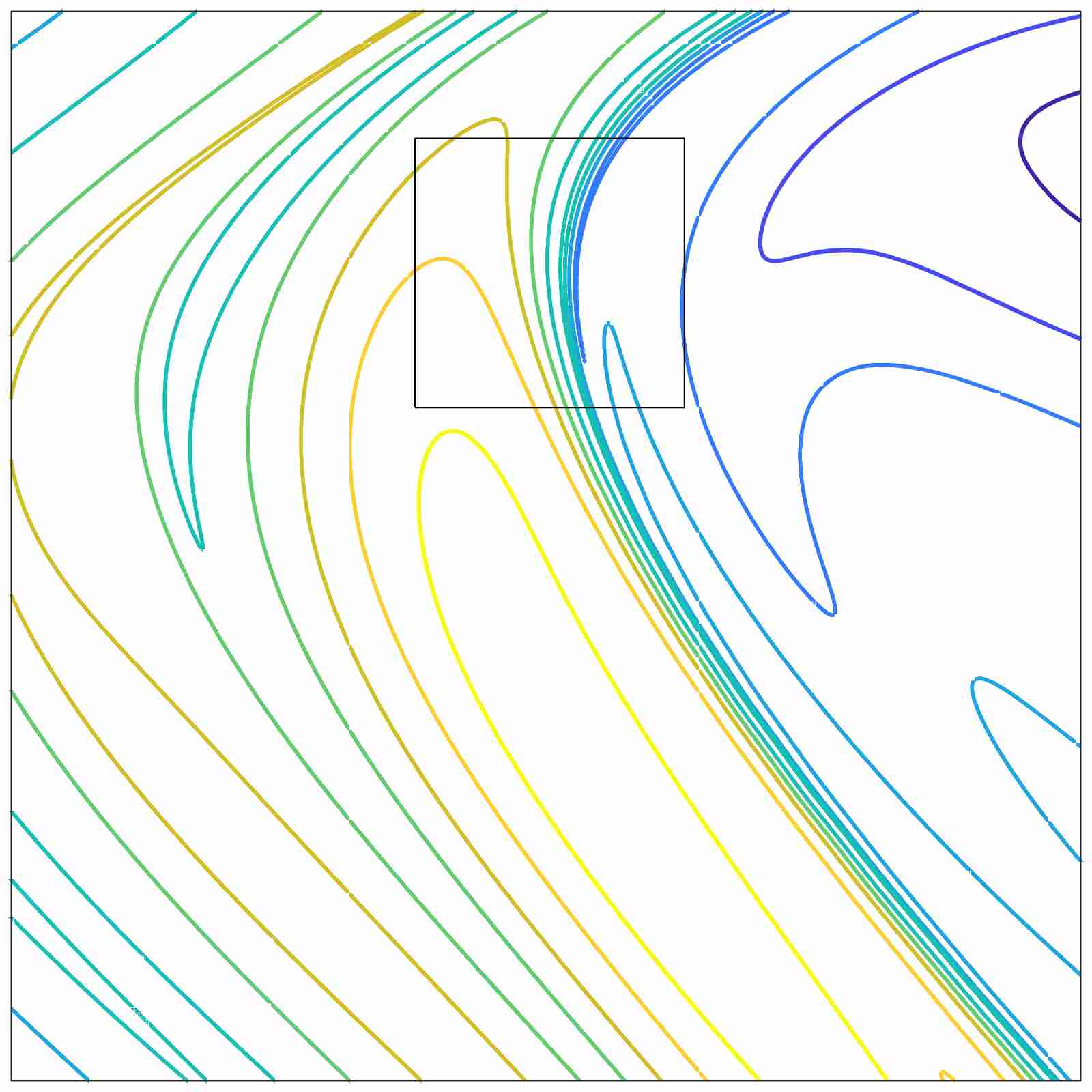}
\caption{$[\frac{3}{8},\frac{5}{8}] \times [\frac{1}{2}, \frac{3}{4}]$}
\label{subfig:Wrz2t2}
\end{subfigure}
\begin{subfigure}{0.21\linewidth}
\includegraphics[width = \linewidth]{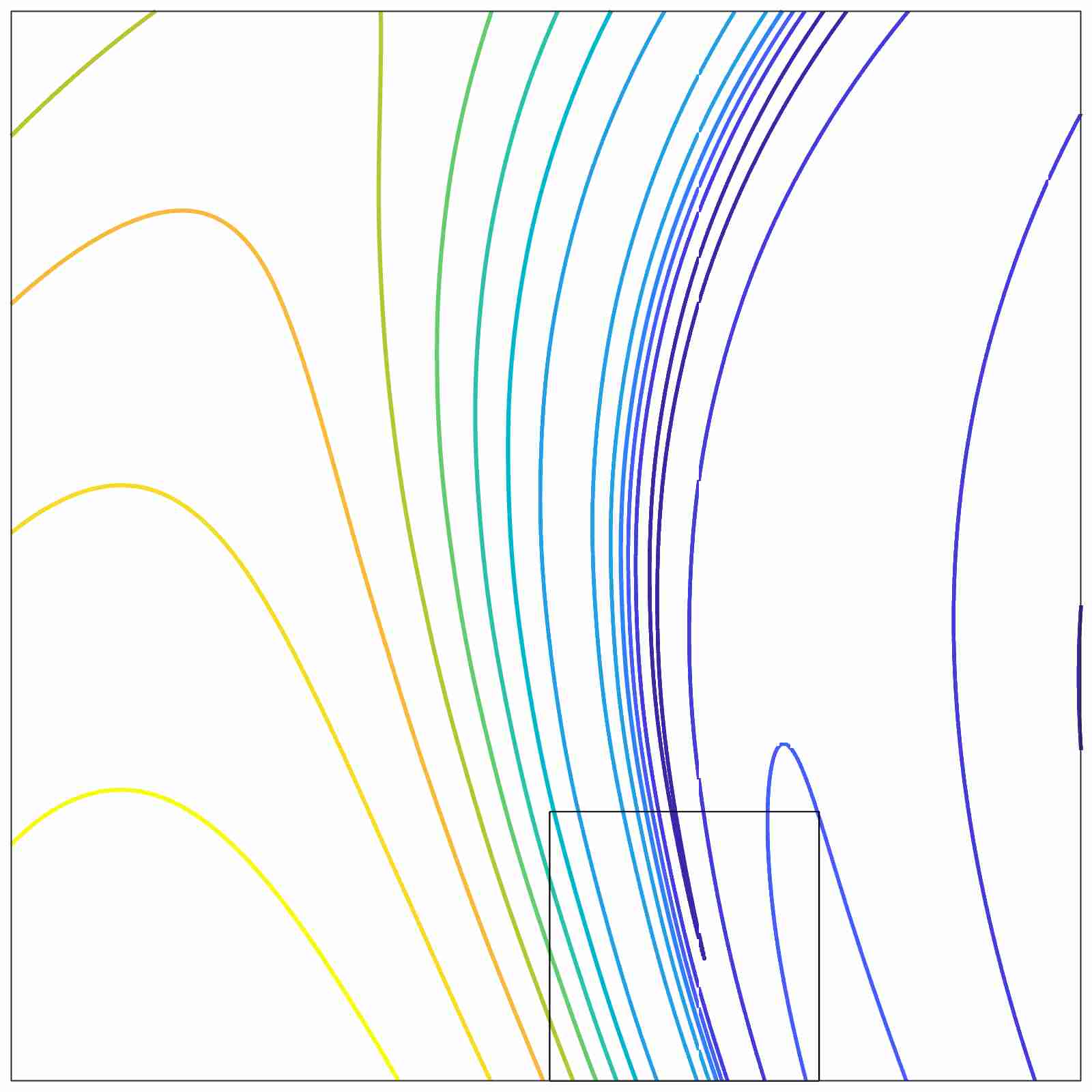}
\caption{$[\frac{15}{32}, \frac{17}{32}] \times [\frac{21}{32}, \frac{23}{32}]$}
\label{subfig:Wrz3t2}
\end{subfigure}
\begin{subfigure}{0.21\linewidth}
 \includegraphics[width = \linewidth]{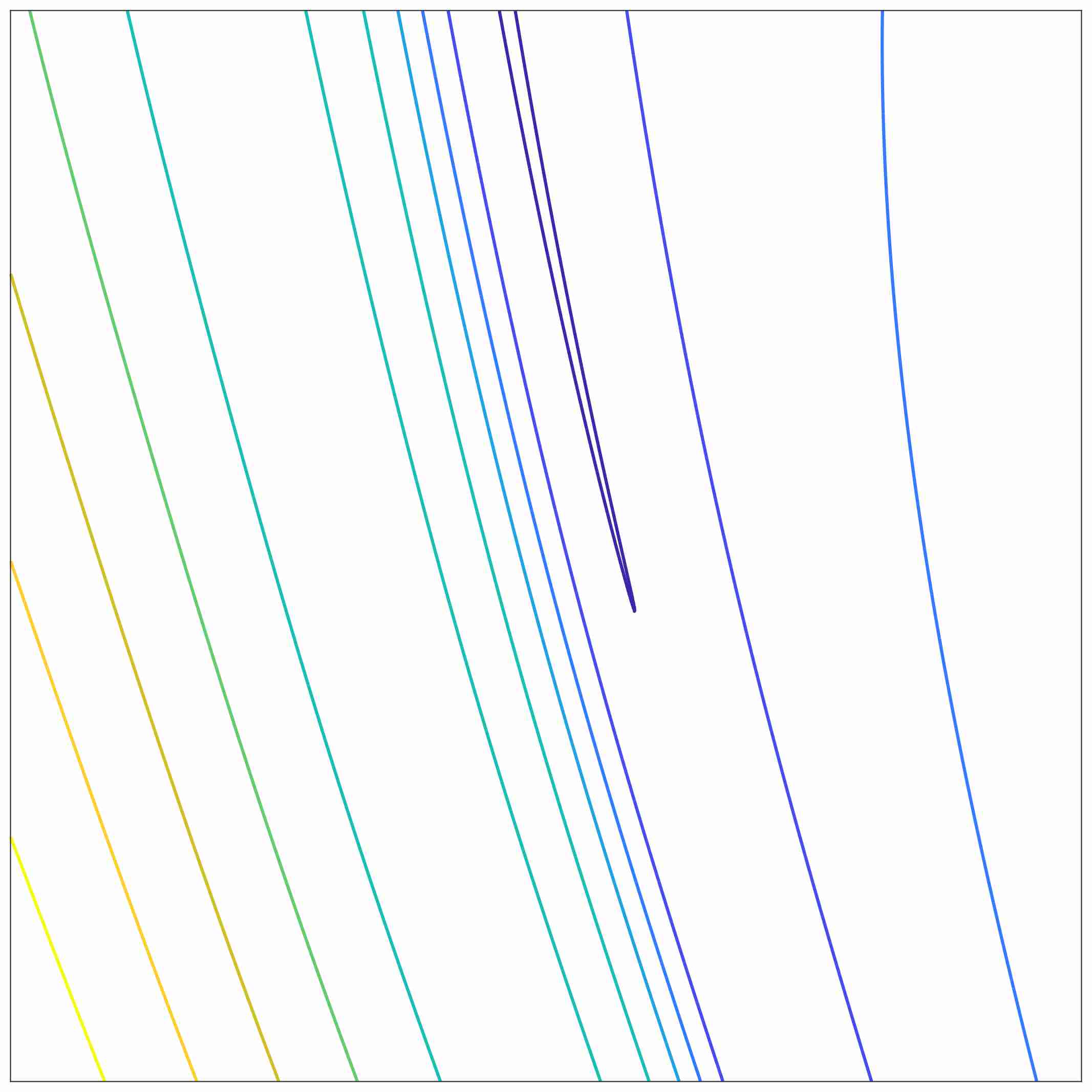}
\caption{$[\frac{1}{2}, \frac{33}{64}] \times [\frac{21}{32}, \frac{43}{64}]$}
\label{subfig:Wrz4t2}
\end{subfigure}
\caption{Gradual $64\times$ zoom on the vorticity at $t=2$.}
\label{fig:Wrzoom_t2}
\end{figure}

\begin{figure}[h]
\centering
\begin{subfigure}{0.21\linewidth}
\includegraphics[width = \linewidth]{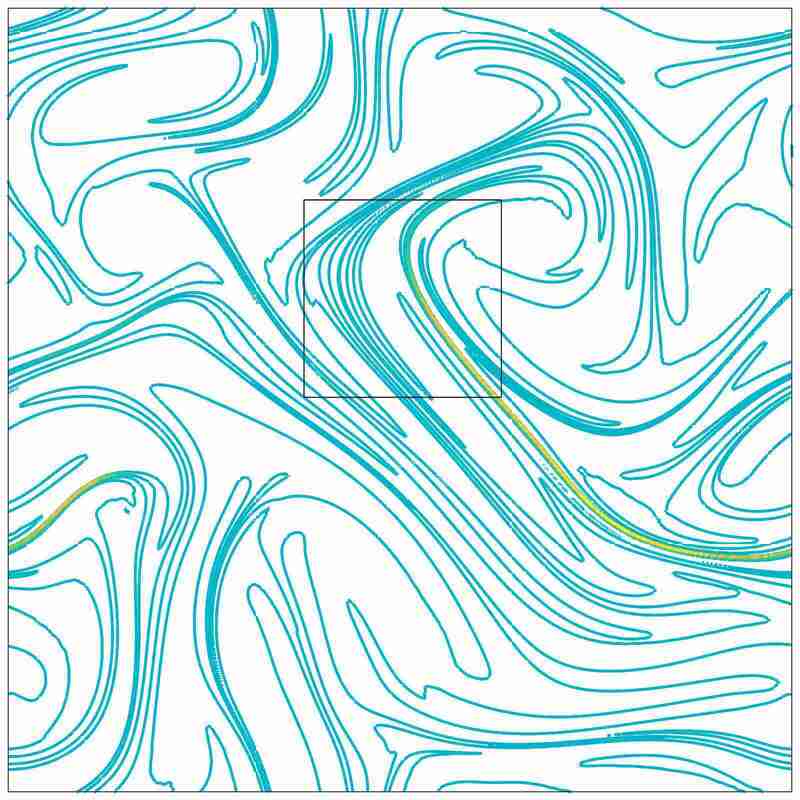}
\caption{$[0,1] \times [0, 1]$}
\label{subfig:LWrz1t2}
\end{subfigure}
\begin{subfigure}{0.21\linewidth}
\includegraphics[width = \linewidth]{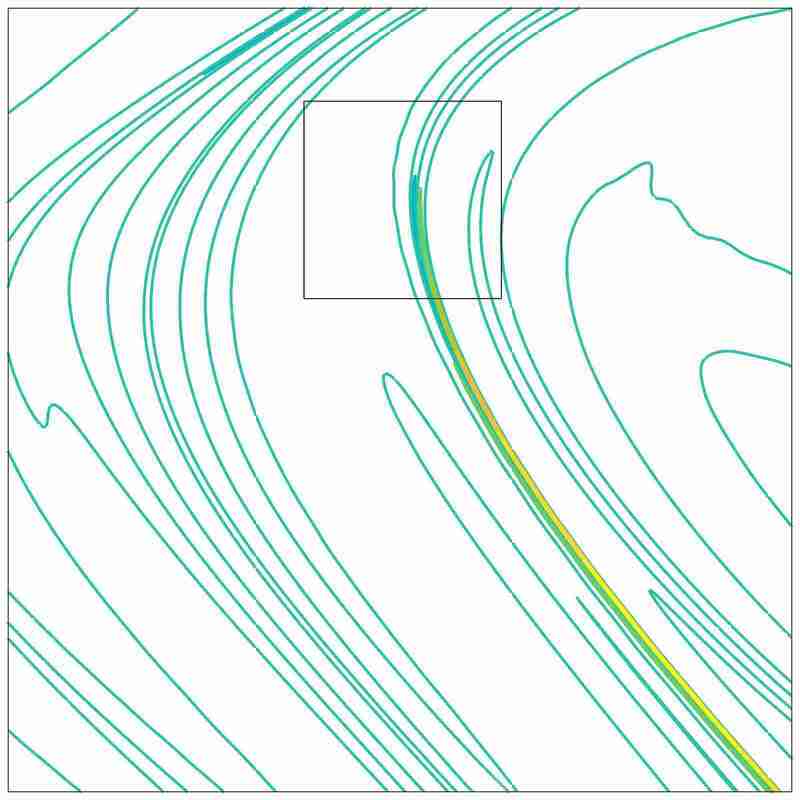}
\caption{$[\frac{3}{8},\frac{5}{8}] \times [\frac{1}{2}, \frac{3}{4}]$}
\label{subfig:LWrz2t2}
\end{subfigure}
\begin{subfigure}{0.21\linewidth}
\includegraphics[width = \linewidth]{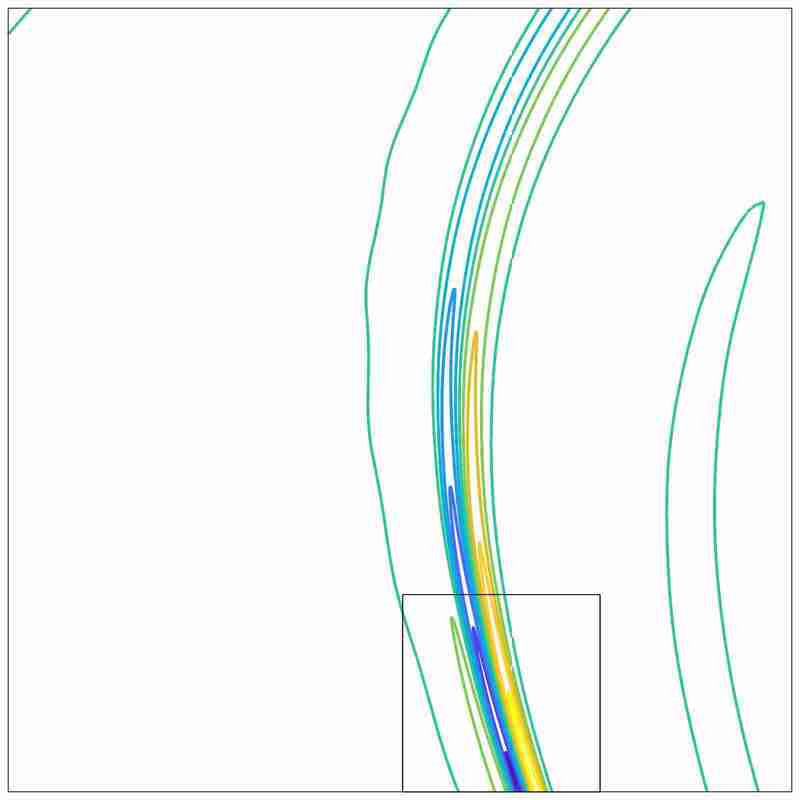}
\caption{$[\frac{15}{32}, \frac{17}{32}] \times [\frac{21}{32}, \frac{23}{32}]$}
\label{subfig:LWrz3t2}
\end{subfigure}
\begin{subfigure}{0.21\linewidth}
 \includegraphics[width = \linewidth]{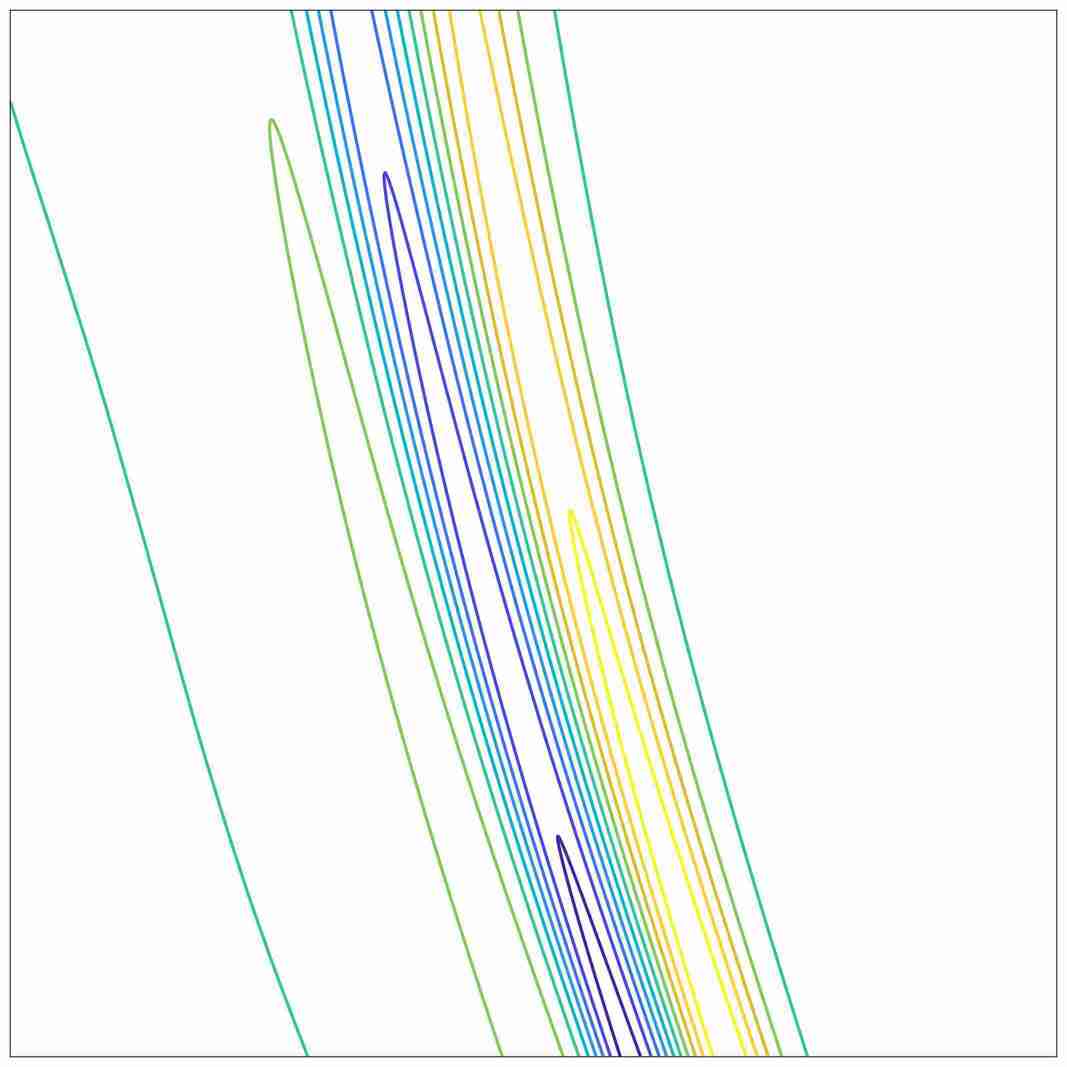}
\caption{$[\frac{1}{2}, \frac{33}{64}] \times [\frac{21}{32}, \frac{43}{64}]$}
\label{subfig:LWrz4t2}
\end{subfigure}
\caption{Gradual $64\times$ zoom on the Laplacian of the vorticity at $t=2$.}
\label{fig:LWrzoom_t2}
\end{figure}

As in the 4-modes test, we also observe slow growth in conservation errors for the random initial data in table \ref{tab:consError_rand}.
\begin{table}[H] \footnotesize
\begin{center}
{\renewcommand{\arraystretch}{1.5}
\begin{tabular}{c | c c c c}
\hline
\hline
$t$ & 0.5 & 1 & 1.5 & 2 \\
\hline
Enstrophy error & $-2.40 \cdot 10^{-6}$ & $1.36 \cdot 10^{-6}$ & $5.95 \cdot 10^{-6}$ & $9.36 \cdot 10^{-6}$  \\
Energy error & $-7.74 \cdot 10^{-7}$ & $-4.71 \cdot 10^{-6}$ & $-3.79 \cdot 10^{-5}$ & $-9.36 \cdot 10^{-5}$  \\
\hline
\end{tabular}} \\
\end{center}
\caption{Conservation errors for random initial condition test using the CM method.}
\label{tab:consError_rand}
\end{table}
In comparison, the enstrophy error for the $8^{\text{th}}$ order, $2048^2$ harmonics Cauchy-Lagrangian method grows from $10^{-14}$ to $10^{-13}$ to $10^{-8}$ for times $0.2$, $0.6$ and $1$.

\subsection{Spatial resolution}

The vorticity solutions in both tests in this section are observed to have increasingly finer spatial features as time progresses. One way to quantify the evolution of the spatial scales is through the Fourier expansion of the solution, in particular, we look at the decay of the magnitudes of the high frequency coefficients. This can be seen from the vorticity spectrum obtained by integrating the square of vorticity over circular shells in Fourier space. That is, let $\hat{\omega}_{\bm{k}}$ be the $\bm{k} = (k_1, k_2)$ coefficient of the Fourier transform of $\omega$, we have
\begin{gather} \label{eq:defVortSpectrum}
E_\omega (K) := \frac{1}{2} \sum_{K \leq |\bm{k}| < K+1} | \hat{\omega}_{\bm{k}} |^2,
\end{gather}
where $|\bm{k}| = \sqrt{k_1^2 + k_2^2}$.

We compare the vorticity spectrum from the CM method to that of Cauchy-Lagrangian (CL8) method presented in \cite{frisch}. Figure \ref{fig:compareSpectrum} shows the overlay of the vorticity spectra obtained from both methods. The vorticity fields are sampled at times $3.45$ and $3.95$ (due to the nature of the time step in the Cauchy-Lagrangian method, the sampling times presented in \cite{frisch} did not land exactly on $t=3.5$ and $4$).

\begin{figure}[h]
\centering
\begin{subfigure}{0.45\linewidth}
\centering
\includegraphics[width = \linewidth]{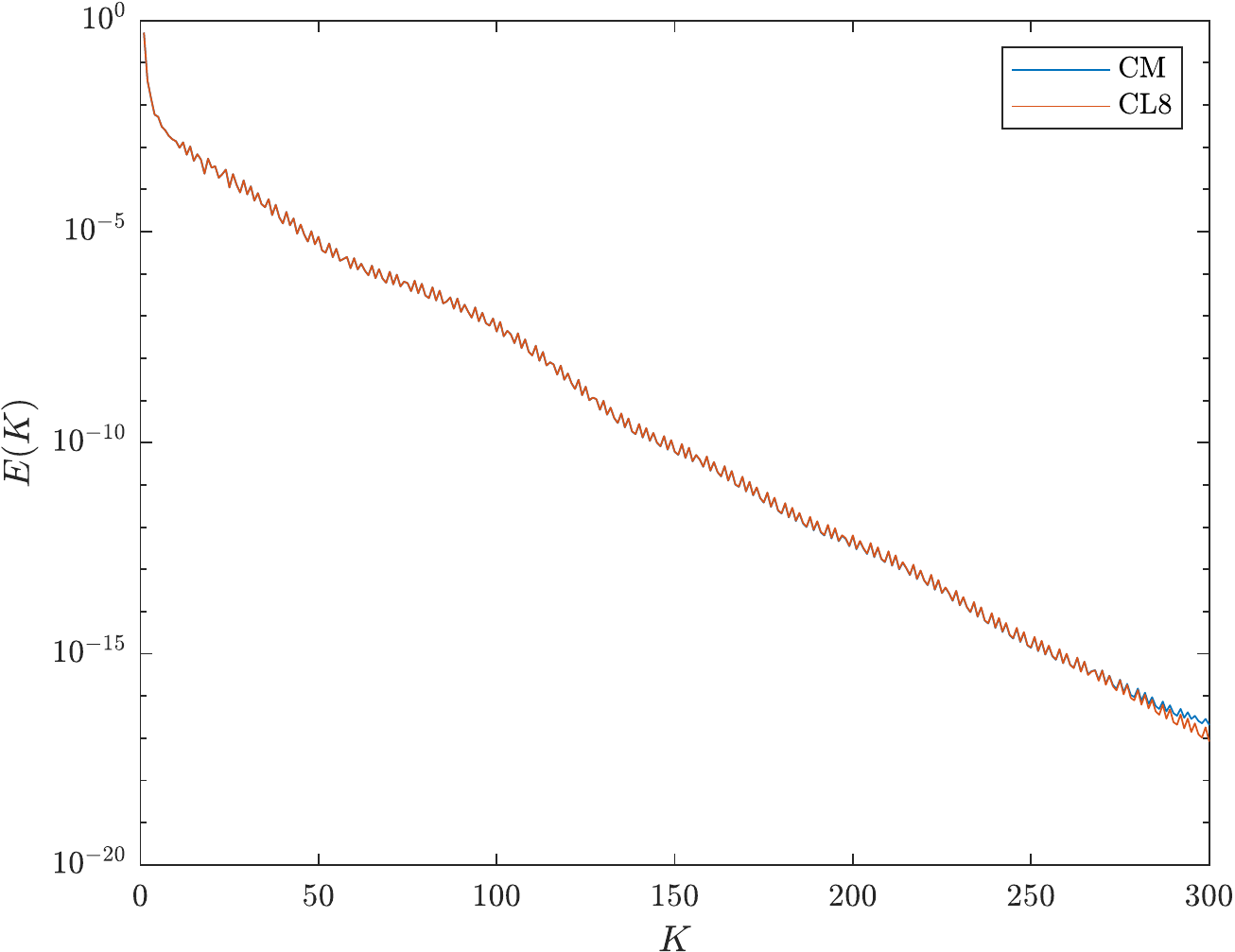}
\caption{$t \approx 3.5$}
\label{subfig:sp35}
\end{subfigure}
\begin{subfigure}{0.45\linewidth}
\centering
\includegraphics[width = \linewidth]{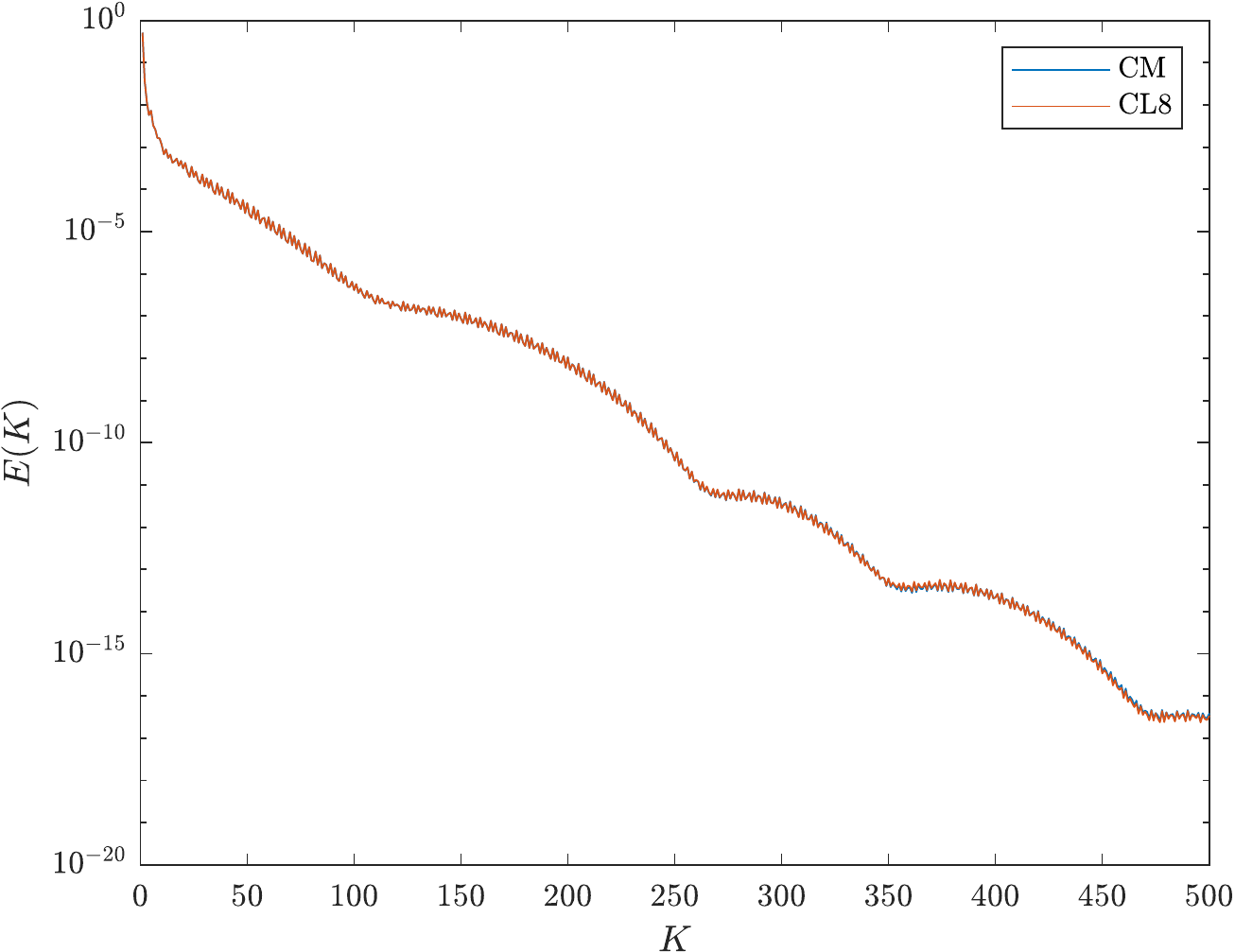}
\caption{$t \approx 4$}
\label{subfig:sp4}
\end{subfigure}
\caption[Decay in the vorticity spectrum at times 3.5 and 4.]{Decay in the vorticity spectrum at times 3.5 and 4\footnotemark.}
\label{fig:compareSpectrum}
\end{figure}

\footnotetext{The Cauchy-Lagrangian method employs variable length time steps. The vorticity spectrum presented in \cite{frisch} are computed at the last time step before reaching times 3.5 and 4. The final times turned out to be approximately 3.45 and 3.95. In this figure, the final times for the CM method are taken to be exactly 3.45 and 3.95.}

\begin{figure}[h]
\centering
\begin{subfigure}{0.45\linewidth}
\centering
\includegraphics[width = \linewidth]{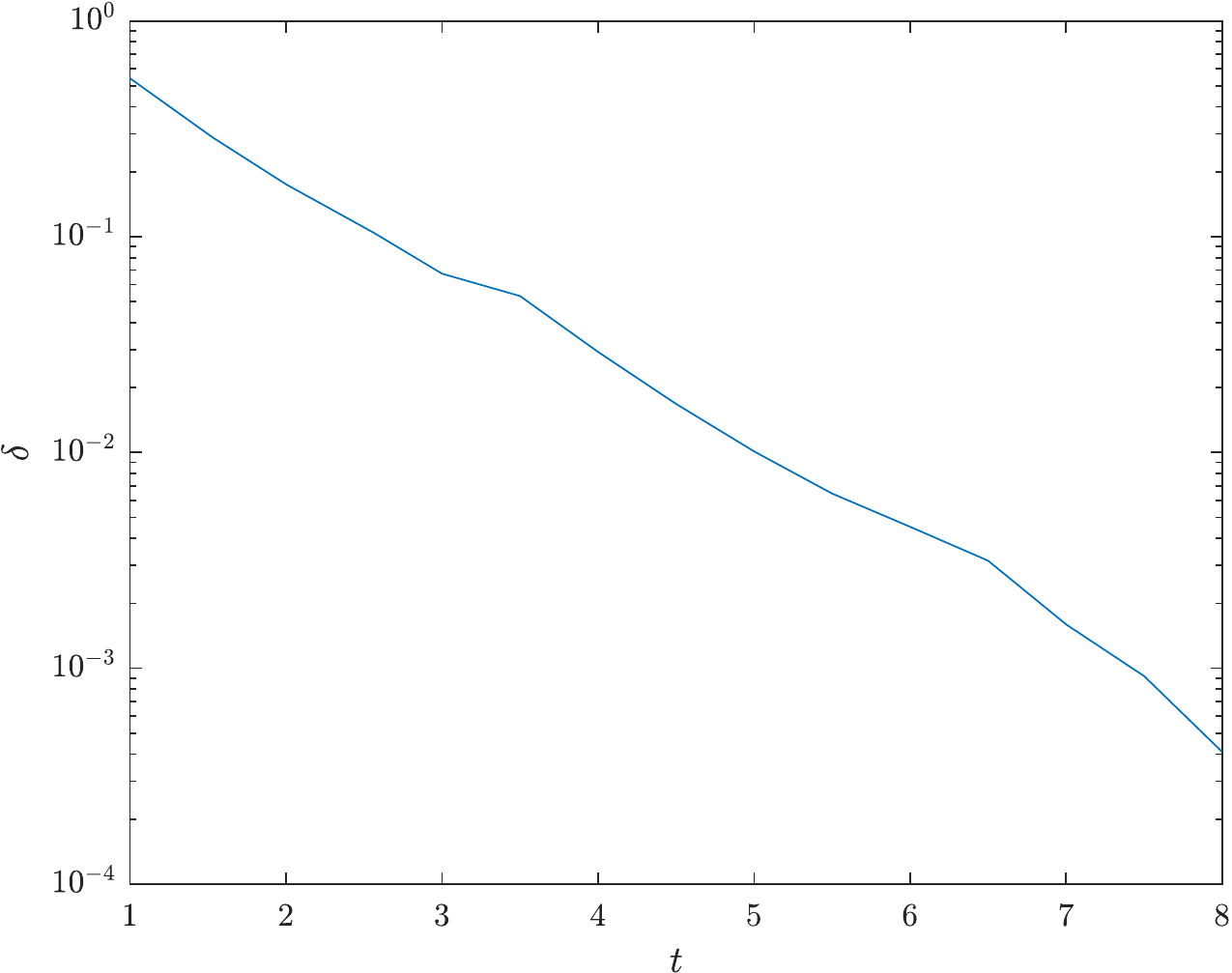}
\caption{4-modes}
\label{subfig:rad8}
\end{subfigure}
\hspace{10pt}
\begin{subfigure}{0.45\linewidth}
\centering
\includegraphics[width = \linewidth]{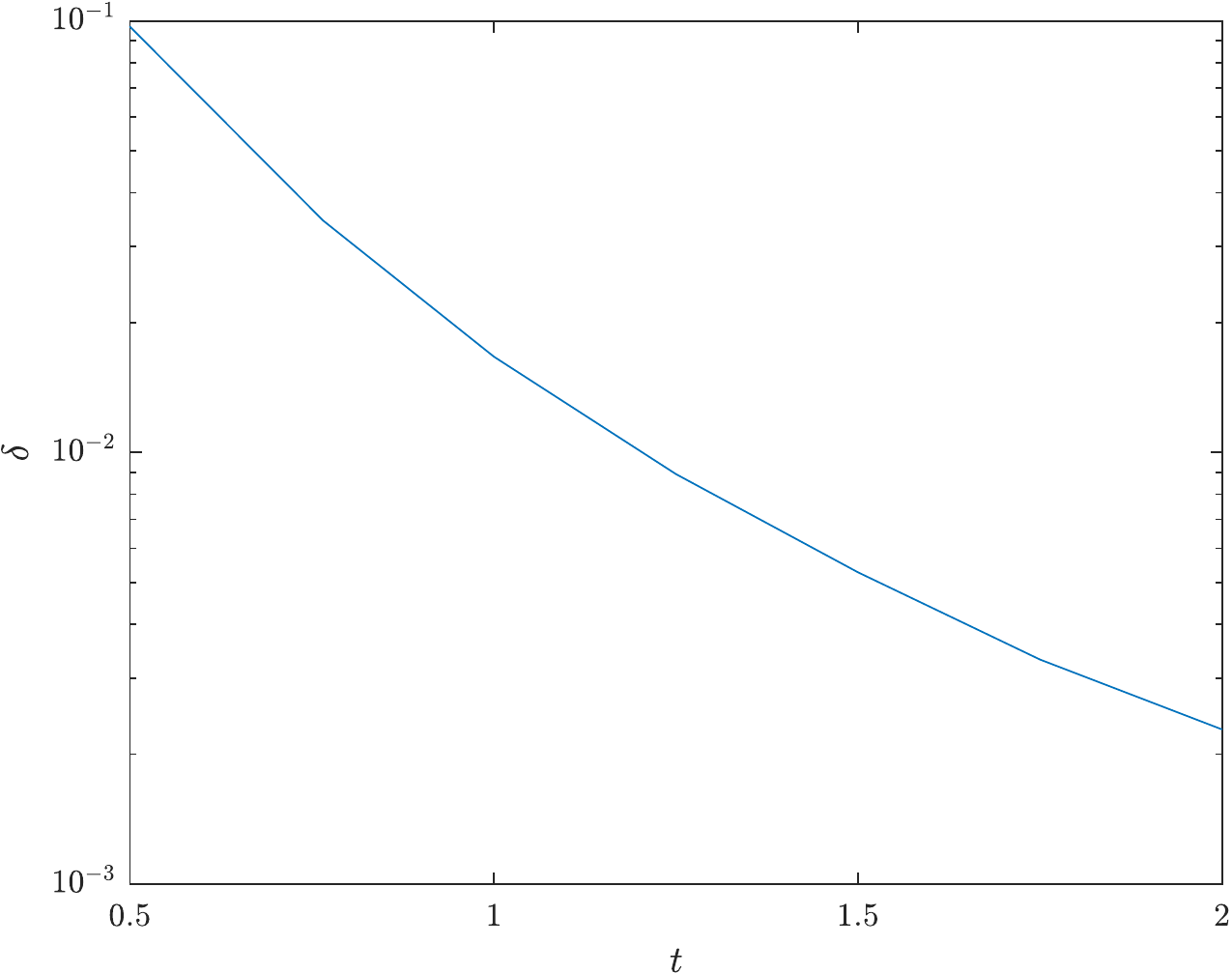}
\caption{Random initial condition}
\label{subfig:radrand}
\end{subfigure}
\caption{Radius of analyticity $\delta(t)$ vs time.}
\label{fig:analyticRadius}
\end{figure}

We see that in figure \ref{fig:compareSpectrum}, the vorticity spectra obtained from the CM method matches almost exactly, up to double precision, the high-fidelity results obtained from the Cauchy-Lagrangian $8096^2$ simulation (The curves are on top of each other. The plots are provided in vector-graphics format, zooming in the tail of the spectrum shows some discrepancies between the two curves. This is in large part due to the two simulations not having exactly the same final times). This suggests that small time deformations can be accurately represented on a coarse $128^2$ grid and the fine scale global time deformations can be reconstructed by the composition of the submaps without loss of resolution.

One measurement of the asymptotic decay of the Fourier coefficients is the radius of analyticity. The decay rate of the vorticity spectrum at high frequency modes indicates the spatial scales present in the solution. Asymptotically, the decay of the vorticity spectrum is typically
\begin{gather} \label{eq:vortSpectDecay}
E_\omega (K) \sim K^\alpha e^{-2\delta K} .
\end{gather}
The rate of the exponential, $\delta$, is the radius of analyticity and governs the spatial truncation error. For a grid which resolves a maximum  frequency of $k_{\max}$, the spatial truncation error scales like $e^{-\delta k_{\max}}$. 

Figure \ref{fig:analyticRadius} shows the evolution of the radius of analyticity in time for both numerical tests. The radius is estimated at various times by taking a least-squared fit of the logarithm of the tail of the vorticity spectrum $\log (E_\omega(K))$ with respect to the quantities $log(K)$, $K$ and $1$; we extract $\delta$ from the fitted coefficient for $K$. We see that the reduction in the radius of analyticity is exponential. This implies that in order to maintain a certain level of spatial truncation, the maximum resolved frequency $k_{\max}$ must grow exponentially. In particular, for the 4-modes test, at time 8, a grid size of order $10^4$ would be needed to properly resolve the solution. Carrying out computations with traditional methods on such grids would be difficult on a personal-use computer. We see that CM method allows us to evolve the solution for long times without having to use such large grids: through the submap decomposition, only local time coarse grids computations are required, the fine scale details can be recovered by the composition of the submaps. The CM method in fact dynamically adapts to the spatial resolution necessary to the problem. Through the remapping process, the available numerical resolutions autonomously grows as the spatial features in the solution increase.


\subsection{Illustration of the Arbitrary Subgrid Resolution} \label{subsec:subgrid}
Finally, in this section we provide an illustration of the power semigroup decomposition approach in achieving high subgrid resolution of the solution. For this, we simulate a 2 vortex merger problem. We use two identical Gaussian blobs of variance 0.07 placed 0.3 apart in a periodic domain of width 1 (see figure \ref{fig:vm0}). The two vortices both have clockwise spins and are expected to start spinning around each other and almost merge into a single vortex blob. Due to the lack of viscosity, the vortices do not become a single vortex and will generate instabilities as time goes on.

\begin{figure}[h]
\centering
\includegraphics[width = 0.38\linewidth]{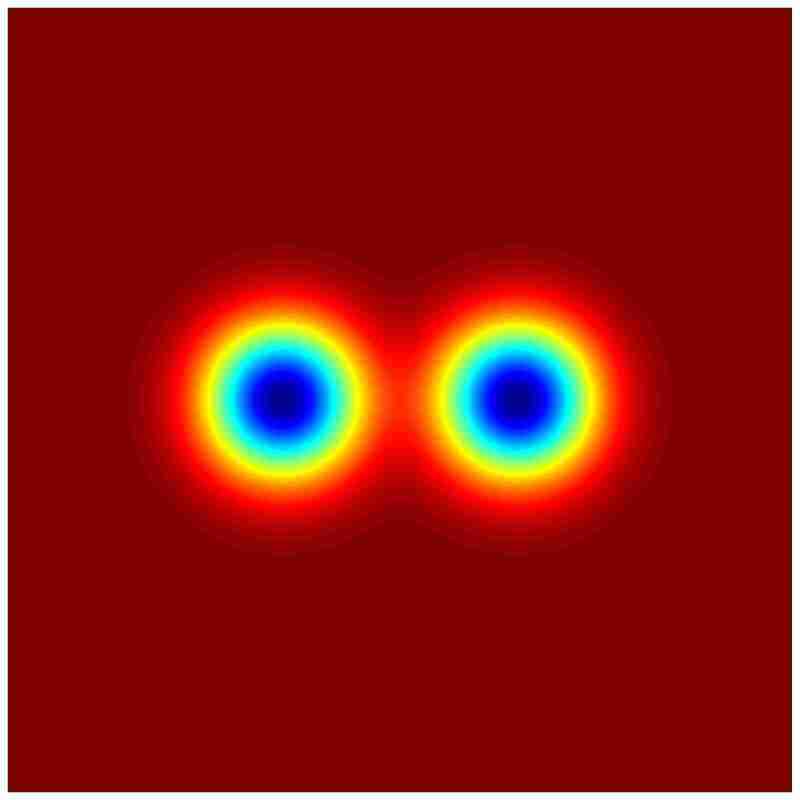}
\caption{Initial vorticity for the vortex merger simulation.}
\label{fig:vm0}
\end{figure}

\begin{figure}[h]
\centering
\includegraphics[width = 0.35\linewidth]{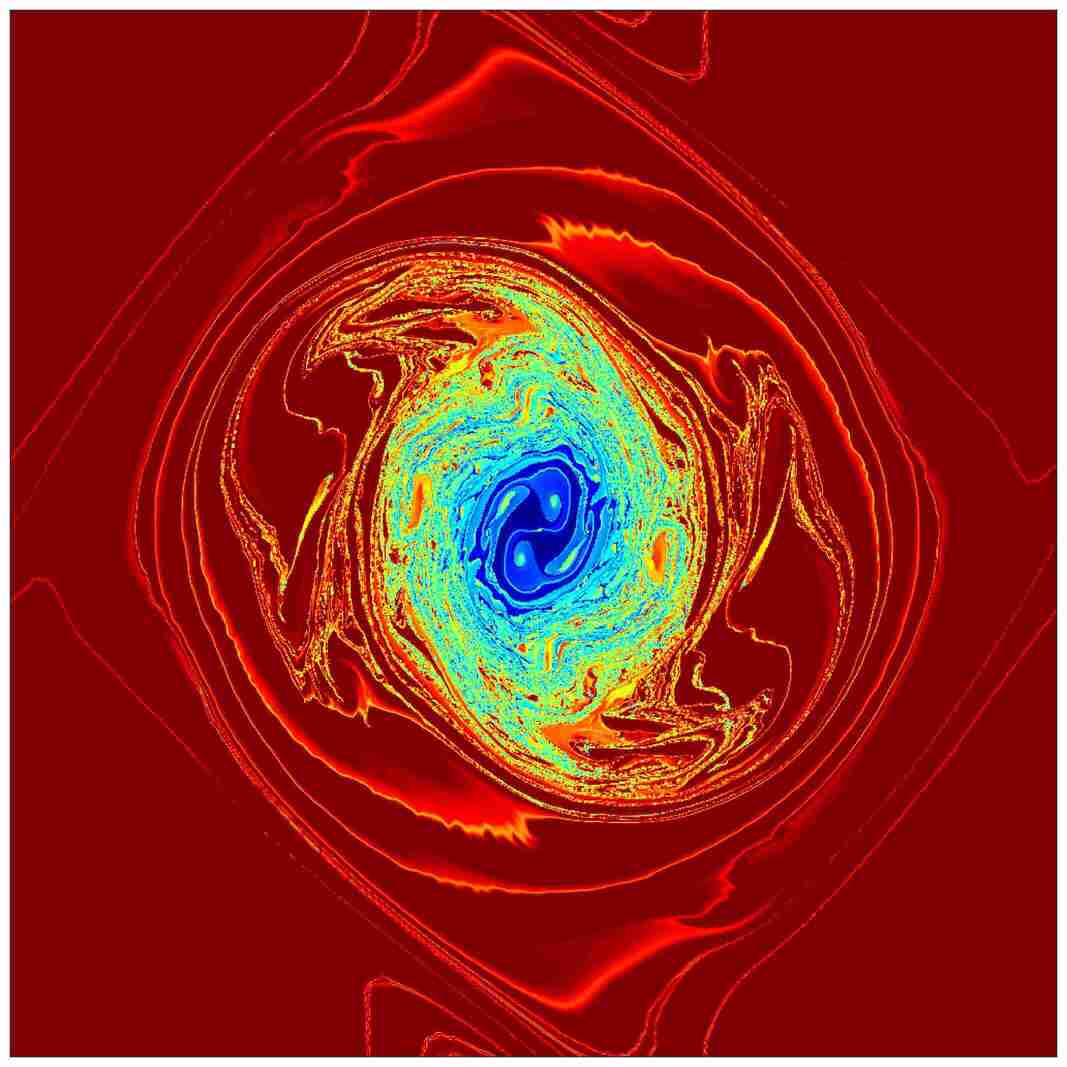}
\caption{Vorticity field at time $t=20$}
\label{fig:vm20}
\end{figure}

\begin{figure}[p!h]
\centering
\begin{subfigure}{0.3\linewidth}
\centering
\includegraphics[width = \linewidth]{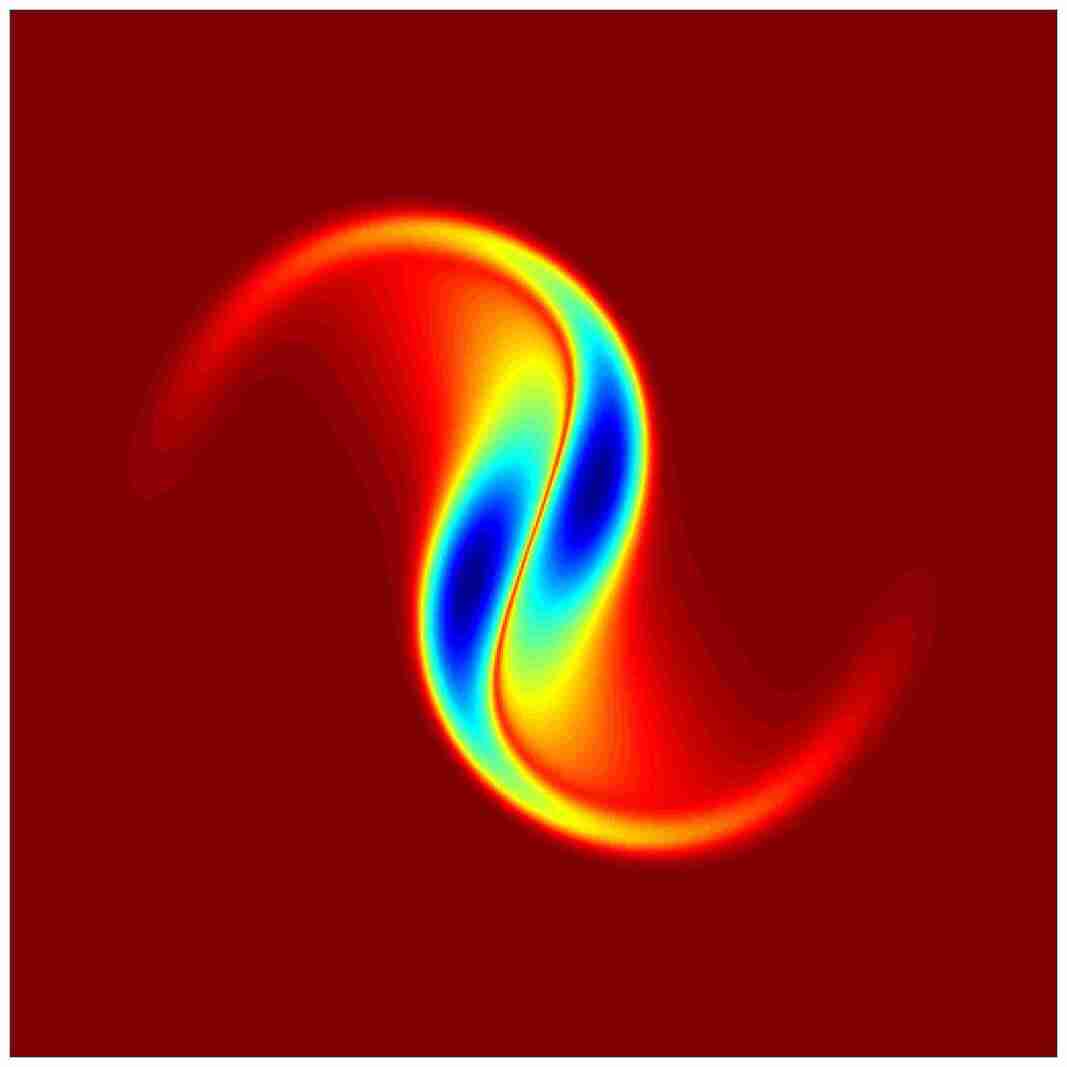}
\caption{$t = 2$}
\label{subfig:vm2}
\end{subfigure}
\begin{subfigure}{0.3\linewidth}
\centering
\includegraphics[width = \linewidth]{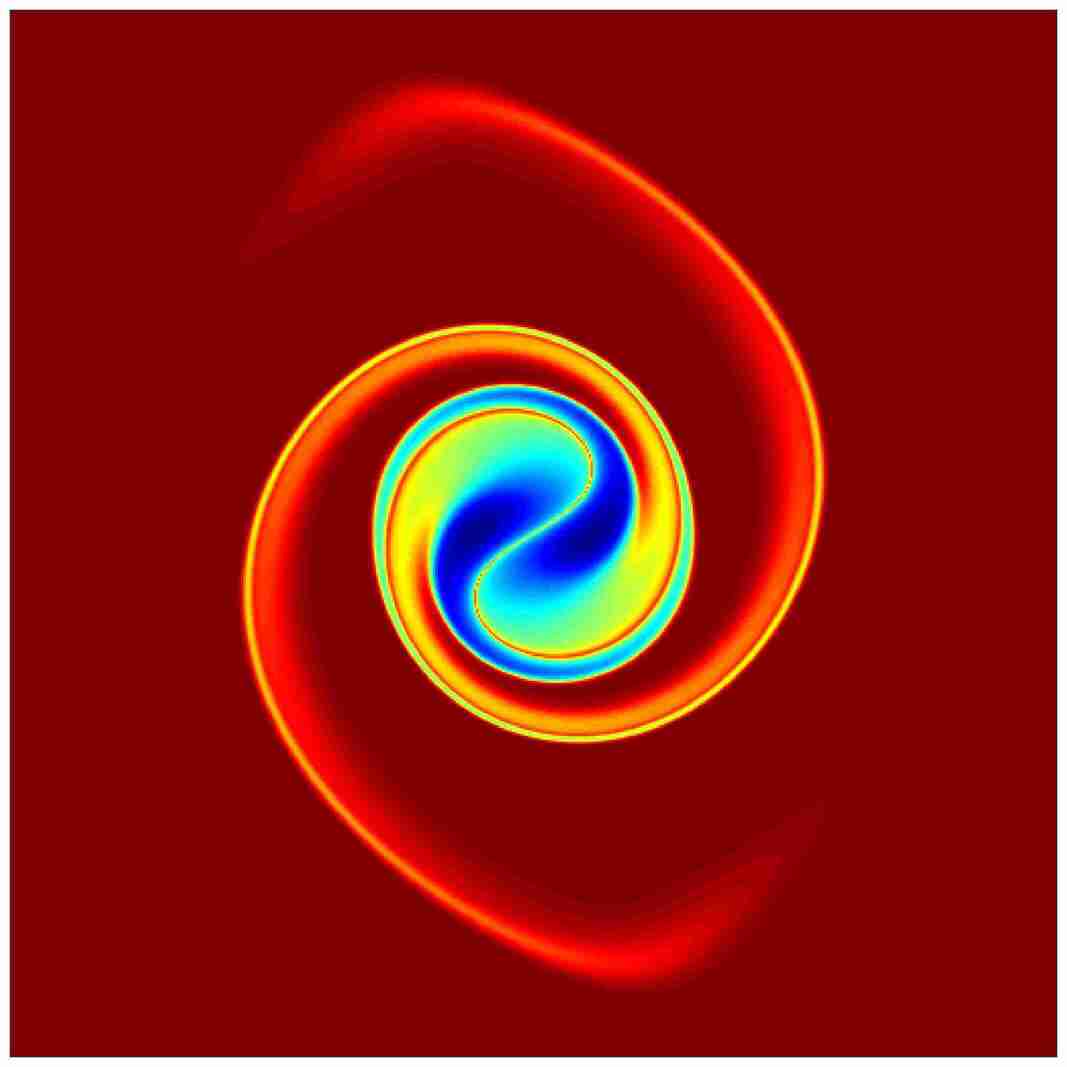}
\caption{$t = 4$}
\label{subfig:vm4}
\end{subfigure}
\begin{subfigure}{0.3\linewidth}
\centering
\includegraphics[width = \linewidth]{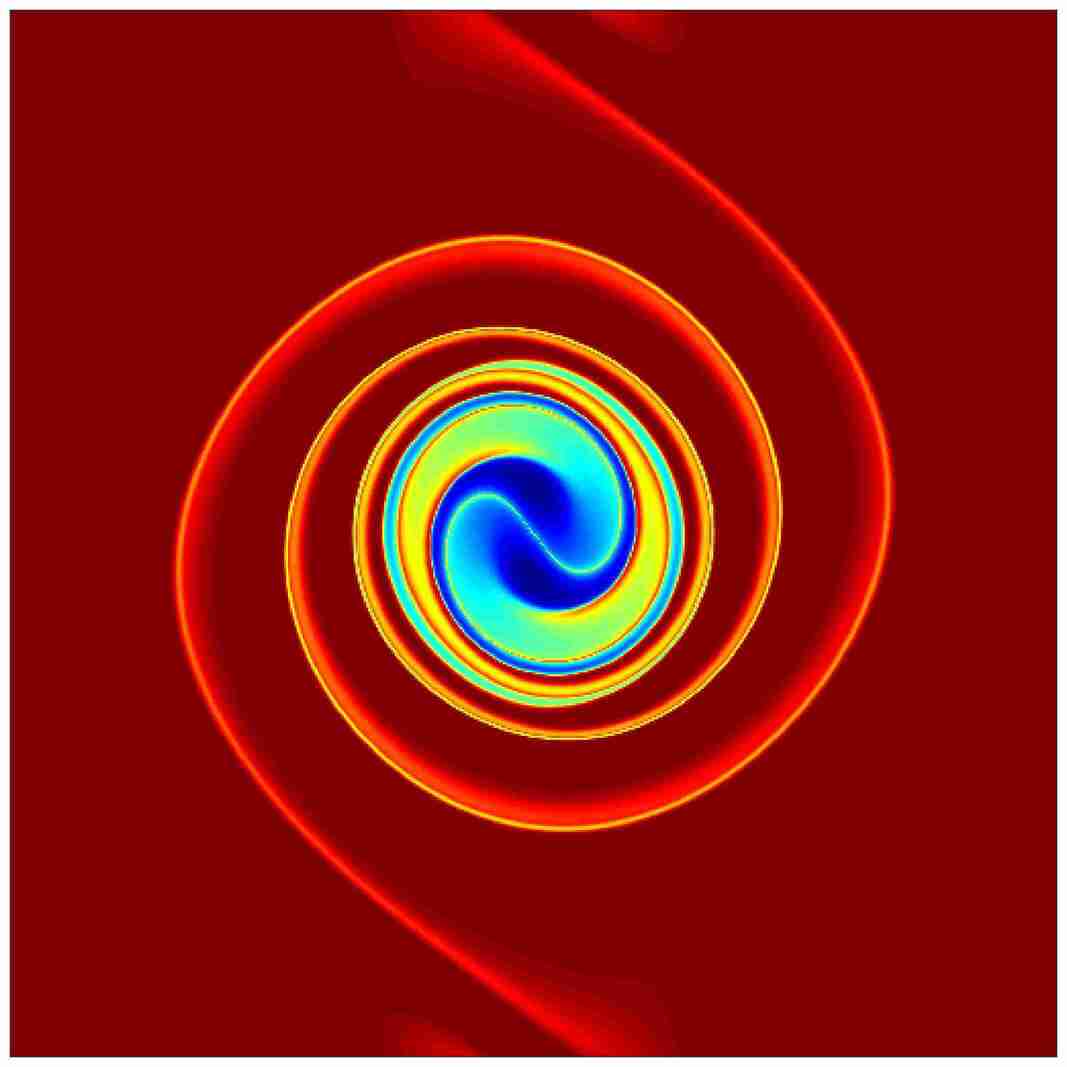}
\caption{$t = 6$}
\label{subfig:vm6}
\end{subfigure}
\begin{subfigure}{0.3\linewidth}
\centering
\includegraphics[width = \linewidth]{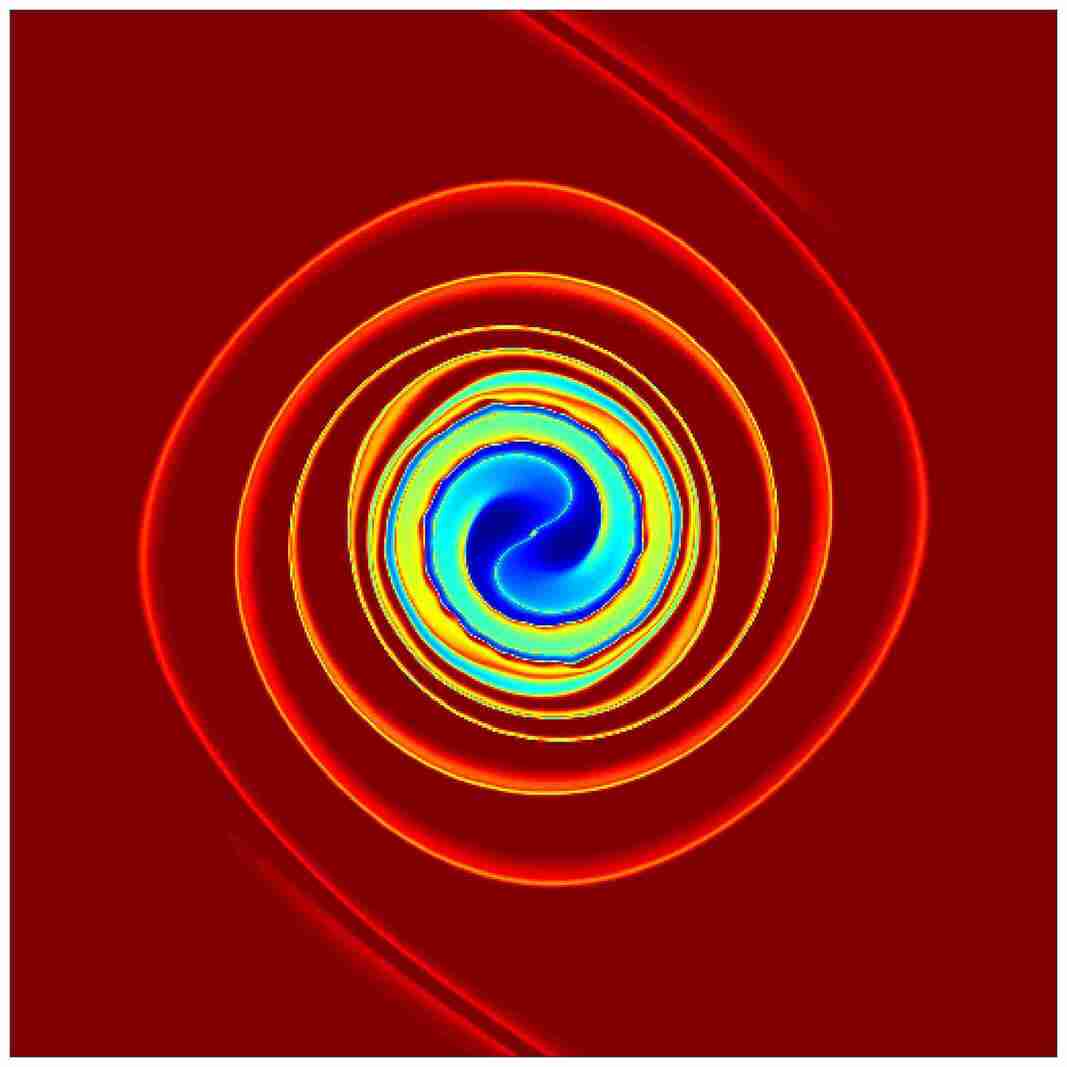}
\caption{$t = 8$}
\label{subfig:vm8}
\end{subfigure}
\begin{subfigure}{0.3\linewidth}
\centering
\includegraphics[width = \linewidth]{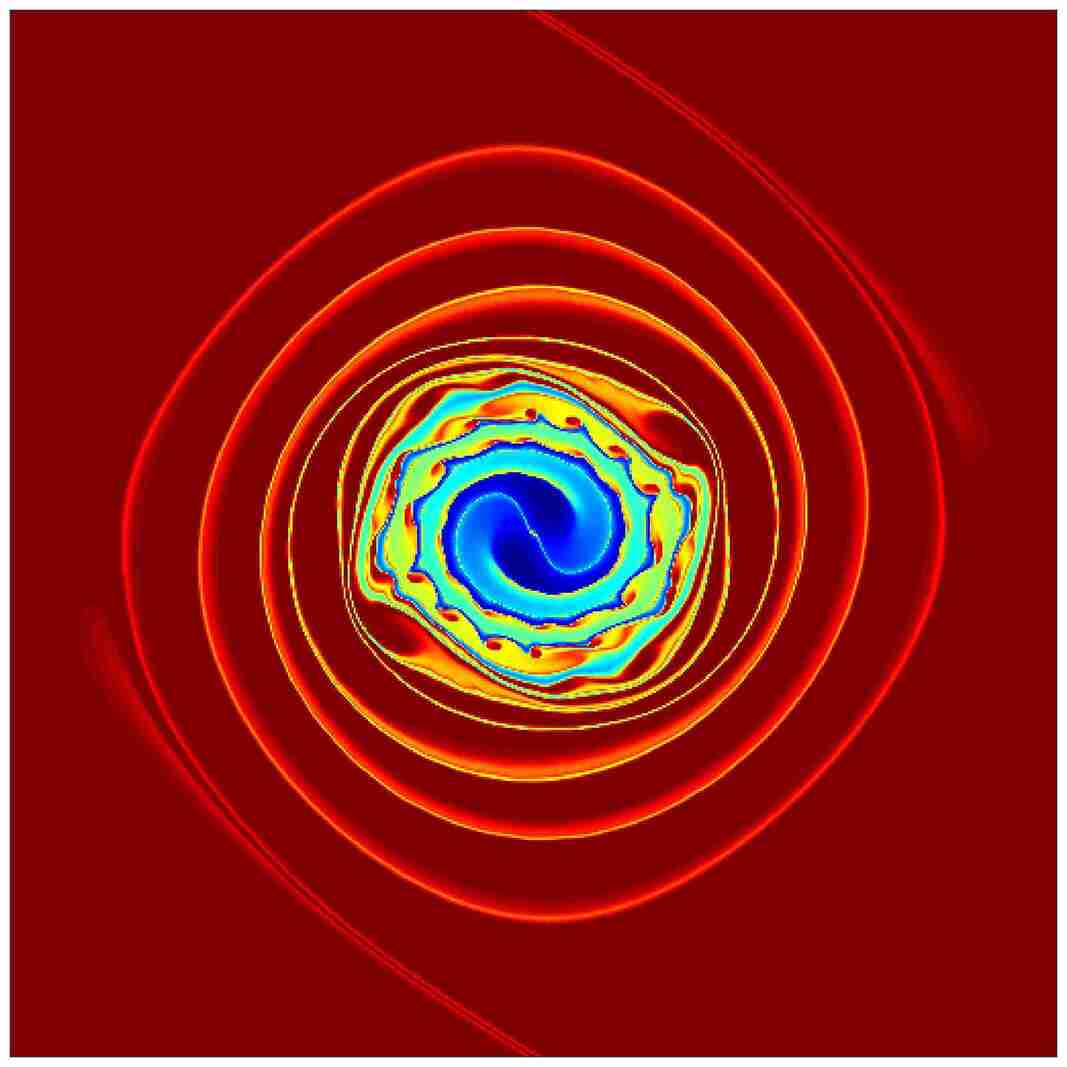}
\caption{$t = 10$}
\label{subfig:vm10}
\end{subfigure}
\begin{subfigure}{0.3\linewidth}
\centering
\includegraphics[width = \linewidth]{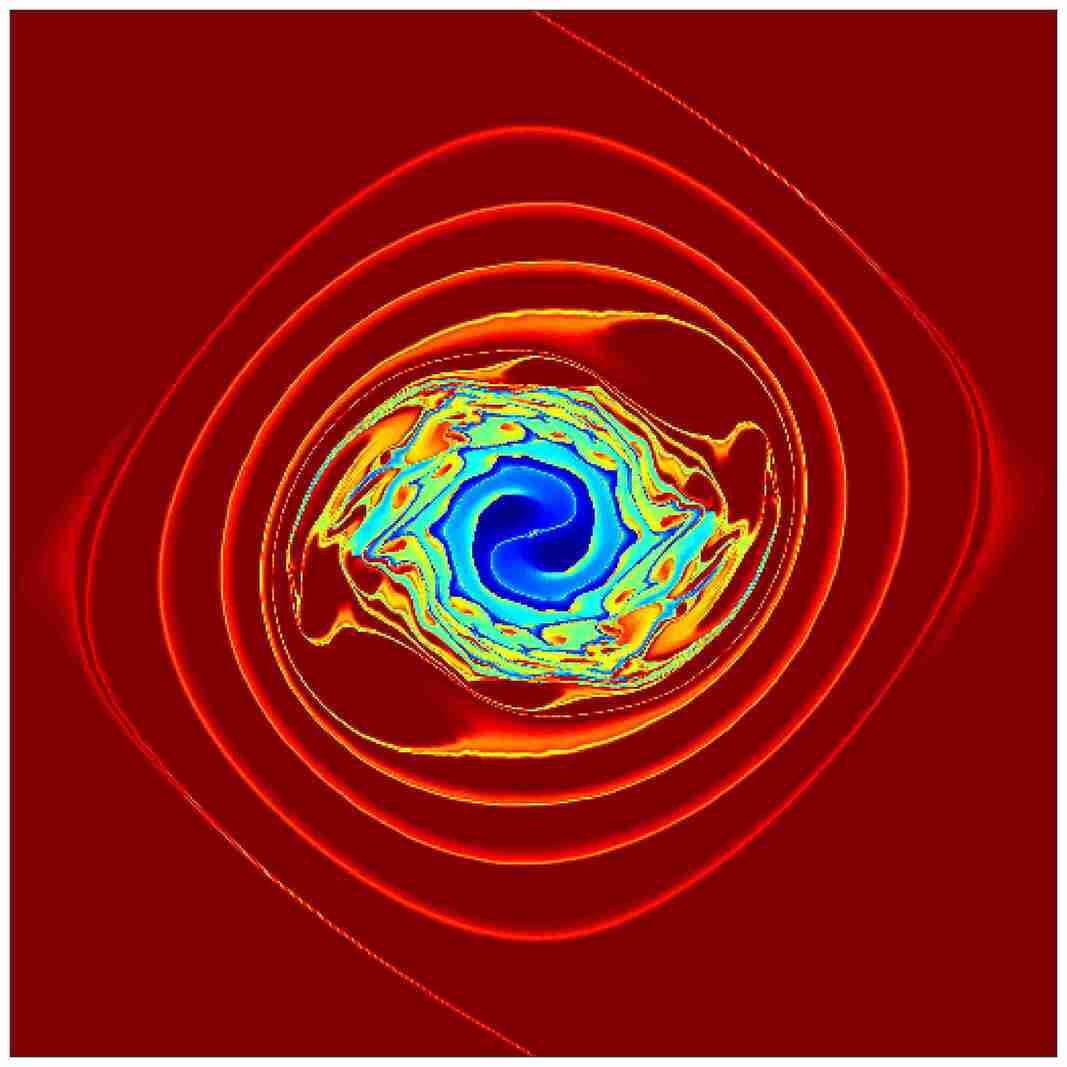}
\caption{$t = 12$}
\label{subfig:vm12}
\end{subfigure}
\begin{subfigure}{0.3\linewidth}
\centering
\includegraphics[width = \linewidth]{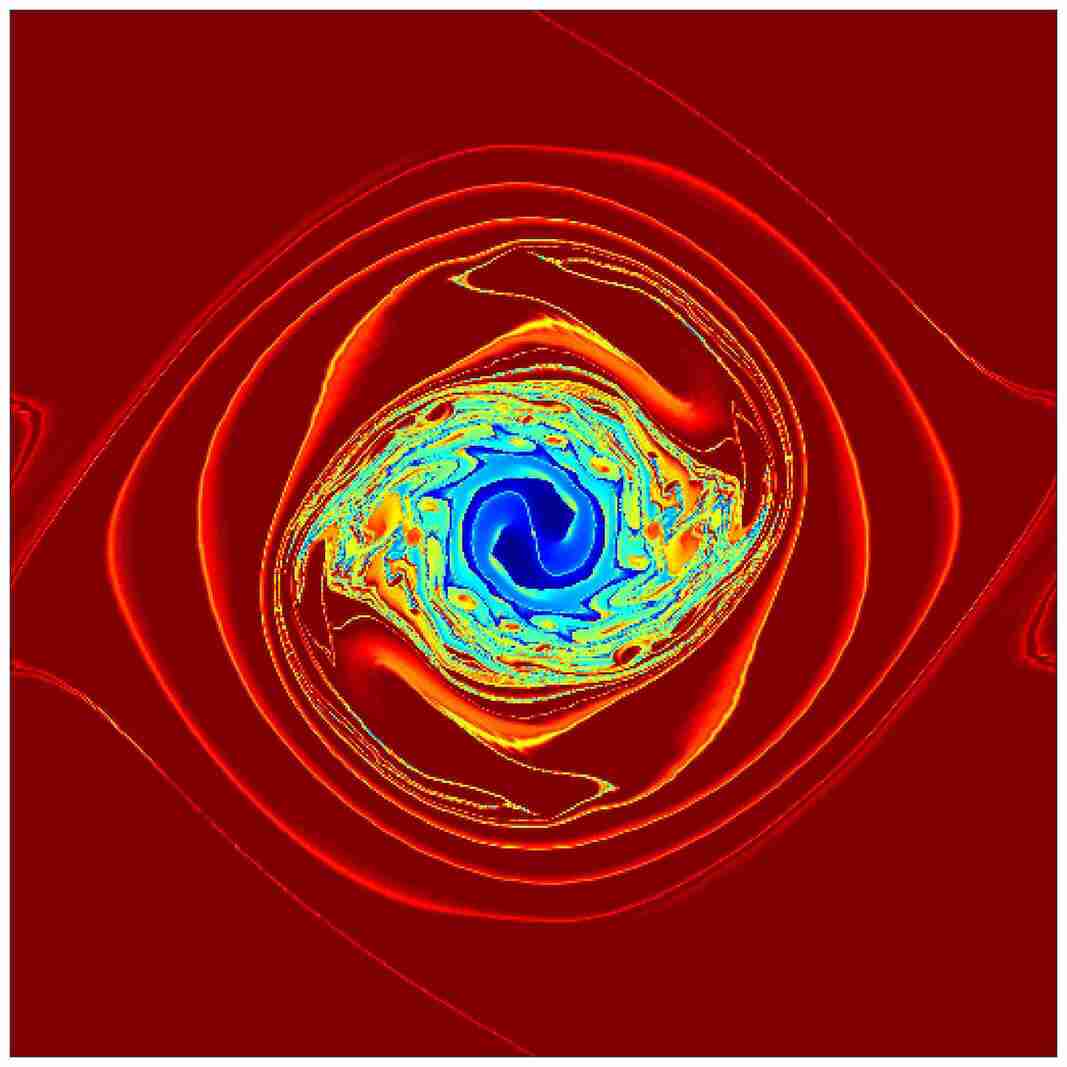}
\caption{$t = 14$}
\label{subfig:vm14}
\end{subfigure}
\begin{subfigure}{0.3\linewidth}
\centering
\includegraphics[width = \linewidth]{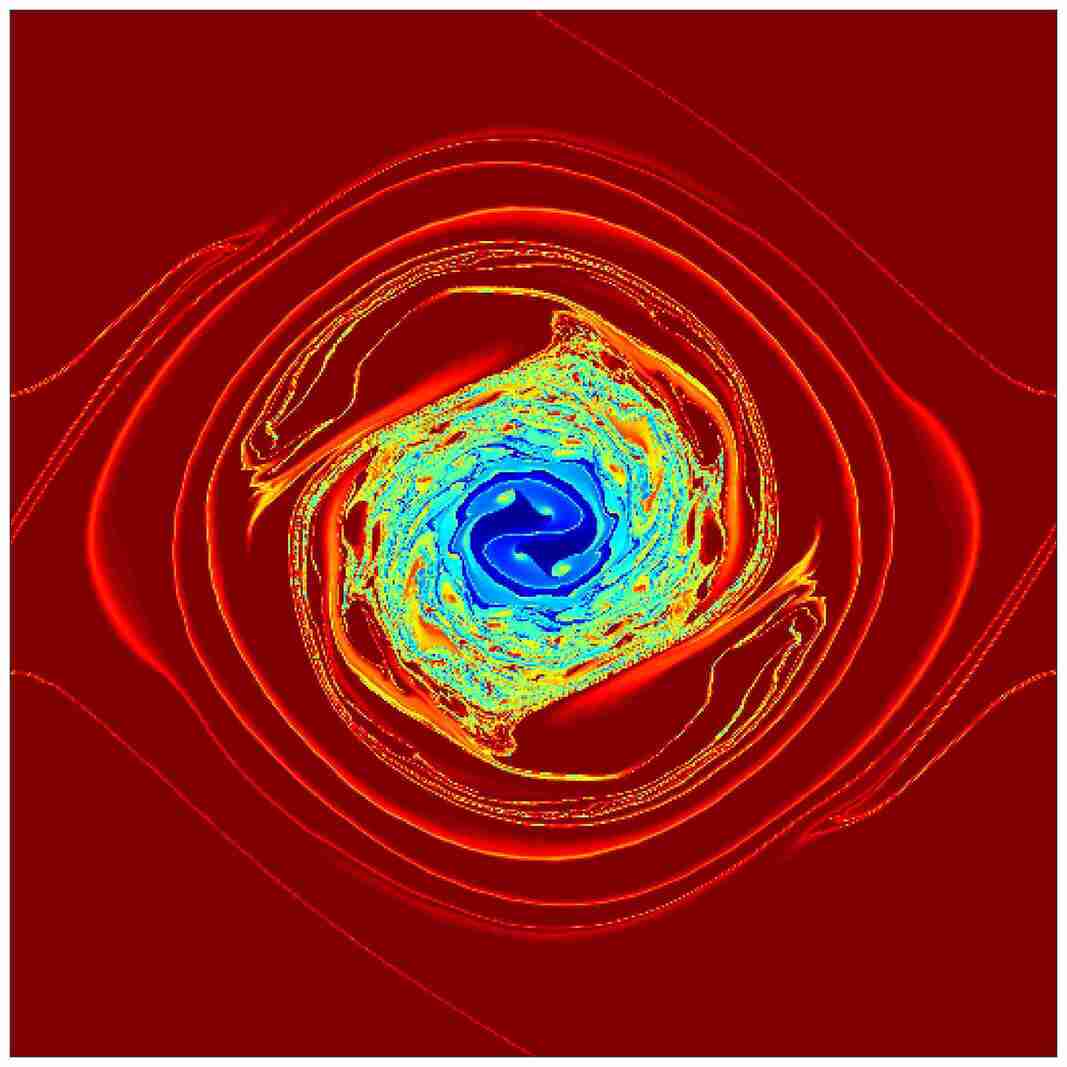}
\caption{$t = 16$}
\label{subfig:vm16}
\end{subfigure}
\begin{subfigure}{0.3\linewidth}
\centering
\includegraphics[width = \linewidth]{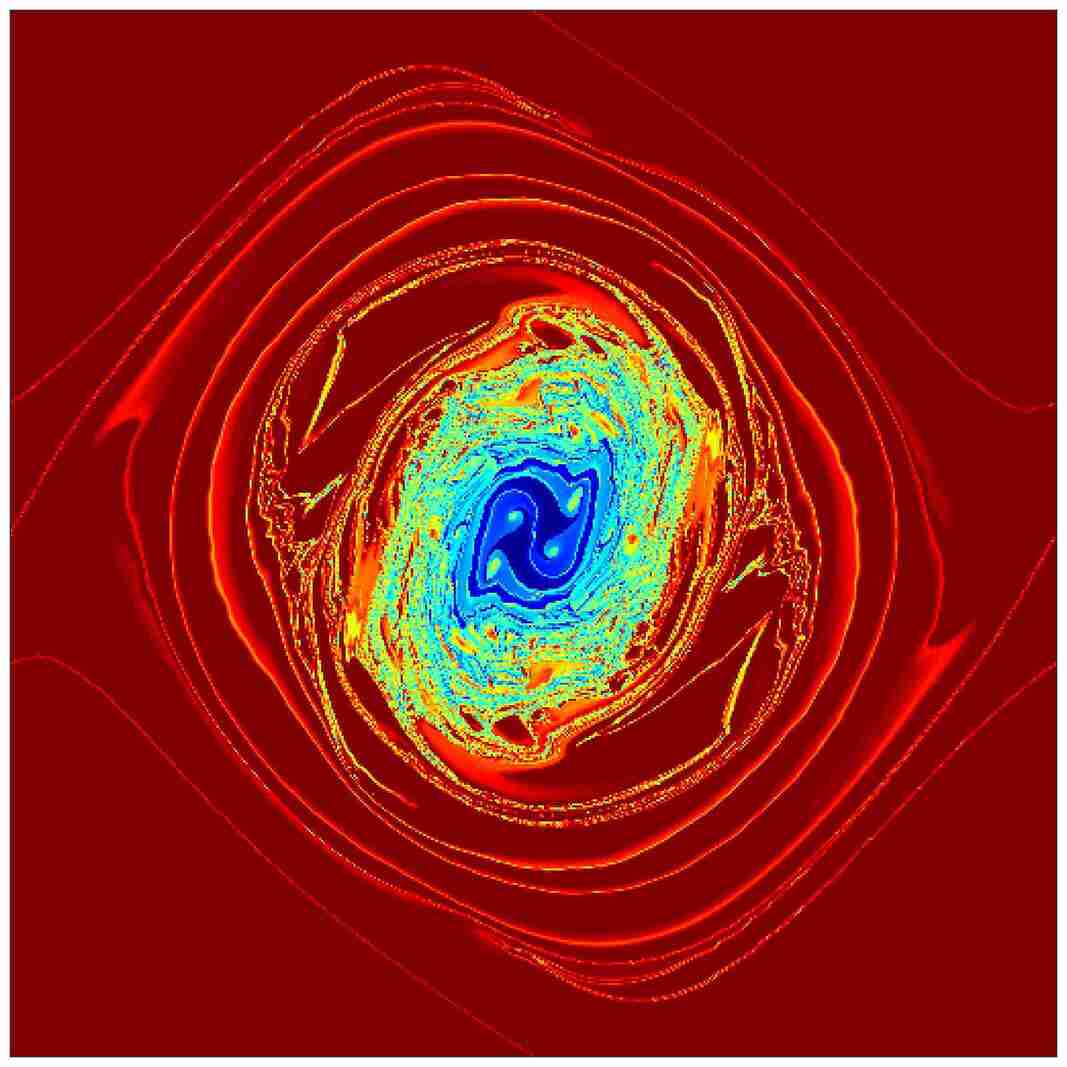}
\caption{$t = 18$}
\label{subfig:vm18}
\end{subfigure}
\caption{Evolution of the vorticity field in the vortex merger simulation.}
\label{fig:vortexMergerTime}
\end{figure}

The fine scale features produced by this flow requires a high amount of spatial resolution to evolve and represent. We performed the simulation using the CM method on a 512 grid for the map and 1/128 time step until $t = 20$. The results are shown in figures \ref{fig:vm20} and \ref{fig:vortexMergerTime}.

The final frame (figure \ref{fig:vm20}) at time 20 is obtained using a composition of 605 submaps. This allows for the representation of a tremendous amount of fine scale subgrid structures. To illustrate this, we take a gradual zoom towards the position $(x,y) = (13/32, 13/32)$ in the last frame. Figure \ref{fig:vortexMergerZoom} shows the zoomed view on the $t=20$ vorticity field. Each subfigure is obtained by evaluating the submap compositions on the subdomain corresponding to the zoomed view. Since the characteristic map has a functional definition, we can use the same number of sample points to generate each picture, therefore obtaining high resolution images of arbitrarily small regions in the domain. For instance, figure \ref{subfig:vmz13} shows the vorticity field in a region of size $1/8192 \times 1/8192$. The image is generated using $768^2$ sample points, providing an accurate depiction of the details seen at the fine scale level.

\begin{figure}[hp]
\centering
\begin{subfigure}{0.25\linewidth}
\centering
\includegraphics[width = \linewidth]{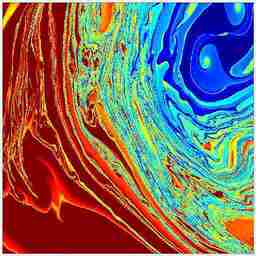}
\caption{width $= 2^{-2}$}
\label{subfig:vmz2}
\end{subfigure}
\begin{subfigure}{0.25\linewidth}
\centering
\includegraphics[width = \linewidth]{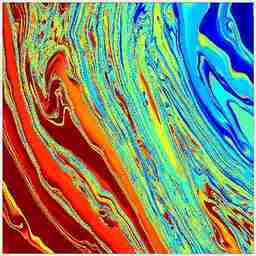}
\caption{width $= 2^{-3}$}
\label{subfig:vmz3}
\end{subfigure}
\begin{subfigure}{0.25\linewidth}
\centering
\includegraphics[width = \linewidth]{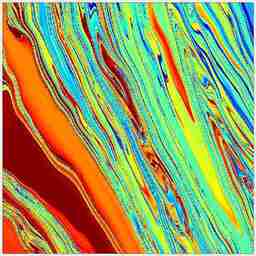}
\caption{width $= 2^{-4}$}
\label{subfig:vmz4}
\end{subfigure}
\begin{subfigure}{0.25\linewidth}
\centering
\includegraphics[width = \linewidth]{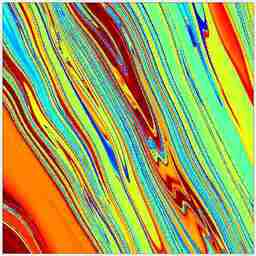}
\caption{width $= 2^{-5}$}
\label{subfig:vmz5}
\end{subfigure}
\begin{subfigure}{0.25\linewidth}
\centering
\includegraphics[width = \linewidth]{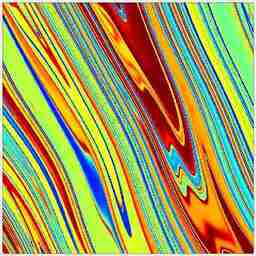}
\caption{width $= 2^{-6}$}
\label{subfig:vmz6}
\end{subfigure}
\begin{subfigure}{0.25\linewidth}
\centering
\includegraphics[width = \linewidth]{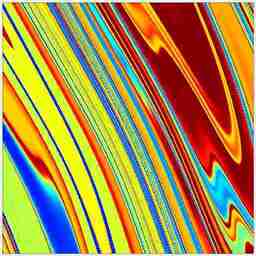}
\caption{width $= 2^{-7}$}
\label{subfig:vmz7}
\end{subfigure}
\begin{subfigure}{0.25\linewidth}
\centering
\includegraphics[width = \linewidth]{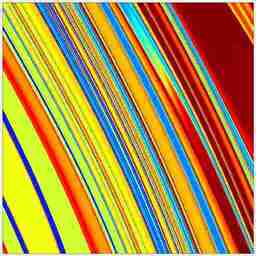}
\caption{width $= 2^{-8}$}
\label{subfig:vmz8}
\end{subfigure}
\begin{subfigure}{0.25\linewidth}
\centering
\includegraphics[width = \linewidth]{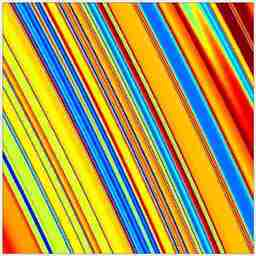}
\caption{width $= 2^{-9}$}
\label{subfig:vmz9}
\end{subfigure}
\begin{subfigure}{0.25\linewidth}
\centering
\includegraphics[width = \linewidth]{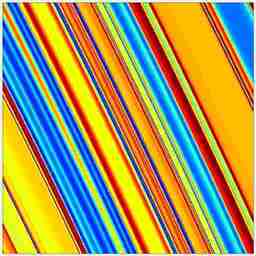}
\caption{width $= 2^{-10}$}
\label{subfig:vmz10}
\end{subfigure}
\begin{subfigure}{0.25\linewidth}
\centering
\includegraphics[width = \linewidth]{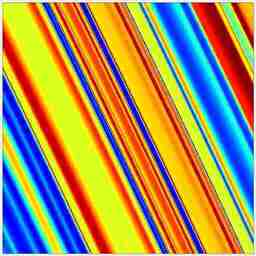}
\caption{width $= 2^{-11}$}
\label{subfig:vmz11}
\end{subfigure}
\begin{subfigure}{0.25\linewidth}
\centering
\includegraphics[width = \linewidth]{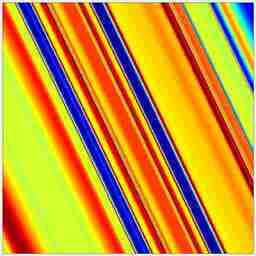}
\caption{width $= 2^{-12}$}
\label{subfig:vmz12}
\end{subfigure}
\begin{subfigure}{0.25\linewidth}
\centering
\includegraphics[width = \linewidth]{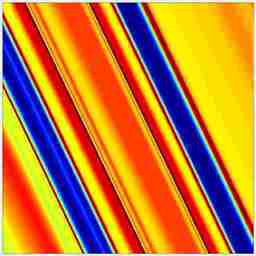}
\caption{width $= 2^{-13}$}
\label{subfig:vmz13}
\end{subfigure}
\caption{Gradual zoom on the last frame, each subfigure is a $2\times$ zoom on the previous.}
\label{fig:vortexMergerZoom}
\end{figure}

The numerical experiments in section \ref{sec:numTests} showcase several advantageous properties of the CM method for the 2D incompressible Euler equations. Firstly, the submap decomposition using the semigroup property of the characteristic map allows for quick and accurate computations on a coarse grid, circumventing the usual requirement of increasing spatial resolution due to exponential vorticity gradient growth. As evidence, solutions from a $128^2$ grid CM solver achieves the same vorticity spectrum as an $8192^2$ grid direct vorticity solver (see figure \ref{fig:compareSpectrum}). Furthermore, due to the volume preserving property of the characteristic map, the CM method achieves high accuracy enstrophy conservation for all times. Whereas other methods experience a spike in enstrophy error when the vorticity fields becomes complicated, the enstrophy conservation in CM is independent of the current time vorticity and is a direct result of volume preservation. This allows the enstrophy error to grow only linearly in time, regardless of the complexity of the vorticity field. Lastly, the functional definition of the numerical vorticity through composition with the backward map allows for an arbitrary spatial resolution of the solution. Furthermore, the submap decomposition generates the correct scales of gradients, ensuring that the increasing fine scale features in the vorticity solutions are properly represented as the resolution increases. This property is evidenced in figures \ref{fig:Wzoom_t4}, \ref{fig:LWzoom_t4}, \ref{fig:Wzoom_t8}, \ref{fig:LWzoom_t8} and more thoroughly in \ref{fig:vortexMergerZoom} in section \ref{subsec:subgrid}, where we zoom in on the solution to show the arbitrary spatial resolution achieved by the CM method.

\section{Conclusion}

In this paper, we have presented a novel numerical method for solving the 2D incompressible Euler equations. The work in this paper is an extension of the CM method for linear advection and is based on the GALS and Jet-Scheme frameworks. This method is unique in that it solves for the deformation map generated by the fluid flow and captures the geometry of the problem; all evolved quantities of interest can be obtained from this transformation. As a result, this scheme is characterized by the arbitrarily fine subgrid resolution it provides on the solutions, and a lack of artificial dissipation. Several key observations has lead to the development of this method. Firstly, the arbitrarily fine scales typically generated by an inviscid flow lead to the functional representation of the vorticity field though the pullback by the characteristic map, this approach not only preserves fine scales but more importantly avoids spatial truncations of the vorticity field hence eliminating artificial dissipation. These properties are demonstrated in the tests in section \ref{sec:numTests}, in particular in the zoomed view of the solutions. Secondly, the possible exponential growth in the vorticity gradient lead to the use of the semigroup structure of the flow maps to decompose the characteristic map: exponential growth can be generated by a composition of maps of fixed resolution. Lastly, the assumption that the dynamics of the fluid is mainly governed by the large scale low frequency features of the velocity allowed us to carry out the characteristic map computations on a coarse grid, improving the efficiency of the method. We drew a parallel between the use of coarse scale velocity and the Lagrangian-Averaged Euler equations.

The use of the characteristic map for the simulation of inviscid fluids opens many new possibilities for future research. The next step is to include the vorticity-stretching term in order to generalize this 2D method to 3 dimensional problems. There is also a possibility of including forcing terms and potentially extensions to the method which deals with forcing terms while maintaining the characteristic structure. Furthermore, even though the current hyperbolic advective structure of the method is somewhat incompatible with the parabolic diffusion term, there may be modifications to the framework which allows for the integration of a viscosity term which would allow for the use of the CM method on the 3D incompressible Navier-Stokes equations. Lastly, at a longer term, the implementation of boundary conditions and the application of the method to complex geometries is also an important problem to study. These are the current directions of our research as we believe that the CM method provides a unique and appropriate framework for solving more general problems in computational fluid dynamics.

\vspace*{1cm}
\bibliographystyle{siamplain}
\bibliography{CMEULER}





\end{document}